\documentclass[a4paper, 12pt]{amsart}
\usepackage{amsmath}
\usepackage{amssymb, latexsym, amsthm}
\usepackage[all]{xy}

\newtheorem{Thm}{Theorem}[section]
\newtheorem{Prop}[Thm]{Proposition}
\newtheorem{Cor}[Thm]{Corollary}
\newtheorem{Lem}[Thm]{Lemma}

\theoremstyle{definition}
\newtheorem*{definition}{Definition}
\newtheorem*{remark}{Remark}

\numberwithin{equation}{section}

\begin{document}

\newcommand{\Coim}{\mathrm{Coim}}
\newcommand{\Z}[0]{\mathbb{Z}}
\newcommand{\Q}[0]{\mathbb{Q}}
\newcommand{\F}[0]{\mathbb{F}}
\newcommand{\N}[0]{\mathbb{N}}
\renewcommand{\O}[0]{\mathcal{O}}
\newcommand{\p}[0]{\mathfrak{p}}
\newcommand{\m}[0]{\mathrm{m}}
\newcommand{\Tr}{\mathrm{Tr}}
\newcommand{\Hom}[0]{\mathrm{Hom}}
\newcommand{\Gal}[0]{\mathrm{Gal}}
\newcommand{\Res}[0]{\mathrm{Res}}
\newcommand{\id}{\mathrm{id}}
\newcommand{\cl}{\mathrm{cl}}
\newcommand{\mult}{\mathrm{mult}}
\newcommand{\adm}{\mathrm{adm}}
\newcommand{\tr}{\mathrm{tr}}
\newcommand{\pr}{\mathrm{pr}}
\newcommand{\Ker}{\mathrm{Ker}}
\newcommand{\ab}{\mathrm{ab}}
\newcommand{\sep}{\mathrm{sep}}
\newcommand{\triv}{\mathrm{triv}}
\newcommand{\alg}{\mathrm{alg}}
\newcommand{\ur}{\mathrm{ur}}
\newcommand{\Coker}{\mathrm{Coker}}
\newcommand{\Aut}{\mathrm{Aut}}
\newcommand{\Ext}{\mathrm{Ext}}
\newcommand{\Iso}{\mathrm{Iso}}
\newcommand{\M}{\mathcal{M}}
\newcommand{\GL}{\mathrm{GL}}
\newcommand{\Fil}{\mathrm{Fil}}
\newcommand{\an}{\mathrm{an}}
\renewcommand{\c}{\mathcal }
\newcommand{\W}{\mathcal W}
\newcommand{\R}{\mathcal R}
\newcommand{\crys}{\mathrm{crys}}
\newcommand{\st}{\mathrm{st}}
\newcommand{\CM}{\mathrm{CM\Gamma }}
\newcommand{\CV}{\mathcal{C}\mathcal{V}}
\newcommand{\G}{\mathrm{G}}
\newcommand{\Map}{\mathrm{Map}}
\newcommand{\Sym}{\mathrm{Sym}}
\newcommand{\Spec}{\mathrm{Spec}}
\newcommand{\Gr}{\mathrm{Gr}}
\newcommand{\I}{\mathrm{Im}}
\newcommand{\Frac}{\mathrm{Frac}}
\newcommand{\LT}{\mathrm{LT}}
\newcommand{\Alg}{\mathrm{Alg}}
\newcommand{\MG}{\mathrm{M\Gamma }}
\newcommand{\To}{\longrightarrow}
\newcommand{\md}{\mathrm{mod}}
\newcommand{\MF}{\mathrm{MF}}
\newcommand{\CMF}{\mathcal{M}\mathcal{F}}
\newcommand{\Aug}{\mathrm{Aug}}
\renewcommand{\c}{\mathcal }
\newcommand{\uL}{\underline{\mathcal L}}
\newcommand{\Md}{\mathrm{Md}}
\newcommand{\wt}{\widetilde}
\newcommand{\op}{\mathrm}
\newcommand {\w}{\op{wt}}
\newcommand{\Ad}{\op{Ad}}
\newcommand{\ad}{\op{ad}}
\newcommand{\D}{\mathcal D}
\newcommand{\fr}{\mathfrak}

\title[Groups of automorphisms of local fields]
{Groups of automorphisms of local fields of period $p$ and 
nilpotent class $<p$}
\author{Victor Abrashkin}
\address{Department of Mathematical Sciences, 
Durham University, Science Laboratories, 
South Rd, Durham DH1 3LE, United Kingdom \ \&\ Steklov 
Institute, Gubkina str. 8, 119991, Moscow, Russia
}
\email{victor.abrashkin@durham.ac.uk}
\date{November 8, 2016}
\subjclass[2010]{}

\begin{abstract} 
Suppose $K$ is a finite field extension of $\Q _p$ 
containing a primitive $p$-th root of unity. 
Let $\Gamma _{<p}$ be the Galois group of 
a maximal $p$-extension of $K$ with the Galois group of period $p$ 
and nilpotent class $<p$. In the paper we describe 
the ramification filtration 
$\{\Gamma _{<p}^{(v)}\}_{v\geqslant 0}$ and relate it to an explicit 
form of the Demushkin relation for $\Gamma _{<p}$. 
The results are given in terms of 
Lie algebras attached to the appropriate  $p$-groups by the classical equivalence of 
the categories of $p$-groups and Lie algebras of nilpotent class $<p$. 
\end{abstract}
\maketitle

\section*{Introduction} \label{S0} 

Everywhere in the paper $p$ is a prime number, $p>2$. 

If $G$ is a topological group and $s\in\N$ then $C_s(G)$ is the closure of 
the subgroup of commutators of order $\geqslant s$.   
With this notation, 
$G/G^pC_s(G)$ is the maximal quotient of $G$ of period $p$ 
and nilpotent class 
$<s$. Similarly, if $L$ is a topological Lie $\F _p$-algebra 
then $C_s(L)$ is 
the closure of the ideal of commutators of order $\geqslant s$ and 
$L/C_s(L)$ is the maximal quotient 
of nilpotent class $<s$ of $L$. For any topological $\F _p$-module $\c M$ 
we use the notation $L_{\c M}=L\hat\otimes_{\F _p}\c M$. In particular, if $k$ is a finite field extension of 
$\F _p$ and $\sigma $ is the Frobenius automorphism of $k$ then $\id _L\otimes\sigma $ acts 
on $L_k$. For simplicity, we denote $\id _L\otimes\sigma $ just by $\sigma $. Note that 
$L_k|_{\sigma =\id }=L$. 

Suppose $\Q [[X,Y]]$ is a free associative algebra in two  (non-commuting)
variables $X$ and $Y$ with coefficients in $\Q $. 
Then the classical Campbell-Hausdorff 
formula 
$$X\circ Y=\log (\exp(X)\cdot \exp(Y))=X+Y+(1/2)[X,Y]+\dots $$ 
has $p$-integral coefficients modulo $p$-th commutators. 
Therefore, for any 
topological 
Lie $\F _p$-algebra $L$ of nilpotent class $<p$, we can introduce the  
topological group $G(L)$ which equals $L$ as a set and is provided 
with the Campbell-Hausdorff composition law 
$l_1\circ l_2=l_1+l_2+(1/2)[l_1,l_2]+\dots $. The correspondence 
$L\mapsto G(L)$ induces equivalence of the category 
of Lie $\F _p$-algebras of nilpotent class $s_0<p$ 
and the category of $p$-groups of period $p$ of the same nilpotent class 
$s_0$ \cite{La}.  
Note that under this equivalence any morphism of Lie algebras $L_1\To L$ 
is at the same time a group homomorphism $G(L_1)\To G(L)$. In particular, the ideals $I$ of a 
Lie algebra $L$ are precisely all normal subgroups $G(I)$ in $G(L)$, and two elements 
$l_1,l_2$ of the Lie algebra $L$ are congruent modulo the ideal $I$ if and only 
if these elements (when considered as elements of the group $G(L)$)  
are congruent modulo the normal subgroup $G(I)$.

Let $K$ be a complete discrete valuation field with finite residue field 
$k\simeq\F _{p^{N_0}}$, $N_0\in\N $. 
Denote by $K_{sep}$ a separable closure of $K$ and 
set $\op{Gal}(K_{sep}/K)=\Gamma _K$. 

A profinite group structure of $\Gamma _K$  
is well-known, \cite{JW}. Most significant information 
about this structure comes from the maximal $p$-quotient 
$\Gamma _K(p)$ of $\Gamma _K$, \cite{Ko, Se2, Se3}.  As a matter of fact, 
the structure of $\Gamma _K(p)$ is not too complicated: its 
(topological) module of generators equals $K^*/K^{*p}$ and  
if $K$ has no non-trivial $p$-th roots of unity 
(e.g. if $\op{char}K=p$) then $\Gamma _K(p)$ is pro-finite free; otherwise, 
$\Gamma _K(p)$  has only one (the Demushkin) relation of a very special form. 

On the other hand, $\Gamma _K$ has additional structure given by the decreasing 
series of normal (ramification) subgroups $\Gamma _K^{(v)}$, $v\geqslant 0$. 
This additional structure on $\Gamma _K$ (or even on 
the pro-$p$-group $\Gamma _K(p)$) 
is sufficient to recover all properties of the original complete discrete  
valuation field $K$, \cite{Mo, Ab6, Ab11}. 

Note that on the level of abelian extensions the 
ramification filtration of $\Gamma _K^{ab}$ 
is completely described by class field theory and 
has very simple structure. But already on 
the level of $p$-extensions with Galois groups of 
nilpotent class $\geqslant 2$,   
the ramification filtration starts demonstrating 
highly non-trivial behaviour, cf. \cite{Ab2, Ab4, Go1, Go2}. 

In \cite{Ab1, Ab2, Ab3} the author introduced new techniques 
(nilpotent Artin-Schreier theory) which allowed us to work 
with $p$-extensions  
of characteristic $p$ with Galois groups of nilpotent class $<p$. 
As we have mentioned already,  such groups come from Lie algebras 
and our main result describes the 
ideals coming from ramification 
subgroups. 

Consider the case of complete discrete valuation fields $K$ of mixed 
characteristic containing a primitive $p$-th root of unity $\zeta _1$.  
Let $K_{<p}$ be the maximal $p$-extension of $K$ in 
$K_{sep}$ with the Galois group 
of nilpotent class $<p$ and period $p$. Then $\Gamma _{<p}:=\op{Gal}(K_{<p}/K)=
\Gamma _K/\Gamma _K^pC_p(\Gamma _K)$ is a group 
with  finitely many generators and one relation. 
(This terminology makes sense in the category of 
$p$-groups of nilpotent class $<p$ and period $p$.)
Let $\{\Gamma _{<p}^{(v)}\}_{v\geqslant 0}$ be the ramification filtration 
of $\Gamma _{<p}$. If $L$ is a Lie $\F _p$-algebra 
such that $\Gamma _{<p}=
G(L)$ then for all $v$, $\Gamma _{<p}^{(v)}=G(L^{(v)})$, where 
$L^{(v)}$ are ideals in $L$. In this paper we determine 
the structure of $L$   
and ``ramification'' ideals $L^{(v)}$. In particular, 
the Demushkin relation in $L$ 
appears in our setting in terms related directly to 
the ramification ideals $L^{(v)}$.

Note that a similar technique (papers in progress) 
can be used to treat not only more general groups $\Gamma _{<p}(M):=
\Gamma _K/\Gamma _K^{p^M}C_p(\Gamma _K)$, $M\in\N $, but 
also the case of higher local fields $K$.

For the first approach to the above problem cf. \cite{Zi}, where   
the ramification filtration in 
$\Gamma _K^pC_2(\Gamma _K)/\Gamma _K^pC_3(\Gamma _K)$ was studied  
under some restrictions to the basic field $K$. The methods and techniques 
from \cite{Zi} could not be applied to a more general situation.  
The principal advantage of our method is that from the very beginning 
we work with the whole group $\Gamma _{<p}$ rather than with the  
quotients of its central series. 

\subsection{Main steps}  \label{S0.1} 
\ \ 

a) {\it Relation to the characteristic $p$ case.}

Let $\pi _0$ be a fixed uniformizer in $K$ and 
$\wt{K}=K(\{\pi _n\ |\ n\in\N\})$, where $\pi _n^p=\pi _{n-1}$. 
Then the field-of-norms functor $X$ \cite{Wtb1}, 
gives us a complete discrete  
valuation field  $X(\wt K)=\c K$  
of characteristic $p$ with residue field $k$ and fixed uniformizer 
$t=\varprojlim \pi _n$.  
We have also a natural identification of $\c G=
\Gal (\c K_{sep}/\c K)$ with $\Gamma _{\wt{K}}=\Gal (\bar K/\tilde{K})$, 
which is  compatible 
with the appropriate ramification filtrations in 
$\c G$ and $\Gamma _K$ 
via the Herbrand function $\varphi _{\wt{K}/K}$.  
This gives us the following fundamental short exact sequence 
in the category of $p$-groups (where $\c G_{<p}:=\c G/\c G^pC_p(\c G)$) 
\begin{equation} \label{E0.1} 
\c G_{<p}\overset{\iota _{<p}}
\To \Gamma _{<p}\To \op{Gal}(K(\pi _1)/K)
\left (=\langle\tau _0\rangle ^{\Z /p}\right )\To 1\, ,
\end{equation} 
where $\tau _0$ is such that $\tau _0(\pi _1)=\zeta _1\pi _1$. 
\medskip 

b) {\it Nilpotent Artin-Schreier theory.} 
\medskip 

This theory allows us to fix an identification 
$\c G _{<p}=G(\c L)$, where $\c L$ is a profinite 
Lie algebra over $\F _p$.  
The identification depends only on the above uniformizer $t$ in $\c K$ 
and a choice of 
$\alpha _0\in k$ such that $\op{Tr}_{k/\F _p}(\alpha _0)=1$. 
This theory also 
provides us with 
the system of free generators 
$\{D_{an}\ |\ \op{gcd}\,(a,p)=1, n\in\Z /N_0 \}\cup \{D_0\}$ 
of $\c L_k$. Note that we shall treat $D_0$ in the context of 
all $D_{an}$ by setting for all 
$n\in\Z /N_0$, $D_{0n}=(\sigma ^n\alpha _0)D_0$. 
\medskip 

c) {\it Ramification filtration in 
$\c G _{<p}$.}
\medskip 

With respect to the above identification 
$\c G_{<p}=G(\c L)$, the ramification subgroups 
$\c G _{<p}^{(v)}$ come from the ideals $\c L^{(v)}$ 
of $\c L$. In \cite{Ab1, Ab2, Ab3} 
we constructed explicitly the elements $\c F^0_{\gamma ,-N}\in\c L_k$ with  
non-negative 
$\gamma \in\Q $ and $N\in\Z $, such that 
for any $v\geqslant 0$ and sufficiently large $N\geqslant \wt{N}(v)$, 
$\c L^{(v)}$ appears as the minimal ideal in $\c L$ such that 
$\c F^0_{\gamma ,-N}\in\c L^{(v)}_k$ for all $\gamma\geqslant v$.  
\medskip 

d) {\it Fundamental sequence of Lie algebras.}
\medskip 

Using the above mentioned equivalence of the categories 
of $p$-groups and Lie algebras 
we can replace \eqref{E0.1} by the following exact 
sequence of Lie $\F _p$-algebras 
\begin{equation} \label{E0.2} 
0\To \c L/\c L(p)\To L\To \F _p\tau _0\To 0\, , 
\end{equation}
where $G(\c L(p))=\op{Ker}\,\iota _{<p}$ and 
$G(L)=\Gamma _{<p}$. If $\tau _{<p}$ is a lift 
of $\tau _0$ to $L$ then the structure of \eqref{E0.2} 
can be given via the differentiation 
$\op{ad}\tau _{<p}$ on $\bar{\c L}=\c L/\c L(p)$. 
\medskip 

e) {\it Replacing $\tau _0$ by $h \in\op{Aut}\c K$}.  
\medskip 

When studying the structure of \eqref{E0.2} we can approximate $\tau _0$ by 
some $h\in\op{Aut}\c K$. 
 This automorphism $h$ is defined in terms of the expansion 
of $\zeta _1$ in powers of our fixed uniformizer $\pi _0$. 
Then the formalism of nilpotent Artin-Schreier theory 
allows us to specify a lift $\tau _{<p}$, 
to find the ideal $\c L(p)$ and to 
introduce a recurrent procedure of  
obtaining the values $\op{ad}\tau _{<p}(D_{an})\in\bar{\c L}_k$ and 
$\op{ad}\tau _{<p}(D_0)\in\bar{\c L}$.
\medskip 

f) {\it Structure of $L$.}
\medskip 

Analyzing the above recurrent procedure modulo $C_2(\bar{\c L})_k$ 
we can see that the knowledge of the elements 
$\op{ad}\tau _{<p}(D_{an})$ allows us to kill all generators 
$D_{an}$ of $\bar{\c L}_k$ with $a>e^*:=e_Kp/(p-1)$. 
(Here $e_K$ is the ramification index of $K$ over $\Q _p$.)  
In other words, $L_k$ has the minimal system of generators 
$\{D_{an}\ |\ 1\leqslant a<e^*, n\in\Z/N_0\}
\cup\{D_0\}\cup\{\tau _{<p}\}$. On the other hand,  
$\op{ad}\tau _{<p}(D_0)\in C_2(\bar{\c L})\subset C_2(L)$ and, therefore, gives us  
the (unique) Demushkin  relation in $L$.
\medskip 

g) {\it Ramification subgroups $L^{(v)}$ in $L$.}
\medskip 

For $v> e^*$, all ramification ideals $L^{(v)}$ are contained in $\bar{\c L}$ and  
come from the appropriate 
ideals ${\c L}^{(v')}$, where the upper indices 
$v$ and $v'$ are related by the Herbrand function $\varphi _{\wt{K}/K}$ 
of the field extension 
$\wt K/K$. 
As one of immediate 
applications we found 
for $2\leqslant s<p$, the biggest 
upper ramification numbers $v[s]$ of the maximal $p$-extensions 
$K[s]$ of $K$ with the Galois groups of period $p$ and 
nilpotent class $\leqslant s$. We shall get the remaining ramification ideals 
$L^{(v)}$ with $v\leqslant e^*$ if we 
specify a ``good'' lift $\tau _{<p}$ 
of $\tau _0$, i.e. such that $\tau _{<p}\in L^{(e^*)}$. 
(The concept of a ``good'' lift is formalized in 
the definition of arithmetical lift in Subsection \ref{S4.2}.) 
This is the most difficult 
part of the paper 
where we need a technical result from \cite{Ab3}. 
\medskip 

h) {\it Explicit formulas for $\op{ad}\tau _{<p}$ 
with ``good'' $\tau _{<p}$.}
\medskip 

The formulas for $\op{ad}\tau _{<p}(D_{an})$ and $\op{ad}\tau _{<p}(D_0)$ 
can be obtained modulo $C_3(L_k)$ as a second central step 
in our recurrent procedure mentioned in above item e), 
cf. calculations in Subsection \ref{S3.6}. In Section \ref{S5} we 
obtain a general formula for $\op{ad}\tau _{<p}(D_0)$. This gives  
an explicit form of the Demushkin relation in terms of the 
ramification generators $\c F^0_{\gamma ,-N}$ from item c). 
\medskip 

\subsection{Main results} \label{S0.2}

Introduce the weights $\op{wt} (l)$ of elements 
$l\in\c L_k$ by setting $\op{wt} (D_{an})=s\geqslant 1$ if 
$(s-1)e^*\leqslant a<se^*$, i.e. $\op{wt}(D_{an})=[a/e^*]+1$.

\begin{Thm} \label{T0.1} 
 {\rm a)} \ $\c L(p)=\{l\in\c L\ |\ \op{wt}(l)\geqslant p\}$;
 \medskip  

 {\rm b)}\ if  
 $\c L(s)=\{l\in\c L\ |\ \op{wt}(l)\geqslant s\}$  
 then $C_s(L)=\c L(s)/\c L(p)$. 
\end{Thm}

Suppose for all $a$, $V_{a0}\in \bar{\c L}_k$ are such that 
$\ad\tau _{<p}(D_{a0})=V_{a0}$. In particular, 
$V_{00}=\alpha _0V_0$, where $V_0=(\ad\tau _{<p})D_0\in\bar{\c L}$. 
The knowledge of 
these elements determines uniquely the differentiation $\ad\tau _{<p}$ 
(note that for all $n$, $\ad\tau _{<p}(D_{an})=\sigma ^n(V_{a0})$). 

Suppose $E(X)=\exp (X+X^p/p+\dots +X^{p^n}/p^n+\dots)\in \Z _p[[X]]$ 
is the Artin-Hasse exponential. 

Let $\omega (t)\in k[[t]]$ be such that 
$E(\omega (\pi _0))=\zeta _1\,\op{mod}\,p$. 
\medskip 

\begin{Thm} \label{T0.2}
 The elements $V_{a0}$ can be found from the following 
recurrent relation in $\bar{\c L}_{\c K}$
$$ 
\sigma c_1-c_1+\sum _{a\in \Z ^0(p)}t^{-a}V_a=$$ 
$$-\sum _{k\geqslant 1}\,\frac{1}{k!}\,t^{-(a_1+\dots +a_k)}
\omega (t) ^p
[\dots [a_1D_{a_10},D_{a_20}],\dots ,D_{a_k0}]$$
$$-\sum _{k\geqslant 2}\frac{1}{k!}t^{-(a_1+\dots +a_k)}
[\dots [V_{a_1},D_{a_20}],\dots ,D_{a_k0}]$$
$$-\sum _{k\geqslant 1}\,\frac{1}{k!}\,t^{-(a_1+\dots +a_k)}
[\dots [\sigma c_1,D_{a_10}],\dots ,D_{a_k0}], $$
where in all last three sums the indices  
$a_1,\dots ,a_k$ run over the set  $\Z ^0(p):=\{a\in\N\ |\ \op{gcd}(a,p)=1\}\cup\{0\}$.
\end{Thm}

In the above system of equations we are looking for the solutions of the form 
$\{c_1\in \bar{\c L}_{\c K}, \{V_{a0}\in\bar{\c L}_k\ |\ a\in\Z ^0(p)\}\}$. These solutions 
correspond to different choices of the lift $\tau _{<p}$ of $\tau _0$, in particular, 
$c_1$ is (very strict) invariant of such a lift $\tau _{<p}$. 

Suppose $c_1=\sum _{m\in\Z}c_1(m)t^m$. 

Let $\bar{\c L}^{(e^*)}$ be the image of $\c L^{(e^*)}$ in $\bar{\c L}$. 

Let   
$\omega (t)^p=\sum _{j\geqslant 0}A_jt^{e^*+pj}$ 
with coefficients $A_j\in k$. 

\begin{Thm} \label{T0.3} 
 $\tau _{<p}$ is a ``good'' lift, cf. Subsection \ref{S0.1} step g),  
 iff  
 $$c_1(0)\equiv \sum _{j\geqslant 0}
\sum _{0\leqslant i<\wt{N}(e^*)}\sigma ^i(A_j\c F^0_{e^*+pj,-i})                          
\,\op{mod}\,\bar{\c L}_k^{(e^*)}\, ,$$
cf. item c) for the definition of $\wt{N}(e^*)$. 
\end{Thm}

\begin{Thm} \label{T0.4}
 {\rm a)}\ If $v>e^*$ then $\Gamma _{<p}^{(v)}=G(L^{(v)})$, where 
 $L^{(v)}$ is the image of 
$\c L^{(v^*)}$ in $\bar{\c L}$ and $v^*=e^*+p(v-e^*)$; 
\medskip 

{\rm b)}\ if $v\leqslant e^*$ and $\tau_{<p}$ is ``good''  
then $\Gamma _{<p}^{(v)}=G(L^{(v)})$, where $L^{(v)}$ is  
generated by the image of $\c L^{(v)}$ in $\bar{\c L}$ and $\tau _{<p}$. 
\end{Thm}

\begin{Thm} \label{T0.5}
 If $2\leqslant s<p$ then $v[s]=e_K(1+s/(p-1))-1/p$. 
\end{Thm}

\begin{remark} 
$v[1]=e^*(=e_K(1+1/(p-1))$ is a well-known fact which follows directly 
from definitions and Kummer theory.
\end{remark}

 Consider the set of all $(a_1,n_1,\dots ,a_s,n_s)$, where all $a_i\in\Z ^0(p)$,   
$n_i\in\Z $ are such that $n_1\geqslant n_2\geqslant\dots\geqslant n_s=0$ and 
$\sum _{1\leqslant i\leqslant s}[a_i/e^*]\leqslant p-1-s$. 

Let 
$\delta ^+(e^*)$ be the minimum of positive 
values of 
$$(e^*+pj)-p^{-n_1}(a_1p^{n_1}+\dots +a_sp^{n_s})\,,$$
where $(a_1,n_1,\dots ,a_s,n_s)$ runs over the set of above defined vectors 
and $j$ runs over the set of all non-negative integers. 
Set 
$$N^+(e^*)=\min\{n\in\N \ |\ p^n\delta ^+(e^*)\geqslant e^*(p-1)\}\, .$$ 

**********

Fix $N^0\geqslant N^+(e^*)-1$ and set $\Omega ^0=\sum _{j\geqslant 0}
A_j\c F^0_{e^*+pj,-N^0}$.  
\medskip  

Introduce the operators $F_0$ and $G_0$ on $\bar{\c L}_k$ such that for any $l\in\bar{\c L}_k$,  
$$F_0(l)=\hskip -4pt \sum _{1\leqslant k<p}\hskip -4pt\frac{\alpha _0^{k-1}}{k!}
[\dots [l,\underset{k-1\ \text{times}}
{\underbrace{D_{0}],\dots ,D_{0}}}], 
G_0(l)=\hskip -4pt \sum _{0\leqslant k<p}\hskip -4pt \frac{\alpha _0^k}{k!}[\dots [l,
\underset{k\ \text{times}}{\underbrace{D_{0}],\dots ,D_{0}}}]\, .$$

Consider the relation 
\begin{equation} \label{E0.3} 
(G_0\sigma -\id )c^0+F_0(V_0)=-G_0\sigma ^{N^0+1}\Omega ^0\,.
 \end{equation}

\begin{Thm} \label{T0.6} {\rm a)}\ There is a bijection between different lifts $\tau _{<p}$ and 
solutions $(c^0, V_0)$ of relation \eqref{E0.3}, with $c^0\in\bar{\c L}_k$ and 
$V_0\in\bar{\c L}$. 
\medskip 

{\rm b)}\ If $\tau _{<p}$ corresponds to $(c^0,V_0)$ then the Demushkin relation 
appears in the form 
$(\ad\,\tau _{<p})D_0=V_0$;
\medskip 

{\rm c)}\ If $N^0\geqslant\wt{N}(e^*)$ then 
 $\tau _{<p}$ is ``good'' if and only if $c^0\in \bar{\c L}_k^{(e^*)}$. 
\end{Thm}

\begin{Cor} \label{C0.7} 
 {\rm a)}\ For any lift $\tau _{<p}$, 
 $$(\op{ad}\,\tau _{<p})D_0
+\sum _{0\leqslant n<N_0}\sigma ^n(\Omega ^0)\in [\bar{\c L},D_0] ;$$
\medskip 
{\rm b)}\ if $k=\F _p$ then there is a ``good'' lift $\tau _{<p}$, such that 
the Demushkin relation appears in the form  
$(\op{ad}\tau _{<p})D_0+F_0^{-1}(\Omega ^0)=0$.  
 \end{Cor}

\subsection{Concluding remarks} \label{S0.3} 

Our description of $\Gamma _{<p}$ 
together with its ramification filtration may serve as a 
guide to what we could 
expect a nilpotent local class field theory should be. 
Our approach gives the objects of this theory 
on the level of quotients of nilpotent class $<p$ 
together with induced ramification filtration. 
Regretfully, our description is not functorial: 
it depends on a choice of a uniformizing element in $K$. 
 
It would be very interesting to compare our results with 
the construction of $\Gamma _K$ in \cite{LF1}, cf. also \cite{KdS}. 
This construction uses iterations of the Lubin-Tate theories via  
the field-of-norms functor and is done inside the group of 
formal power series with the 
operation given by their composition. However, it 
is not clear how to extract from that  
construction even well-known properties of  
the Galois group of a maximal $p$-extension of $K$.  
\medskip  

The content of this paper is arranged in a slightly different order 
compared to above principal steps a)-h). In Section 1 we briefly 
discuss auxiliary facts and constructions from the characteristic $p$ case. 
In Section 2 we study an analogue $\c G_{h}$ of 
$\Gamma _{<p}$ which appears if 
we replace $\tau _0$ by a suitable $h\in\op{Aut}\c K$; we also  describe  
the commutator subgroups of $\c G_{h}$ and, in particular, 
find the appropriate ideal $\c L(p)$.   
In Section 3 we develop the techniques 
allowing us to switch the languages of $p$-groups and Lie algebras. 
In Section 4 we establish the Criterion to 
characterize ``good'' lifts $h_{<p}$ of $h$  
and in Section 5 we compute the appropriate \lq\lq Demushkin\rq\rq\ relation for such ``good'' lifts. 
Finally, in Section 6 we prove that all our results obtained for the group 
$\c G_{h}$ are actually valid in the context of the group $\Gamma _{<p}$. 
\medskip 

{\it Acknowledgements.} The author expresses a deep gratitude to the referee: 
his advices allowed the author to avoid a considerable amount of inexactitudes and 
to improve very much the quality of the original exposition.

\section{Preliminaries}  \label{S1} 

\subsection{Covariant nilpotent Artin-Schreier theory} \label{S1.1} 

Suppose $\c K$ is a field of characteristic $p$,  
$\c K_{sep}$ is a separable closure of $\c K$ and 
$\c G=\Gal (\c K_{sep}/\c K)$. We assume that the composition $g_1g_2$ 
of $g_1,g_2\in\c G$ is such that for any $a\in\c K_{sep}$, 
$g_1(g_2a)=(g_1g_2)a$.

In \cite{Ab1, Ab2, Ab3} the author developed a nilpotent analogue of the classical 
Artin-Schreier theory of cyclic extensions of fields of characteristic $p$. 
The main results of this theory (which will 
be called the contravariant nilpotent Artin-Schreier theory) 
can be briefly explained as follows. 

Let $\c G^0$ be the group such that $\c G^0=\c G$ as sets but 
for any $g_1,g_2\in\c G$ their composition in $\c G^0$ equals $g_2g_1$. 
In other words, we assume that $\c G^0$ acts on $\c K_{sep}$ via 
$(g_1g_2)a=g_2(g_1(a))$.

Let $L$ be a Lie $\F _p$-algebra of nilpotent class $<p$. 
Then the absolute Frobenius $\sigma $ and 
$\c G^0$ act on $L_{\c K_{sep}}$ through the second 
factor. We have $L_{\c K_{sep}}|_{\sigma =\id }=L$ and  
$(L_{\c K_{sep}})^{\c G^0}=L_{\c K}$.  

For any $e\in G(L_{\c K})$, the set of $f\in G(L_{\c K_{sep}})$ such that 
$\sigma (f)=f\circ e$ is not empty. Define 
the group homomorphism 
$\pi ^0_f(e):\c G^0\to G(L)$ by setting for any $g\in \c G^0$, 
$\pi ^0_f(e):g\mapsto g(f)\circ (-f)$. 

\begin{remark} 
 Strictly speaking $g(f)$, where $g\in\c G^0$, should be written in the form 
$(\id _{L}\otimes g)f$ but in most cases we use the first notation. 
On the other hand, we would prefer the second notation if, 
say, $g\in \op{Aut} \c K_{sep}$ and 
$g|_{\c K}\ne\id _{\c K}$. (Similarly, we have already agreed in the Introduction 
 to use the notation 
$\sigma $ instead of $\id _{L}\otimes\sigma $.) 
\end{remark}

We have the following properties:
\medskip 

a) for any group homomorphism $\eta :\c G^0\to G(L)$ 
there are $e_{\eta }\in G(L_{\c K})$ and $f_{\eta }\in G(L_{\c K_{sep}})$ such that 
$\sigma (f_{\eta })=f_{\eta }\circ e_{\eta }$ and 
$\eta =\pi ^0_{f_{\eta }}(e_{\eta })$;
\medskip 

b) two homomorphisms $\pi ^0_f(e)$ and $\pi ^0_{f_1}(e_1)$ 
from $\c G^0$ to $G(L)$ are conjugated via some element from $G(L)$ iff 
there is an 
$x\in G(L_{\c K})$ such that $e_1=(-x)\circ e\circ \sigma (x)$. 
\medskip 

The covariant version of the above theory can be developed quite similarly. 
We just use the relations $\sigma (f)=e\circ f$ and 
$g\mapsto (-f)\circ g(f)$ to define the group homomorphism 
$\pi _f(e):\c G\To G(L)$. Then we have the obvious analogs 
of above properties 
a) and b) with  the opposite formula $e_1=\sigma (x)\circ e\circ (-x)$ 
in the case b). 

In this paper we use the covariant theory but need some results from 
\cite{Ab3} which were obtained in the contravariant setting. These  
results can be adjusted to the covariant theory just by 
replacing all involved group or Lie structures to the opposite ones, e.g. 
cf. Subsection \ref {S1.4} below.

\subsection{Lifts of analytic automorphisms} \label{S1.2}
Let $\Aut \,\c K$ and $\Aut \,{\c K_{sep}}$ 
be the groups of continuous automorphisms of 
$\c K$ and $\c K_{sep}$, respectively.  
For $h\in\Aut\, {\c K}$, let $h_{sep}\in\Aut\, {\c K_{sep}}$ be 
a lift of $h$, i.e. $h_{sep}|_{\c K}=h$.

Suppose $L$ is a Lie $\F _p$-algebra of nilpotent class $<p$.  
Let $e\in G(L_{\c K})$, choose $f\in G(L_{\c K_{sep}})$ such that 
$\sigma (f)=e\circ f$, set $\eta =\pi _f(e)$ and $\c K_e=
\c K_{sep}^{\Ker\,\eta }$. Then $\c K_e$ does not depend on 
a choice of $f$: if $f'\in G(L_{\c K_{sep}})$ is such that 
$\sigma (f')=e\circ f'$ then $f'=f\circ l$ with 
$l\in G(L)$ 
and $\Ker\,\eta =\Ker\,\pi _{f'}(e)$. 

\begin{Prop} \label{P1.1} 
Suppose $\eta :\c G\To G(L)$ is epimorphic. Then 
the following conditions are equivalent:
 \medskip 
 
 {\rm a)} $h_{sep}(\c K_e)=\c K_e$;
 \medskip 
 
 {\rm b)} there are $c\in G(L_{\c K})$ and $A\in\Aut L$ such that 
 $(\id _{L}\otimes h_{sep})(f)=c\circ (A\otimes\id _{\c K_{sep}})(f)$. 
\end{Prop}

\begin{proof} Let $e_1=(\id _{L}\otimes h)e$, $f_1=(\id _{L}\otimes h_{sep})f$ 
and $\eta _1=\pi _{f_1}(e_1)$. Then 
for any $g\in\c G$, 
we have \ \ 
$\eta _1(g)=(-f_1)\circ g(f_1)=$\newline 
$$(\id _L\otimes h)((-f)\circ (h_{sep}^{-1}\,g\, h_{sep})f)=
\eta (h_{sep}^{-1}\,g\, h_{sep}).$$ 
Therefore, $\eta _1$ is equal to the composition of the conjugation 
by $h_{sep}$  on $\c G$ (we shall denote it by $\op{Ad}\, h_{sep}$ below) and $\eta $. 
Then $h_{sep}(\c K_e)=\c K_e$ means that $\Ker\,\eta =\Ker\,\eta _1$.  
This implies the existence of an automorphism $A$ of the group $G(L)$ 
(which is automatically automorphism of the Lie algebra $L$) such that 
$\eta _1=A\eta $. 

Now let $f'=(A\otimes\id _{\c K_{sep}})f$ and $e'=(A\otimes \id _{\c K})e$. 
Then  
$\pi _{f'}(e')g=(A\otimes\id _{\c K_{sep}})((-f)\circ 
 g(f))=(A\eta )g=\eta _1(g)$. This means that 
$f'$ and $f_1$ give the same morphisms $\c G\to G(L)$ and there is  
$c\in G(L_{\c K})$ such that $f_1=c\circ f'$, that is 
a) implies b). Proceeding in the opposite direction we can deduce 
b) from a).
\end{proof}

\begin{remark}
 From the proof of the above proposition it follows 
that a choice of the lift $h_{sep}$ uniquely determines its ingredients 
$c\in L_{\c K}$ and $A\in\Aut _{Lie} L$. 
Indeed, $A$ appears as $\Ad (h_{sep}|_{\c K_e})$ (with respect to the identification 
$\c G/\Ker\,\eta =G(L)$ induced by $\eta $) and 
$c$ is recovered then as $(\id _L\otimes h_{sep})f\circ (A\otimes \id _{\c K_{sep}})(-f)$. 
This shows that the couple $(c,A)$ depends only on the restriction 
$h_{sep}|_{\c K_e}$ and we can consider the map $h_{sep}|_{\c K_e}\mapsto (c,A)$ 
from the set of all lifts of $h$ to $\c K_e$ to the set of appropriate couples $(c,A)$. 
But the knowledge of $(c,A)$ allows us to recover uniquely the element 
$(\id _{\c L}\otimes h_{sep})f$ and  
the Galois group $\Gal (\c K_e/\c K)$ acts strictly on the set of all such elements. 
Therefore, any couple $(c,A)$ appears from no more than one lift of $h$ to $\c K_e$, that is 
the map $h_{sep}|_{\c K_e}\mapsto (c,A)$ is injective. 
We will study this map in more details below, 
cf. Proposition \ref{P2.3}. 
\end{remark}

\subsection{The identification $\eta _0$} \label{S1.3}

Let $\c K=k((t))$ be a complete discrete valuation field of Laurent formal 
power series in variable $t$ with coefficients in $k\simeq \F _{p^{N_0}}$, 
$N_0\in\N $. Choose $\alpha _0\in k$ such that $\op{Tr}_{k/\F _p}\alpha _0=1$.

Let $\Z ^+(p)=\{a\in\N\ |\ (a,p)=1\}$ and $\Z ^0(p)=\Z ^+(p)\cup\{0\}$. 
Denote by $\wt{\c L}_k$ a free pro-finite Lie algebra over $k$ with the set of 
free generators $\{D_{an}\ |\ a\in\Z ^+(p), n\in\Z /N_0\}\cup\{D_0\}$.  
As earlier, denote by the same symbol $\sigma $, the $\sigma $-linear automorphism 
of $\wt{\c L}_k$ such that $\sigma :D_0\mapsto D_0$ and 
for all $a\in \Z ^+(p)$ and $n\in\Z /N_0$, $\sigma :D_{an}\mapsto D_{a,n+1}$. 
Then $\wt{\c L}^0:=\wt{\c L}_k|_{\sigma =\id }$ is a free pro-finite 
Lie $\F _p$-algebra and $\wt{\c L}_k=\wt{\c L}_k^0$. 

Let ${\c L}=\wt{\c L}^0/C_p(\wt{\c L}^0)$. 

For any $n\in\Z /N_0$, set 
$D_{0n}=\sigma ^n(\alpha _0)D_0$. 

Let $e=\sum _{a\in\Z ^0(p)}t^{-a}D_{a0}\in G({\c L}_{\c K})$ and let  
$f\in G({\c L}_{\c K_{sep}})$ be such that 
$\sigma (f)=e\circ f$. 
Then the morphism $\eta =\pi _{f}(e)$ induces the  
isomorphism of topological groups $\eta _0:\c G_{<p}:=
\c G/\c G^pC_p(\c G)\tilde\To G({\c L})$. 
\medskip 

In the remaining part of the paper we shall use 
(without additional notice) the above introduced notation $e$, $f$, $\eta $ and $\eta _0$. 
The appropriate field $\c K_e$ coincides with   
$\c K_{sep}^{\c G^pC_p(\c G)}$ and will be denoted by $\c K_{<p}$. 
\medskip 

Note that $f\in G(\c L_{\c K_{<p}})$. In particular, 
if $h_1,h_2\in\,\op{Aut}\,\c K_{sep}$ are such 
that $h_1|_{\c K}=h_2|_{\c K}$ and $(\id _{\c L}\otimes h_1)f=(\id _{\c L}\otimes h_2)f$ then 
$h_1|_{\c K_{<p}}=h_2|_{\c K_{<p}}$, cf. Remark at the end of Subsection \ref{S1.2}.  Therefore, 
the appropriate choice of the ingredients 
$c\in\c L_{\c K}$ and $A\in\op{Aut}\,\c L$ from Proposition \ref{P1.1} 
can be used to describe efficiently the lifts of automorphisms $h$ 
of $\c K$ to automorphisms $h_{<p}$ of $\c K_{<p}$. 
We shall also use below in Subsections 
\ref{S2.2} and \ref{S4.5} the following interpretation 
of this property. Suppose $\c L _1$ is an ideal in $\c L$  
and $\c K_{<p}^{G(\c L_1)}=\c K _1$. Then 
$f\,\op{mod}\,\c L_{1\c K_{<p}}$ is defined over 
$\c K_1$. In other words, 
$f\,\op{mod}\,\c L_{1\c K_{<p}}\in (\c L/\c L_1)_{\c K_1}
\subset (\c L/\c L_1)_{\c K_{<p}}$, or 
$f\in\c L_{\c K_1}+\c L_{1\c K_{<p}}$. Note that $\eta :\c G\To G(\c L)$ induces 
(via using $f\,\op{mod}\, \c L_{1\c K_{<p}}$) the identification 
$\Gal (\c K_1/\c K)\simeq G(\c L/\c L_1)$. 

\medskip 

If $h\in\op{Aut}\,\c K$ then its lifts to 
$\op{Aut}\,\c K_{<p}$ will be denoted usually by $h_{<p}$. 
  As we have already pointed out, $G(\c L)$ acts transitively 
on the set of all lifts $h_{<p}$ of a given $h$: for any $l\in G(\c L)$, 
$h_{<p}\mapsto h_{<p}*l=h_{<p}\,\eta _0^{-1}(l)$. 
\medskip 

\subsection{The ramification subgroups in $\c G_{<p}$} \label{S1.4} 
 
For $v\geqslant 0$, let $\c G_{<p}^{(v)}$ be the image of 
the ramification subgroup $\c G^{(v)}$ of $\c G$ in $\c G_{<p}$.  
This subgroup corresponds to some ideal ${\c L}^{(v)}$ of the Lie algebra 
${\c L}$ with respect to the identification $\eta _0$.  

When working with the above standard generators of $\c L_k$ we very often 
denote them by $D_{an}$, where $n\in\Z $, by having in mind  that 
they depend only on the residue of $n$ modulo $N_0$, i.e. $D_{an}=D_{a,n+N_0}$. 

For $\gamma\geqslant 0$ and $N\in\N $, introduce 
$\c F^0_{\gamma ,-N}\in{\c L}_k$ such that 
$$\c F^0_{\gamma ,-N}=\sum _{\substack {1\leqslant s<p\\
a_i, n_i }}a_1\eta (n_1,\dots, n_s)[\dots [D_{a_1 n_1},
D_{a_2 n_2}],\dots ,D_{a_s n_s}]$$

Here: 
\medskip 

--- \ $a_1p^{n_1}+a_2p^{n_2}+\dots +a_sp^{n_s}=\gamma $;
\medskip 

--- \ if $0=n_1=\dots =n_{s_1}>\dots >n_{s_{r-1}+1}=\dots =n_{s_r}\geqslant -N$ 
then $\eta (n_1,\dots ,n_s)=(s_1!\dots (s_r-s_{r-1})!)^{-1}$; 
otherwise, $\eta (n_1,\dots ,n_s)=0$.

\begin{Thm} \label{T1.2}
For any $v\geqslant 0$, there is $\wt{N}(v)$ such that 
if $N\geqslant \wt{N}(v)$ is fixed then  
the ideal ${\c L}^{(v)}$ is the minimal ideal in ${\c L}$ 
such that its extension of scalars ${\c L}_k^{(v)}$ contains all 
$\c F^0_{\gamma ,-N}$ with $\gamma\geqslant v$. 
 \end{Thm}
 
 The appropriate theorem in the contravariant setting was obtained 
 in \cite{Ab1} (or in a more general form in the context of 
groups of period $p^M$ in \cite{Ab3}) 
 and uses 
 the elements $\c F_{\gamma ,-N}$ given by the same formula but 
 with the factor $(-1)^{s-1}$. Indeed, when switching to the 
covariant setting all commutators of the form  
 $[\dots [D_{a_1n_1},D_{a_2n_2}],\dots ,D_{a_sn_s}]$ should be replaced by 
 $[D_{a_sn_s},\dots ,[D_{a_2n_2},D_{a_1n_1}]\dots ]=
 (-1)^{s-1}[\dots [D_{a_1n_1},D_{a_2n_2}],\dots ,D_{a_sn_s}]$. 
 \medskip 



\section{The groups $\wt{\c G} _{h}$ and $\c G_{h}$} \label{S2}

\subsection{The automorphism $h$} \label{S2.1} Let $c_0\in p\N $. 
Denote by  $h$ a continuous automorphism   
of $\c K$ such that $h |_k=\id $ and 
$$h(t)=t\left (1+\sum _{i\geqslant 0}\alpha _i(h)t^{c_0+pi}\right )\, ,$$ 
where all $\alpha _i(h)\in k$ and $\alpha _0(h)\ne 0$. 
This automorphism will be fixed in the remaining part of the paper. 

Let $E(X)=\exp\left (\sum _{i\geqslant 0}X^{p^i}/{p^i}\right )\in\Z _p[[X]]$ 
be the Artin-Hasse exponential. 

\begin{Prop} \label{P2.1}  \ \ 

 {\rm a)} There is $\omega _h\in t^{c_0/p}O^*_{\c K}$ such that 
$h(t)=tE(\omega _h^p)$; 
 
{\rm b)} For any $n\geqslant 0$, $h^n(t)\equiv 
tE(n\omega _h^p)\,\op{mod}\,t^{1+pc_0}$. 
\end{Prop}

\begin{proof} 
 For part a), $\omega _h$ appears as a unique element from $tk[[t]]$ such that 
 $E(\omega _h)=1+\sum _{j\geqslant 0}\sigma ^{-1}(\alpha _j(h))t^{c_0/p+j}$. 
 (Use that $x\mapsto E(x)-1$ is bijective on $tk[[t]]$.) 
 For part b), note that 
$h(t)\equiv t\,\op{mod}\,t^{c_0}$ implies that 
$h(t^{c_0+pi})\equiv t^{c_0+pi}\,\op{mod}\,t^{pc_0}$ and, therefore, 
$h(\omega _h^p)\equiv \omega _h^p\,\op{mod}\,t^{pc_0}$. 
Now apply induction on $n$.  If  
our proposition is proved for $n\geqslant 1$ then 
$$h^{n+1}(t)\equiv h(t)h(E(n\omega _h^p))\equiv 
tE(\omega _h^p)E(n\omega _h^p)\equiv tE((n+1)\omega _h^p)\,\op{mod}\,t^{pc_0+1}\, $$
(use that $E(X+Y)\equiv E(X)E(Y)\,\op{mod}\,\deg p$). 
\end{proof}

\begin{remark}
 In all applications below the knowledge of the automorphism $h$ will be essential 
 only modulo $t^{1+pc_0}$ and, therefore, in the above 
proposition we can use instead of  
 $E(X)$ the truncated exponential 
 $\wt{\exp}(X)=1+X+\dots +X^{p-1}/(p-1)!$\,. 
\end{remark}

\subsection{Operators $\c R$ and $\c S$} \label{S2.2} 
 
Suppose $\fr M$ is a profinite $\F _p$-module.
Define the continuous $\F _p$-linear operators 
$\c R,\c S:\fr M_{\c K}\To \fr M_{\c K}$ as follows. 

Suppose $\alpha\in\fr M_k$. 

If $n>0$ then set $\c R(t^n\alpha )=0$ and $\c S(t^n\alpha )=
-\sum _{i\geqslant 0}\sigma ^i(t^n\alpha )$.

For $n=0$, set $\c R(\alpha )=\alpha _0\op{Tr}_{k/\F _p}\alpha $, 
\ \ $\c S(\alpha )=\sum _{0\leqslant j<i<N_0}
(\sigma ^j\alpha _0)\sigma ^i\alpha $.

If $n=-n_1p^m$ with $\op{gcd}(n_1,p)=1$ then set 
$\c R(t^n\alpha )=t^{-n_1}\sigma ^{-m}\alpha $ and 
$\c S(t^n\alpha )=\sum _{1\leqslant i\leqslant m}\sigma ^{-i}(t^n\alpha )$. 

The proof of the following lemma is straightforward.

\begin{Lem} \label{L2.2} 
For any $b\in\fr M_{\c K}$,  
\medskip 

{\rm a)}\ $b=\c R(b)+(\sigma -\id _{\fr M_{\c K}})\c S(b)$;
\medskip 

{\rm b)}\ if $b=b_1+\sigma b_2-b_2$, where 
$b_1\in\sum _{a\in\Z ^+(p)}t^{-a}\fr M_k+\alpha _0\fr M$ 
and $b_2\in\fr M_{\c K}$ then $b_1=\c R(b)$ and $b_2-\c S(b)\in\fr M$. 
\end{Lem}

\begin{remark} 
a) The definition of the above operators $\c R$ and $\c S$ in the cases 
$n>0$ and $n<0$ is self-explanatory. In 
the case $n=0$ we have the following picture behind. 
 For $\alpha\in\c L_k$ and $0\leqslant i<N_0$, set 
$\c R_i(\alpha )=\alpha _0\sigma ^{-i}\alpha $ and $\c S_i(\alpha )=
\sum _{0\leqslant j<i}\sigma ^j(\c R_i(\alpha ))$. Then 
$$\alpha =\sum _{0\leqslant i<N_0}(\sigma ^i\alpha _0)\alpha =
\sum _{0\leqslant i<N_0}\sigma ^i\c R_i(\alpha )=
\sum _{0\leqslant i<N_0}\left ((\sigma -\id )
\c S_i+\c R_i\right )(\alpha )$$ 
$$\c R=\sum _{0\leqslant i<N_0}\c R_i\,,\ \ \ \ 
\c S=\sum _{0\leqslant i<N_0}\c S_i\,,$$  
$$\c S(\alpha )=\sum _{0\leqslant j<i<N_0}
\sigma ^j(\alpha _0\sigma ^{-i}\alpha )
=\sum _{0\leqslant j<i_1<N_0}(\sigma ^j\alpha _0)
\sigma ^{i_1}\alpha \,,$$ 
where $i_1=j-i+N_0$. 
Note that there are many other ways to define $\c S$ in 
the case $n=0$.  
\medskip 

b) A typical situation where we refer to the above lemma appears as follows: 
suppose $\fr{N}\subset\frak{M}$ is an 
$\F _p$-submodule and 
$$b=\sum _{a\in\Z ^+(p)}t^{-a}b_a+\alpha _0b_0+\sigma c-c\, ,$$
with all $b_a\in\fr{M}_{k}$, $b_0\in\fr{M}$ and $c\in\fr{M}_{\c K}$; if $b\in\fr{N}_{\c K}$ then 
all $b_a\in \fr{N}_k$, $b_0\in\fr{N}$ and $c\in\fr{M}+\fr{N}_{\c K}$. 
\end{remark}

\subsection{Specification of $h_{<p}$} \label{S2.3} 
 
We are going to specify a lift $h_{<p}$ of $h$ to $\c K_{<p}$ by 
using formalism of nilpotent Artin-Schreier theory.  
Recall that for any lift $h_{<p}$ of $h$, we have  
a unique $c\in \c L_{\c K}$ and $A=\Ad\, h_{<p}\in\Aut\,\c L$ such that 
$(\id _{\c L}\otimes h_{<p})(f)=c\circ (A\otimes\id _{\c K_{<p}})f$. 
The appropriate map $h_{<p}\mapsto (c,A)$ is injective, cf. Subsection \ref{S1.2}. 
The following proposition describes the image of this map. 

\begin{Prop} \label{P2.3} 
 The correspondence $\Pi :h_{<p}\mapsto (c,A)$ induces a bijection 
of the set of all lifts $h_{<p}$ of $h$ and the set of pairs 
$(c,A)\in\c L_{\c K}\times\Aut\,\c L$
such that 
\begin{equation} \label{E2.1}
 (\id _{\c L}\otimes h)e\circ c=\sigma c\circ (A\otimes\id _{\c K})e\,.
\end{equation}
\end{Prop}

\begin{proof} If $\Pi (h_{<p})=(c,A)$ then 
$$(\id _{\c L}\otimes h)e\circ (\id _{\c L}\otimes h_{<p})f
=(\id _{\c L}\otimes h_{<p})(e\circ f)=
(\id _{\c L}\otimes h_{<p})\sigma f=$$
$$\sigma c\circ (A\otimes\id _{\c K_{<p}})\sigma f=
\sigma c\circ (A\otimes\id _{\c K})e\circ (A\otimes\id _{\c K_{<p}})f$$
$$=\sigma c\circ (A\otimes\id _{\c K})e\circ (-c)\circ (\id _{\c L}\otimes h_{<p})f\, .$$
This proves that $(c,A)$ satisfies identity \eqref{E2.1}. 

Let $l'\in\c L$. Then $\eta _0^{-1}(l')\in\Gal (\c K_{<p}/\c K)$ and 
$h_{<p}\,\eta _0^{-1}(l')$ is again a lift of $h$ to $\c K_{<p}$. 
Therefore, we have a transitive action 
$h_{<p}\mapsto h_{<p}*l':=h_{<p}\eta _0^{-1}(l')$ of 
$G(\c L)$ on the set of all lifts $h_{<p}$. 

At the same time, if $(c,A)$ satisfies \eqref{E2.1} then the new couple 
\linebreak 
$(c,A)*l':=(c\circ (l'\otimes 1),(\Ad l')A)$ is again a solution of $\eqref{E2.1}$. 
Indeed, 
$$(\id _{\c L}\otimes h)e\circ c\circ (l'\otimes 1)=
(\sigma c)\circ (A\otimes \id _{\c K})e\circ (l'\otimes 1)$$
$$=\sigma (c\circ (l'\otimes 1))\circ (-l'\otimes 1)
\circ (A\otimes\id _{\c K})e\circ (l'\otimes 1)\, ,$$
and  
$(-l'\otimes 1)\circ (A\otimes\id _{\c K})\circ (l'\otimes 1)$ acts on $\c L_{\c K}$ as 
$(\Ad l')A\otimes\id _{\c K}$, i.e. $\Ad (l'\otimes 1):\c L_{\c K}\To \c L_{\c K}$ is 
$\c K$-linear. (Indeed, one of most known properties of Campbell-Hausdorff formula, cf. 
\cite{Bou}, Ch.II, Section 6.5, gives that 
$$(-l'\otimes 1)\circ l\circ (l'\otimes 1)=
\sum _{0\leqslant i<p}
[\dots [l,\underset{i\text{  times}}
{\underbrace{l'\otimes 1],\dots ,l'\otimes 1}}]/i!$$ 
depends linearly on $l\in\c L_{\c K}$. ) 

 This defines the action $(c,A)\mapsto (c,A)*l'$ of $G(\c L)$ 
on all solutions $(c,A)$ of \eqref{E2.1}. Verify that 
the map $\Pi $ is compatible with above defined $G(\c L)$-actions. Indeed, 
if $\Pi (h_{<p})=(c,A)$ then $h_{<p}*l'$ sends $f$ to 
$$h_{<p}(f\circ (l'\otimes 1))=c\circ (A\otimes\id _{\c K_{<p}})f\circ (l'\otimes 1)
=$$
$$(c\circ (l'\otimes 1))\circ (-l'\otimes 1)
\circ (A\otimes \id _{K_{<p}})f\circ (l'\otimes 1)$$ 
and therefore, $\Pi (h_{<p}*l')=(c,A)*l'$. 
So, our proposition will be proved if we show 
that $G(\c L)$ acts transitively on the set of all solutions $(c,A)$ of \eqref{E2.1}. 

Suppose $(c,A)$ and $(c',A')$ are solutions of \eqref{E2.1}. Then the existence 
of $l'\in G(\c L)$ such that $(c',A')=(c,A)*l'$ will be implied by the following lemma.

\begin{Lem} \label{L2.4}
 For any $1\leqslant s\leqslant p$, there is $l'_s\in G(\c L)$ 
such that if $(c'_s, A'_s)=(c,A)*l'_s$ 
then $c_s\equiv c'\,\op{mod}\,C_s(\c L_{\c K})$ and $A_s\equiv A'\,\op{mod}\,C_s(\c L)$. 
\end{Lem} 

\begin{proof} [Proof of lemma] Use induction on $s$. 

If $s=1$ there is nothing to prove. 

Suppose lemma is proved for some $1\leqslant s<p$. 

Let $c'=c'_s+\delta $ and $A'=A'_s+\c A$, where $\delta \in C_s(\c L_{\c K})$ 
and $\c A\in\Hom _{\F _p\op{-mod}}(\c L, C_s(\c L))$.  Then 
we have modulo $C_{s+1}(\c L_{\c K})$: 
$$(\id _{\c L}\otimes h)e\circ c'\equiv (\id _{\c L}\otimes h)e\circ c'_s+\delta \, ,$$
$$(\sigma c')\circ (A'\otimes\id _{\c K})e\equiv (\sigma c'_s)\circ (A'_s\otimes\id _{\c K})e+
\sigma (\delta )+(\c A\otimes\id _{\c K})e\, .$$

Because  $(c'_s,A'_s)$ and $(c',A')$ are solutions of \eqref{E2.1} we obtain 
$$\sigma\delta -\delta + \sum _{a\in\Z ^+(p)}
t^{-a}\c A_k(D_{a0})+\alpha _0\c A(D_0)\in C_{s+1}(\c L_{\c K})\, ,$$
where $\c A_k=\c A\otimes k\in\Hom _{k-\op{mod}}(\c L_k, C_{s}(\c L_k))$. 
Now Lemma \ref{L2.2}b) (cf. also remark b) after that lemma) 
implies that $\delta \equiv \delta _0\,\op{mod}\,C_{s+1}(\c L_{\c K})$, 
where $\delta _0\in C_s(\c L)\otimes 1$, 
all $\c A_k(D_{a0})\in C_{s+1}(\c L_k)$ and $\c A(D_0)\in C_{s+1}(\c L)$. 
Therefore, modulo $C_{s+1}(\c L_k)$ the automorphisms 
$A'$ and $A'_s$ coincide on generators of $\c L_k$ 
(use that $\c A_k(D_{an})=\sigma ^n\c A_k(D_{a0})$ for all $n\in\Z /N_0$) and 
$A'\equiv A_s'\,\op{mod}\,C_{s+1}(\c L)$. 

So, for 
$(c,A)*(l'_s\circ \delta )=
(c'_s,A'_s)*\delta =(c'_{s+1}, A'_{s+1})$, we have that 
$$c'_{s+1}=c'_s\circ\delta\equiv c'_s+\delta \equiv c'\,\op{mod}\,C_{s+1}(\c L_{\c K})$$ 
and 
$$A'_{s+1}=(\Ad\, \delta )A'_s\equiv (\Ad\, \delta )A'\equiv A'\,\op{mod}\,C_{s+1}(\c L)\,.$$
The lemma and Proposition \ref{P2.3} are completely proved.  
\end{proof} 
\end{proof}

\begin{remark}
 Suppose  $(c_1,A_1)$ and $(c_2,A_2)$ satisfy the identity \eqref{E2.1} and 
 $c_1\equiv c_2\,\op{mod}\,C_s(\c L_{\c K})$. Then $(A_1\otimes\id _{\c K})e\equiv 
 (A_2\otimes\id _{\c K})e\,\op{mod}\,C_s(\c L_{\c K})$ and this implies that 
 $A_1\equiv A_2\,\op{mod}\,C_s(\c L)$. In particular, 
 if $\Pi (h_{<p})=(c,A)$ then the restriction $h_{<s}$ of $h_{<p}$ 
 to $\c K_{<p}^{C_s(\c L)}$ is uniquely 
determined by the residue $c\,\op{mod}\,C_s(\c L_{\c K})$. 
 Now from the proof of the above proposition it follows that all lifts of a given 
$h_{<s}$ to 
 automorphisms $h_{<s+1}$ of $\c K_{<p}^{C_{s+1}(\c L)}$ 
are uniquely determined by the residues 
 $(c+\delta, A)\,\op{mod}\,C_{s+1}(\c L_{\c K})$, 
 where $\delta\in C_s(\c L)$. 
 \end{remark}
\medskip 

Using the above proposition and operators $\c R$ and $\c S$ from Subsection \ref{S2.2} 
we can specify a unique choice $h^0_{<p}$ in the set of all lifts of $h$ 
by specifying a unique solution $(c^0,A^0)$ of \eqref{E2.1} as follows. 

Suppose $1\leqslant s<p$ and we have chosen 
$(c_s,A_s)\in \c L_{\c K}\times \Aut \c L$ such that 
the identity \eqref{E2.1} holds modulo $C_s(\c L_{\c K})$. If $s=1$ we just choose 
$c_1=0$ and $A_1=\id _{\c L}$. Then we can find the solution 
$(c_{s+1},A_{s+1})\in\c L_{\c K}\times\Aut\,\c L$  
of \eqref{E2.1} modulo $C_{s+1}(\c L_{\c K})$ 
by setting $c_{s+1}=c_s+X_s$ and $A_{s+1}=A_s+B_s$ where $X_s\in C_s(\c L_{\c K})$ and 
$B_s\in\Hom _{\F _p-\op{mod}}(\c L,C_s(\c L))$ must satisfy the relation 

\begin{equation} \label{E2.2}\sigma X_s-X_s+\sum _{a\in\Z ^0(p)}t^{-a}B_s(D_{a0})\equiv 
 \end{equation}

$$(\id _{\c L}\otimes h)e\circ c_s-
\sigma c_s\circ (A_s\otimes\id _{\c K})e\,\op{mod}\,C_{s+1}(\c L_{\c K})\,.$$

 By Lemma \ref{L2.2}b) the recurrence relation \eqref{E2.2} uniquely determines  
the elements $B_s(D_{a0})\,\op{mod}\,C_{s+1}(\c L_k)$ but the element 
$X_s$ is determined only up to elements of $C_s(\c L)\,\op{mod}\,C_{s+1}(\c L)$.  
(This will affect the right-hand 
side of \eqref{E2.2} at the next $(s+1)$-th step and so on.) 
Note that the knowledge of the elements 
$B_s(D_{a0})\,\op{mod}\,C_{s+1}(\c L_k)$ determines uniquely 
the automorphism $A_{s+1}$ modulo $C_{s+1}(\c L)$ because for 
all $n\in\Z /N_0$, $A_{s+1}(D_{an})=\sigma ^nA_{s+1}(D_{a0})$. By Proposition \ref{P2.3} 
all solutions $X_s$ correspond to 
different extensions of a given automorphism of $\c K_{<p}^{C_s(\c L)}$ 
to an automorphism of $\c K_{<p}^{C_{s+1}(\c L)}$ (cf. also the remark after the proof of 
that proposition).  
In particular, we can uniquely specify the lift $h^0_{<p}$ by 
specifying $(\id _{\c L}\otimes h^0_{<p})f$ if 
we take at each $s$-th step the solutions of \eqref{E2.2} 
in the form 
$\sum _{a\in\Z ^0(p)}t^{-a}B_s(D_{a0})
=\c R(\c B_s)$ and $X_s=\c S(\c B_s)$, 
where $\c B_s$ is the RHS in \eqref{E2.2}. As a result, the pair $(c^0,A^0):=(c_p,A_p)$ 
satisfies the identity \eqref{E2.1} and defines the lift $h^0_{<p}$. 
\medskip

\begin{remark}  It is not easy to control 
the lifts $h_{<p}$ because condition \eqref{E2.2} contains  
highly non-trivial the Campbell-Hausdorff operation $\circ $. In Section \ref{S3} 
we resolve this problem by introducing the procedure of linearization. 
\end{remark}
\medskip

\subsection{The group $\wt{\c G}_{h}$} \label{S2.4} 

Denote by $\wt{\c G}_{h}$ the group of 
all lifts $\tilde h_{<p}\in\op{Aut}\,\c K_{<p}$ of the elements $\tilde h$  
of the closed subgroup 
in $\Aut\,\c K$ generated by $h$.

Use the identification $\eta _0$ from Subsection \ref{S1.3} to 
obtain a natural 
short exact sequence of profinite $p$-groups 
\begin{equation} \label{E2.3}
1\To G(\c L)\To \wt{\c G}_{h}\To \langle h \rangle\To 1 
\end{equation} 
For any $s\geqslant 2$, $C_s(\wt{\c G}_{h})$ is a subgroup in 
$G(\c L)$ and, therefore, $\c L_{h}(s):=C_s(\wt{\c G}_{h})$ 
is a Lie subalgebra of $\c L$. Set $\c L_{h}(1)=\c L$. Note that  
for any $s_1,s_2\geqslant 1$, 
we have 
$[\c L_{h}(s_1),\c L_{h}(s_2)]\subset \c L_{h}(s_1+s_2)$. 
\medskip 

Define the weight filtration $\c L(s)$, $s\in\N $, in $\c L$ by setting 
 $\w (D_{an})=s$ 
if $(s-1)c_0\leqslant a<sc_0$. With this notation 
$\c L(s)_k$ is generated  over $k$ 
by all $[\dots [D_{a_1n_1},D_{a_2n_2}],\dots ,D_{a_rn_r}]$ such that 
$\sum _{i}\w (D_{a_in_i})\geqslant s$.  
For any $s_1,s_2\geqslant 1$, we also have that 
$[\c L(s_1),\c L(s_2)]\subset\c L(s_1+s_2)$.

\begin{Thm} \label{T2.5} 
 For all $s\in\N $, $\c L_{h}(s)=\c L(s)$. 
\end{Thm}

\begin{proof} Let $h^0_{<p}$ be the lift constructed at the end of Subsection \ref{S2.3}. 
Then $h^0_{<p}\in\wt{\c G}_h$ is a preimage of $h$  
in short exact sequence \eqref{E2.3}. 

Let $\c L^{lin}=(\sum _{a,n}kD_{an})|_{\sigma =\id }$ be 
\lq\lq the subspace of linear terms\rq\rq\  of $\c L$. 
We have the following properties:
\medskip 

$\bullet $\ $\c L(s+1)=\c L^{lin}\cap \c L(s+1)
+\c L(s+1)\cap C_2(\c L)$; 
\medskip 

$\bullet $\ $\c L(s+1)\cap C_2(\c L)=\sum _{s_1+s_2=s+1} \left [\c L(s_1),\c L(s_2)\right ]$;
\medskip 

$\bullet $\ $\c L_h(s+1)$ is the ideal in $\c L$ generated by  
$[\c L_h(s),\c L]$ and the elements of the form 
$(\op{Ad}\,h^0_{<p})l\circ (-l)$, where $l\in \c L_h(s)$.
\medskip  
 
Let $(\op{Ad}\,h^0_{<p})D_0=\wt{D}_0$ and for all $a\in\Z ^+(p)$,  
$(\op{Ad}\,h^0_{<p})D_{a0}=\wt{D}_{a0}$.

\begin{Lem} \label{L2.6} 
 We have: 
\medskip 

{\rm a)}\ $\wt{D}_0\equiv D_0\,\op{mod}\,(\c L(3)+\c L(2)\cap C_2(\c L))$;
\medskip  
 
{\rm b)}\ if $a\in\Z ^+(p)$ and $\op{wt}(D_{an})=s$ then 
$$\wt{D}_{a0}\equiv D_{a0}-\sum _{i\geqslant 0}\alpha _i(h)aD_{a+c_0+pi,0}
\,\op{mod}\,(\c L(s+2)_k+\c L(s+1)_k\cap C_2(\c L_k))\,,$$
where $\alpha _i(h)\in k$ are such that 
$h(t)= t(1+\sum _{i\geqslant 0}\alpha _i(h)t^{c_0+pi})$. 
\end{Lem} 

We prove this Lemma below after finishing the proof of Theorem \ref{T2.5}. Clearly, 
Lemma \ref{L2.6} has the following corollaries: 
\medskip 

(c1)\ {\it if $l\in \c L(s)$ then $(\op{Ad}\,h^0_{<p})l\circ (-l)\in \c L(s+1)$};
\medskip 

(c2)\ {\it if $l\in \c L^{lin}\cap\c L(s+1)$ then there is an 
$l'\in \c L^{lin}\cap\c L(s)$ such that} 
$\op{Ad}\,h^0_{<p}(l')\circ (-l')\equiv l\,\op{mod}\,
\c L(s+1)\cap C_2(\c L)$ (use that $\alpha _0(h)\ne 0$). 
\medskip 

Prove theorem by induction on $s\geqslant 1$.  

Clearly, $\c L_h(1)=\c L(1)$. 

Suppose $s_0\geqslant 1$ and for $1\leqslant s\leqslant s_0$, 
$\c L_h(s)=\c L(s)$. 

Then $[\c L_h(s_0),\c L]=[\c L(s_0), \c L(1)]\subset \c L(s_0+1)$ 
and applying (c1) we obtain that $\c L_h(s_0+1)\subset \c L(s_0+1)$. 

In the opposite direction, note that by inductive assumption, 

$$\c L(s_0+1)\cap C_2(\c L)=\sum _{s_1+s_2=s_0+1}
\left [\c L_h(s_1),\c L_h(s_2)\right ]\subset \c L_h(s_0+1)$$
and then from (c2) we obtain that $\c L^{lin}\cap\c L(s_0+1)
\subset \c L_h(s_0+1)$. So, $\c L(s_0+1)\subset\c L_h(s_0+1)$ and  
Theorem \ref{T2.5} is completely proved. 
\end{proof}

\begin{proof} [Proof of Lemma \ref{L2.6}] 

Let 
$$\c N=\sum_{s\geqslant 1}t^{-c_0s}\c L(s)_{\m },$$
where $\m $ is the maximal ideal of the valuation ring $O_{\c K}$ of $\c K$. 
Clearly, $\c N$ has the structure of Lie algebra over $\F _p$.

Let 
$$\tilde e:=(\op{Ad}\,h^0_{<p} \otimes\id _{\c K})e
=\sum _{a\in\Z ^+(p)}t^{-a}
\wt{D}_{a0}+\alpha _{0}\wt{D}_{0}\, .$$
Then recovering $\tilde e$ from the following relation 
\begin{equation} \label{E2.4} (\id _{\c L}\otimes h)e\circ c^0=(\sigma c^0)\circ \tilde{e}\,,
\end{equation}
 where $c^0\in G(\c L_{\c K})$,  is a part of the procedure of 
specifying of the lift $h^0_{<p}$ described at the end of Subsection 
\ref{S2.3}, i.e. $\tilde e=(A^0\otimes\id _{\c K})e$.

Now note that $e\in\c N$ and the operators $\c R$ and $\c S$ 
map $\c N$ to itself. Therefore, when following the procedure 
of specifying $h^0_{<p}$ at each step we obtain  that 
$\c B_s, \c R(\c B_s), \c S(\c B_s)\in\c N$ and, therefore, 
$\tilde e, c^0, \sigma c^0\in\c N$. 

For any $i\geqslant 0$, introduce the ideals $\c N(i):=t^{c_0i}\c N$ of $\c N$. 
Note that for all $i\geqslant 0$, 
the operators $\c R$ and $\c S$ map $\c N(i)$ to itself.

Consider the following properties: 
\medskip 

a)\  $(\id_{\c L}\otimes h)e=e+e_1\,\op{mod}\,\c N(2)$, 
where $e_1=e_1^++e_1^-\in\c N(1)$  
with 
$$e_1^-=-\sum _{\substack{i\geqslant 0\\a\in\Z ^+(p)}}t^{-a}a 
\alpha _i(h)D_{a+c_0+pi,0},\ \ e_1^+=-\sum _{\substack{i\geqslant 0\\0<a<c_0+pi}}
a\alpha _i(h)t^{-a+c_0+pi}D_{a0}\, $$
(note that $e_1^+\in\c L_{\m }$ and, therefore, $\c R(e_1^+)=0$);
\medskip 

b)\   the congruence $(\id _{\c L}\otimes h)e\equiv e\,\op{mod}\,\c N(1)$ implies 
that $\tilde{e}\equiv e\,\op{mod}\,\c N(1)$ and 
$c^0, \sigma c^0\in \c N(1)$: indeed, in the procedure of 
specifying of $h^0_{<p}$ we have for all $s$, that $c_s,\sigma c_s\in\c N(1)$ and 
$(A_s\otimes\id _{\c K})e\equiv e\,\op{mod}\,\c N(1)$; 
\medskip 

c)\ \ \  
 $\tilde e=(-\sigma c^0)\circ (\id _{\c L}\otimes h)e\circ c^0
\equiv (c^0-\sigma c^0)+e+e_1\,\op{mod}\,
\c N(2)+t^{c_0}\wt{\c N}^{(2)},$  
where 
$\wt{\c N}^{(2)}:=\sum _{s\geqslant 2}t^{-sc_0}(\c L(s)\cap C_2(\c L))_{\m }$ 
(use that $[\c N(1),\c N(1)]\subset\c N(2)$ and 
$[\c N(1),\c N]\subset t^{c_0}\wt{\c N}^{(2)}$);
\medskip 

d)\  
$\c R(\c N(2)+t^{c_0}\wt{\c N}^{(2)})\subset \c N(2)+t^{c_0}\wt{\c N}^{(2)}$, 
$\c R(\tilde e-e-e_1^-)=\tilde e-e-e_1^-$, 
$\c R(c^0-\sigma c^0+e_1^+)=0$ and, therefore, c) implies that 
$$\tilde e\equiv e+e_1^-\,\op{mod}\,\c N(2)+t^{c_0}\wt{\c N}^{(2)}$$  
or, more explicitly,   
$$\tilde{e}\equiv {\hskip-6pt}
\sum _{a\in\Z ^+(p)}t^{-a}{\hskip-4pt}\left (D_{a0}-
a\sum _{i\geqslant 0}\alpha _i(h)D_{a+c_0+pi,0}\right )+\alpha _0D_0\,
\op{mod}\,\c N(2)+t^{c_0}\wt{\c N}^{(2)}\,.$$

It remains to prove that this congruence is equivalent to the statement of our lemma.
Note that any element $l\in\c L_{\c K}$ can be uniquely presented as 
$l=\sum _{b\in\Z }t^bl_b$, where all $l_b\in\c L_k$ and $l_b\to 0$ if $b\to -\infty $. 

Suppose $s\geqslant 1$ and $-(s-1)c_0\geqslant b>-sc_0$.

Then it follows directly from definitions that: 
\medskip 

--- \ if $l\in\c N$ then $l_b\in\c L(s)_k$;
\medskip 

---\ if $l\in\c N(2)$ then $l_b\in\c L(s+2)_k$;
\medskip 

--- \ if $l\in t^{c_0}\wt{\c N}^{(2)}$ then $l_b\in \c L(s+1)_k\cap C_2(\c L_k)$. 
\medskip 

It remains to compare the coefficients in the last congruence for $\tilde e$. 
\end{proof}
\medskip

 \subsection{The group $\c G_{h}$} \label{S2.5} 

Let $\c G_{h}=\wt{\c G}_{h}/\wt{\c G}^{p}_{h}C_p(\wt{\c G}_{h})$. 

\begin{Prop} \label{P2.7} 
Exact sequence \eqref{E2.3} induces the following 
exact sequence of $p$-groups 
\begin{equation} \label{E2.5} 
1\To G(\c L)/G(\c L(p))\To \c G_{h}\To \langle h\rangle \,\op{mod}\,
\langle h^{p}\rangle \To 1
\end{equation}
 \end{Prop}
 
 \begin{proof}  Set 

 $$\c M:=\c N+\c L(p)_{\c K}=\sum _{1\leqslant s<p}
 t^{-sc_0}\c L(s)_{\m }+\c L(p)_{\c K}$$

 $$\c M_{<p}:=\sum _{1\leqslant s<p}
 t^{-sc_0}\c L(s)_{\m _{<p}}+\c L(p)_{\c K_{<p}}$$  
 where $\m _{<p}$  is the maximal ideal 
of the valuation ring of $\c K_{<p}$.  
 
 Then $\c M$ has the induced structure of a Lie 
$\F _p$-algebra (use the Lie bracket from $\c L_{\c K}$) and  
for $i\geqslant 0$, $\c M(i):=t^{ic_0}\c M$ is a decreasing 
filtration of ideals in $\c M$. Note that $e\in\c M$. 
 
Similarly, 
 $\c M_{<p}$ is a Lie $\F _p$-algebra (containing $\c M$ as its subalgebra) and  
for $i\geqslant 0$, $\c M_{<p}(i):=t^{ic_0}\c M_{<p}$ is a decreasing 
filtration of ideals in $\c M_{<p}$,     
$\c M_{<p}(i)\cap \c M=\c M(i)$. 

We have a natural embedding of $\bar{\c M}:=\c M/\c M(p-1)$ 
into $\bar{\c M}_{<p}:=\c M_{<p}/\c M_{<p}(p-1)$, and the induced decreasing filtrations 
of ideals $\bar{\c M}(i)$ and $\bar{\c M}_{<p}(i)$ 
(where $\bar{\c M}(p-1)=\bar{\c M}_{<p}(p-1)=0$)  
are compatible with this embedding.  

Note that for all $i\geqslant 0$, we have also  
$(\id _{\c L}\otimes h-\id _{\c M})^i\c M\subset \c M(i)$.

\begin{Lem} \label{L2.8} 
 $f, \sigma f\in\c M_{<p}$. 
\end{Lem}

\begin{proof} Prove by induction  on $1\leqslant s\leqslant p$ that 
$f,\sigma f\in\c M_{<p}+\c L(s)_{\c K_{<p}}$.  

If $s=1$ then $f\in\c L_{\c K_{<p}}=\c M_{<p}+\c L(1)_{\c K_{<p}}$. 

Suppose $1\leqslant s_0<p$ and $f, \sigma f\in\c M_{<p}+\c L(s_0)_{\c K_{<p}}$. 

For $1\leqslant s\leqslant s_0+1$ let $j_s=\op{rk}_{\F _p}(\c L/\c L(s))$. Then 
$0=j_1<j_2<\dots <j_{s_0+1}$. Let $l_1,\dots ,l_{j_{s_0+1}}\in\c L$ be such that 
for all $1\leqslant s\leqslant s_0+1$, $l_{j_s+1},\dots ,l_{j_{s_0+1}}$ give an $\F _p$-basis of 
$\c L(s)$ modulo $\c L(s_0+1)$. This means that for all such $s$, 
the elements $l_{j_s+1},\dots ,l_{j_{s+1}}$ form  $\F _p$-basis of $\c L(s)$ modulo $\c L(s+1)$. 

With above notation for $1\leqslant j\leqslant j_{s_0+1}$, 
there are unique $b_j\in\c K_{<p}$ such that $f\equiv \sum _jb_jl_j\,\op{mod}\,\c L(s_0+1)_{\c K_{<p}}$. 
By inductive assumption, 
if $s<s_0$ and $l_j\in\c L(s)\setminus \c L(s+1)$ then $b_j, \sigma b_j\in \m _{<p}t^{-c_0s}$ and 
we must prove that if $l_j\in \c L(s_0)$ then $b_j\in\m _{<p}t^{-c_0s_0}$.

Let $e\circ f=e+f+X(f,e)$. Then  
$X(f,e)\in\,\c M_{<p}+\c L(s_0+1)_{\c K_{<p}}$ (use that $e\in\c M_{<p}$ and 
$[\c M_{<p},\c L(s_0)_{\c K_{<p}}]\subset\c L(s_0+1)_{\c K_{<p}}$)  and, therefore,  
$\sigma f-f\in\c M_{<p}+\c L(s_0+1)_{\c K_{<p}}$. 

Thus, $\sigma f-f\equiv \sum _ja_jl_j$, where for all $s\leqslant s_0$ and 
$j_s<j\leqslant j_{s+1}$, we have $a_j\in\m _{<p}t^{-c_0s}$. In particular, 
for the indices $j_{s_0}<j\leqslant j_{s_0+1}$, we have  
$\sigma b_j-b_j\in\m _{<p}t^{-c_0s_0}$. 
Therefore, 
$$\sigma (b_jt^{c_0s_0/p})-t^{c_0s_0(1-1/p)}(b_jt^{c_0s_0/p})\in\m _{<p}\,,$$
and this implies that $b_jt^{c_0s_0/p}\in\m _{<p}$ and $\sigma b_j,b_j\in\m _{<p}t^{-c_0s_0}$. 
Lemma \ref{L2.8} is proved. 
 \end{proof}

 Consider the orbit of $\bar f:=f\,\op{mod}
 \,\c M_{<p}(p-1)$ with 
 respect to the natural action of $\wt{\c G}_{h}
\subset \Aut \,\c K_{<p}$ 
 on $\bar{\c M}_{<p}$. Prove that the stabilizer $\c H$ of $\bar f$ equals 
 $\wt{\c G}_{h}^{p}C_p(\wt{\c G}_{h})$.

 If $l\in G(\c L)$ then the corresponding element $\eta _0^{-1}(l)\in\c G_{<p}$ 
sends $f$ to $f\circ l$. 
This means that if 
$l\in\c H\cap G(\c L)$ then (use that $\c M(p-1)\subset\c L_{\m }+\c L(p)_{\c K}$) 
$$l\in \c M_{<p}(p-1)\cap \c L=\c M (p-1)\cap \c L
=\c L(p)_{\c K}\cap \c L=\c L(p)= C_p(\wt{\c G}_h)\, .$$ 
Therefore, 
$\c H\cap G(\c L)=C_p(\wt{\c G}_{h})\subset\c H$ and we have the induced embedding 
$\kappa :G(\c L)/G(\c L(p))
\To \wt{\c G}_h/\c H$. 

 Note that $\wt{\c G}_{h}^{p}\,\op{mod}\,C_p(\wt{\c G}_{h})$ 
 is generated by $h _{<p}^{0p}$ (as earlier, $h^0_{<p}$ is the lift 
 chosen in the end of Subsection \ref{S2.3}). This follows from the fact 
 that any finite $p$-group of nilpotent class 
 $<p$ is $P$-regular, cf. \cite{Ha} Subsections 12.3-12.4. In particular, 
for any $g\in G(\c L)$, 
$$(h^0_{<p}\circ g)^p\equiv h_{<p}^{0p}\circ g'\,\op{mod}\,C_p(\wt{\c G}_h)\,,$$ 
where $g'$ is the product of $p$-th powers of elements from 
$G(\c L)$, but $G(\c L)$ has period $p$.

 Recall that 
 $(\id _{\c L}\otimes h^0_{<p})f=c^0\circ (A^0 \otimes\id _{\c K_{<p}})f$   
with $c^0\in \c N(1)$, cf. Subsection \ref{S2.4},  and 
 $A^0=\op{Ad}\,(h^0 _{<p})$. Then   
 $h _{<p}^{0p}(f)$ is equal to 
$$(\id _{\c L}\otimes h) ^{p-1}\left (c^0\circ (A^0\otimes h^{-1})c^0\circ 
 \dots \circ (A^0\otimes h^{-1})^{p-1}c^0\right )\circ (A^{0p}\otimes\id _{\c K_{<p}})f\, .$$

 Note that  if $l\in\c L(s)$ then $A^0(l)\equiv l\,\op{mod}\,\c L(s+1)$. 
This implies that   
$(A^0-\id _{\c L})^{p}\c L\subset\c L(p)$ and, therefore, 
 $(A^{0p}\otimes\id _{\c K_{<p}})\bar f=\bar f$. 

For similar reasons we have for any $i$, that 
$(A^0\otimes\id _{\c K}-\id _{\c N})\c N(i)\subset\c N(i+1)$. 
At the same time, $h(t)\equiv t\,\op{mod}\,t^{1+c_0}$ implies that 
for any $n\in\c N(i)$, $(\id _{\c L}\otimes h^{-1})n\equiv n\,\op{mod}\,\c N(i+1)$. 
This implies that $B=A^0\otimes h^{-1}$ is an automorphism of 
the Lie $\F _p$-algebra $\c N$ and 
for all $i\geqslant 0$, 
 $(B-\id _{\c N})\c N(i)\subset 
 \c N(i+1)$.  

\begin{Lem} \label{L2.9} 
 For any $m\in\c N(1)$, $m\circ B(m)\circ\dots \circ B^{p-1}m\in\c N(p)$. 
\end{Lem}

\begin{proof} Consider the Lie algebra $\mathfrak{M}=\c N(1)/\c N(p)$ with 
the filtration $\{\mathfrak{M}(i)\}_{i\geqslant 1}$ induced by the filtration 
$\{\c N(i)\}_{i\geqslant 1}$. 
This filtration is central, i.e. for any $i,j\geqslant 1$, 
$[\mathfrak{M}(i), \mathfrak{M}(j)]\subset\mathfrak{M}(i+j)$. 
In particular, the nilpotent class of $\mathfrak{M}$ is $<p$. 

The operator $B$ induces the operator on $\mathfrak{M}$ which we denote also by $B$. 
Clearly, $B=\wt{\exp}\c B$, where $\wt{\exp}$ is the truncated exponential 
(cf. Subsection \ref{S2.1}) and $\c B$ is a differentiation on $\mathfrak{M}$ such that 
for all $i\geqslant 1$, $\c B(\mathfrak{M}(i))\subset \mathfrak{M}(i+1)$. 

Let $\wt{\mathfrak M}$ be a semi-direct product of $\mathfrak{M}$ and the trivial 
Lie algebra $\F _pw$ via $\c B$. This means that 
$\wt{\fr{M}}=\fr{M}\oplus\F _pw$ as $\F _p$-module, 
$\mathfrak{M}$ and $\F _pw$ are Lie subalgebras of 
$\wt{\mathfrak{M}}$ and for any $m\in\mathfrak{M}$, $[m,w]=\c B(m)$. 
Clearly, 
$C_2(\wt{\mathfrak{M}})=[\wt{\mathfrak{M}},\wt{\mathfrak{M}}]\subset \mathfrak{M}(2)$.  
This implies that   $\wt{\mathfrak{M}}$ 
has nilpotent class $<p$ and we can consider the 
$p$-group $G(\wt{\mathfrak{M}})$. This group has nilpotent class $<p$ and period $p$ 
(because for any $\bar m\in\wt{\mathfrak{M}}$, its $p$-th power in $G(\wt{\mathfrak{M}})$ 
equals $p\bar m=0$). 

Note that the conjugation by $w$ in $G({\mathfrak{M}})$ is given by the automorphism 
$\wt{\exp}\c B=B$. Indeed, 
if $m\in\mathfrak{M}$ then 
$$B(m)=(\wt{\exp}\c B)m=\sum _{0\leqslant n<p}\c B^n(m)/n!=(-w)\circ m\circ w,$$
cf. the reference to \cite{Bou} in the proof of Proposition \ref{P2.3}.

In particular, for any element   
$\bar m=m\,\op{mod}\,\c N(p)\in \mathfrak{M}$, 
we have $w_1\circ \bar m=B(\bar m)\circ w_1$, where $w_1=-w$. Therefore, 
$0=(\bar m\circ w_1)^p=\bar m\circ B(\bar m)\circ\dots \circ B^{p-1}(\bar m)\circ w_1^p$, 
and it remains to note that $w_1^p=w^{-p}=0$. 
\end{proof} 

Applying the above Lemma we obtain that 
$$c^0\circ (A^0\otimes h^{-1})c^0\circ 
 \dots \circ (A^0\otimes h^{-1})^{p-1}c^0\in\c N(p)\subset \c M(p-1)$$ 
and, therefore, 
$h^{0p}_{<p}(\bar f)=\bar f$.  

Thus, we proved that $\wt{\c G}^{p}_{h}C_p(\wt{\c G}_{h})\subset\c H$. 
\medskip 
 
Suppose $g=h_{<p}^{0m}l\in\c H$ with some $l\in G(\c L)$. Then we have 
$$g(f)\equiv f\,\op{mod}\,\c M_{<p}(p-1)\,.$$  
This congruence in the Lie algebra $\c M_{<p}$ 
can be replaced by the equivalent congruence 
$g(f)\equiv f\,\op{mod}\,G(\c M_{<p}(p-1))$ in the corresponding $p$-group 
$G(\c M_{<p})$, cf. comments to the equivalence 
$L\mapsto G(L)$ in the beginning of Introduction. Therefore, 
$g(f)=b\circ f$ where $b\in \c M_{<p}(p-1)$. Note that for obvious reasons  
$\sigma (b)\in \c M_{<p}(p-1)$. Then the equality  
$$g(e)\circ b\circ f=g(e)\circ g(f)=g(\sigma f)=
\sigma b\circ \sigma f=\sigma b\circ e\circ f$$ 
implies that $g(e)\equiv e\,\op{mod}\, \c M (p-1)$ and we obtain  
$$(\id\otimes h)^m(e)\equiv e\,\op{mod}\,\c M(p-1)\,.$$

Clearly, $\c L_{\m }+\c L(p)_{\c K}\supset \c M (p-1)$ 
and, therefore, for the element 
$$e_{<p}=\sum _{a\in\Z ^0(p)\cap [0, (p-1)c_0]}t^{-a}D_{a0}$$
we obtain 
$(\id _{\c L}\otimes h^m)e_{<p}\equiv e_{<p}\,\op{mod}\,\c L_{\m }$. 

This means for all $a\in\Z ^0(p)\cap [0, (p-1)c_0]$, $h^m(t^{-a})\equiv t^{-a}\,\op{mod}\,\m $,  
and we obtain that $m\equiv 0\,\op{mod}\,p$ (take e.g. $a=c_0+1$).

 Therefore, $l\in\c H\cap G(\c L)=C_p(\wt{\c G}_h)$ 
and $\c H\subset \wt{\c G}^{p}_{h}C_p(\wt{\c G}_{h})$. 

Finally, $\wt{\c G}_h/\c H=\c G_h$ and it remains to note 
that $\c H\,\op{mod}\,C_p(\wt{\c G}_h)=\langle h_{<p}^{0p}\rangle $ and, therefore,  
$\op{Coker}\kappa =\langle h\rangle \,\op{mod}\,\langle h^p\rangle $.  
\end{proof} 

\begin{Cor} \label{C2.10} 
If $L_h$ is a Lie algebra over $\F _p$ such that $\c G_{h}=G(L_h)$  
then 
\eqref{E2.5} induces 
the following short exact sequence of Lie $\F _p$-algebras 
$$0\To \bar{\c L}\To L_h\To \F _ph\To 0\, ,$$
where, as earlier, $\bar{\c L}=\c L/\c L(p)$. 
\end{Cor} 
\medskip

 \subsection{Ramification estimates} \label{S2.6} 
Use the identification $\eta _0:\c G_{<p}\simeq G(\c L)$ from Subsection \ref{S1.3}   
and set 
for $s\in\N $, $\c K [s]:=\c K_{<p}^{G(\c L(s+1))}$. 
Note that $\c K [s]/\c K$ is Galois and its Galois group is $G(\c L/\c L(s+1))$. 

Denote by 
$v[s]$ the maximal upper ramification number of the extension 
$\c K[s]/\c K$. In other words, 
$$v[s]=\max\{v\ |\ \c G^{(v)} \text{ acts non-trivially 
on }\c K[s]\}\,.$$

\begin{Prop} \label{P2.11}
For all $s\in\N $, $v[s]=c_0s-1$.  
\end{Prop}

\begin{proof}
 Recall that for any $v\geqslant 0$, 
$\pi _{f}(e)(\c G^{(v)})=\c L^{(v)}$ and for a sufficiently large 
$N$, the ideal 
 $\c L^{(v)}_k$ is generated by all $\sigma ^n\c F^0_{\gamma ,-N}$, 
 where $\gamma\geqslant v$, $n\in\Z $ and the elements 
 $\c F^0_{\gamma ,-N}$ are given in Subsection \ref{S1.4}. 
 
 Note that $\c L_k^{(v)}$ is contained in the ideal generated by 
 the monomials 
$\sigma ^n [\dots [D_{a_1n_1}, D_{a_2n_2}],\dots ,D_{a_rn_r}]$ 
 such that $\max\{n_1,\dots ,n_r\}=0$ and $a_1p^{n_1}+\dots +a_rp^{n_r}
 \geqslant v$. So, 
 $$v\leqslant a_1+\dots +a_r\leqslant 
 c_0\w ([\dots [D_{a_1n_1},
D_{a_2n_2}],\dots ,D_{a_rn_r}])-r\, .$$

 If $v>c_0s-1$ then $\w ([\dots [D_{a_1n_1},
D_{a_2n_2}],\dots D_{a_rn_r}])>s+(r-1)/c_0$ implies that all such 
monomials have  weight $\geqslant s+1$ and, therefore, 
$\c L^{(v)}\subset\c L(s+1)$. 
 
 If $v=c_0s-1$ then $\w ([\dots [D_{a_1n_1}, D_{a_2n_2}],
\dots D_{a_rn_r}])\leqslant s$ 
 iff $r=1$ and the only non-zero $a_i$ equals $c_0s-1$. 
 Therefore, $\c L_k^{(v)}\,\op{mod}\,\c L_k(s+1)$ is generated by 
 the images of all $D_{c_0s-1,n}$  
 and $\c L^{(v)}\not\subset \c L(s+1)$. 
\end{proof}
\medskip

\section{Structure of $L_h$}\label{S3} 

In next Sections we use the notation $h_{<p}$ for arbitrary 
lifts of $h$ to $\c K_{<p}$, in particular, we do not require that $h_{<p}$ coincides with $h^0_{<p}$ 
from the end of Subsection \ref{S2.3}. 
We shall use the notation $\c K(p):=\c K_{<p}^{G(\c L(p))}$ and 
$h(p):=h_{<p}|_{\c K(p)}$. Because $G(\c L(p))=C_p(\wt{\c G}_h)$ the elements 
of $\wt{\c G}_h$ map $\c K(p)$ to itself and we have a natural inclusion $\wt{\c G}_h/G(\c L(p))\subset 
\Aut \c K(p)$. The conjugations $\Ad h(p)$ on $G(\bar{\c L})\subset \wt{\c G}_h/G(\c L(p))$ 
(where $\bar{\c L}=\c L/\c L(p)$) 
can be used to recover the group structure on $\wt{\c G}_h/G(\c L(p))$. We have also 
the induced conjugations (which we still denote by $\Ad h(p)$) 
on $\c G_h=\wt{\c G}_h/\wt{\c G}_h^pG(\c L(p))$ and these conjugations can be used to study the 
structure of the group $\c G_h$ and its Lie algebra $L_h$ fom Corollary \ref{C2.10}.

 The conjugations $\Ad h(p)$  appear as unipotent automorphisms  
 of the Lie algebra $\bar{\c L}$ and we can introduce a differentiation 
$\op{ad}\,h(p)$ of $\bar{\c L}$ by the relation 
$\op{Ad}\,h(p)=\wt{\exp}(\op{ad}\,h(p))$, 
where $\wt{\exp }$ is the truncated exponential, cf. Subsection \ref{S2.1}. 
So, the knowledge of the Lie algebra 
$L_h$ is equivalent to the knowledge of the differentiation 
$\op{ad}\,h(p)$.  The lift $h(p)$ of $h$ can be fully desribed via the 
nilpotent Artin-Schreier theory by using the element 
$f\,\op{mod}\,\c L(p)_{\c K_{<p}}\in \bar{\c L}_{\c K(p)}$. As a matter of fact, 
the identification 
$\Gal (\c K(p)/\c K)\simeq G(\bar{\c L})$ is given by 
the correspondence 
$\tau\mapsto (-\bar f)\circ \tau (\bar f)$, where 
$\bar f=f\,\op{mod}\,\c M_{<p}(p-1)$, and the natural identification  
$\bar{\c L}=\bar{\c M}_{<p}|_{\sigma =\id }$.

\subsection{Interpretation of the action of 
$\id _{\bar{\c L}}\otimes h$ on $\bar{\c M}$}\label{S3.1}

Consider the induced action of $\id _{\bar{\c L}}\otimes h$ on $\bar{\c M}$ 
(and agree to use for this action the same notation). 
Recall that $h(t)=tE(\omega _h^p)$, where we can set 
$$\omega _h^p=\sum _{i\geqslant 0}A_i(h)t^{c_0+pi}$$
with all $A_i(h)\in k$, $A_0(h)\ne 0$, cf. Subsection \ref{S2.1}.  

Let $\c H$ be a linear continuous 
operator on $\c L_{\c K}$ such that for all 
$a\in\Z $ and $l\in\c L_k$, $\c H(t^al)=at^a\omega _h^pl$. 
Then on $\bar{\c M}$ we have $\id _{\bar{\c L}}\otimes h=\wt{\exp}(\c H)$ 
(use that $\c H^p=0$ on $\bar{\c M}$ and $E(X)\equiv \wt{\exp}(X)\,\op{mod}\,\deg p$).  

Set for $0\leqslant i<p$, $h_i:=\c H^i/i!:\bar{\c M}\To \bar{\c M}$ and for $i\geqslant p$, 
$h_i=0$. Then for any 
$j\geqslant 0$, $h_i(\bar{\c M}(j))\subset \bar{\c M}(i+j)$ and for any 
natural $n$, $(\id _{\bar{\c L}}\otimes h)^n=\sum _{i\geqslant 0}n^ih_i$. 
An analogue of these properties appears below when we start studying 
the action of $\id _{\bar{\c L}}\otimes h(p)$ on $\bar f\in\bar{\c M}_{<p}$. 

\subsection{General situation} \label{S3.2} 

The situation from above Subsection \ref{S3.1} can be formalized as follows. 

Suppose $\fr{M}$ is an $\F _p$-module (actually we can 
assume that $\fr{M}$ is a module over any ring 
where $(p-1)!$ is invertible).  
Suppose $g:\fr{M}\To\fr{M}$ is an automorphism of the $\F _p$-module $\fr{M}$ 
such that $g^p=\id _{\fr{M}}$. 
Assume that 
\medskip 

$\bullet $\ for any $m\in\fr{M}$, there are $g_i(m)\in\fr{M}$, where $1\leqslant i<p$, 
such that for all $n\geqslant 0$, $g^n(m)=m+\sum _{1\leqslant i<p}g_i(m)n^i$.  
\medskip 

Set $g_0(m)=m$ and $g_i(m)=0$ if $i\geqslant p$.

\begin{Prop} \label{P3.1}
 With above notation we have:
\medskip 

{\rm a)}\ for all $i\geqslant 0$,  $g _i:\fr{M}\To\fr{M}$  are unique linear morphisms;
\medskip 

{\rm b)}\ for all $i\geqslant 0$, $g_i(\fr{M})\subset (g-\id _{\fr{M}})^i(\fr{M})$;
\medskip 

{\rm c)}\ if $i_1,\dots ,i_s\geqslant 0$ then $(g_{i_1}\cdot\ldots\cdot g_{i_s})(\fr{M})
\subset (g-\id _{\fr{M}})^{i_1+\dots +i_s}(\fr{M})$; 
\medskip 

{\rm d)}\ the map $g^U =\sum _{i\geqslant 0}g_i\otimes U^i
:\fr{M}\To \fr{M}\otimes \F _p[[U]]$  determines the 
action of the formal additive group $\mathbb{G}_a=\op{Spf}\,\F _p[[U]]$ 
on $\fr{M}$;
\medskip 

{\rm e)}\ if $1\leqslant i<p$ then $g _i=g _1^i/i!$ 
(here $g_1^i=\underset{i\text{ times }}{\underbrace{g_1\cdot\ldots\cdot g_1}}$). 
\end{Prop}

\begin{proof} 
 For any $m\in\fr{M}$, $g _1(m),\dots ,g _{p-1}(m)$ are unique solutions 
of the non-degenerate system of equations 
$$\sum _{1\leqslant i<p}g _i(m)n^i=g^n(m)-m\,$$ 
where $n=1,\dots ,p-1$. Therefore, all $g_i(m)$ are unique and depend linearly on $m$. 
This proves a). 

For $i\geqslant 0$ and $F\in\mathfrak M\otimes\F _p[[U]]$, 
define the $i$-th differences 
$(\Delta ^iF)(U)\in\fr{M}\otimes \F _p[[U]]$ by setting 
$\Delta ^0F=F$ and 
$$(\Delta ^{i+1}F)(U)=(\Delta ^iF)(U+1)-(\Delta ^iF)(U).$$ 

In particular, for $0\leqslant j<i$, $\Delta ^i(m\otimes U^j)= 0$ and 
$(\Delta ^i)(m\otimes U^i)=i!m$.  Therefore, for any $i\geqslant 0$, 
\begin{equation}\label{E3.1}
 (\Delta ^ig^U(m))|_{U=0}=i!g_i(m)+
\sum _{j>i}f_{ij}g_j(m),
\end{equation}
where all $f_{ij}\in \F _p$. 
Note that for every value $n_0\geqslant 0$,  
$$(\Delta ^1g^U)(m)|_{u=n_0}=g(g^U(m)|_{u=n_0})-g^U(m)|_{u=n_0}\in (g-\id _{\fr{M}})(\fr{M}),$$
$$(\Delta ^2g^U)(m)|_{u=n_0}=g ((\Delta ^1g^U)(m)|_{u=n_0})-(\Delta ^1g^U)(m)|_{u=n_0}
\in (g-\id _{\fr{M}})^2(\fr{M})$$ 
and so on. Therefore, for any $i\geqslant 0$, 
$$(\Delta ^ig^U)(m)|_{U=n_0}\in (g-\id _{\fr{M}})^i\fr{M}\,.$$
Then \eqref{E3.1} implies (use $i=p-1$) that $g_{p-1}(m)\in (g-\id _{\fr{M}})^{p-1}(\fr{M})$ 
and then by descending induction on $i$ that $g_i(m)\in
(g-\id _{\fr{M}})^i(\fr{M})$. This proves b). 

In c) use induction on $s$. The case $s=1$ is proved in b). If $s>1$ then 
we must prove with $j=i_2+\dots i_s$ that 
$$g_{i_1}((g-\id _{\fr{M}})^{j}\fr{M})
\subset (g-\id _{\fr{M}})^{i_1+j}\fr{M}\,.$$ 
This can be obtained from a) by replacing $\fr{M}$ to $(g-\id _{\fr{M}})^j\fr{M}$.

For any natural numbers $n_1,n_2$ the relation $g^{n_1+n_2}(m)=g^{n_2}(g^{n_1}(m))$ 
means that 
$$\sum _{0\leqslant i<p}(n_1+n_2)^ig _i=\sum _{0\leqslant i_1,i_2<p}
n_2^{i_2}n_1^{i_1}g _{i_2}\circ g _{i_1}\,,$$
and implies that we have the appropriate identity of formal power series 
$$(g^U\otimes\id _{\mathbb{G}_a})\circ g^U =(\id _{\fr{M}}\otimes \Delta _{\mathbb{G}_a})
\circ g^U \,,$$ 
with the coaddition $\Delta =\Delta _{\mathbb{G}_a}$ in $\mathbb{G}_a$ such that 
 $\Delta (U)=U\otimes 1+1\otimes U$. This proves d). 

If $i\geqslant 1$ the above identity for $g^U$ implies the identity 
$$(g^U\otimes\id _{\mathbb{G}_a^{i}})\circ\dots \circ (g^U\otimes\id _{\mathbb{G}_a})
\circ g^U =(\id _{\fr{M}}\otimes\Delta ^{(i)})\circ g^U\,, $$
where 
$\Delta ^{(i)}=(\Delta \otimes\id _{\mathbb{G}_a^{i-1}})
\circ\dots \circ (\Delta \otimes\id _{\mathbb{G}_a})\circ\Delta  $  
is the $i$-th coaddition $\F _p[[U]]\To \F _p[U]^{\otimes i}$ 
for $\mathbb{G}_a$. Then e) can be obtained by compairing the coefficients for $U^{\otimes i}$ 
in this identity.  
 \end{proof}

\begin{definition} 
 $dg^U:=g_1\otimes U:\fr{M}\To \fr{M}\otimes U$ is the differential of $g$. 
\end{definition}

By above Proposition \ref{P3.1}e) the action of $g$ on $\fr{M}$ can be uniquely 
recovered from its differential $dg^U$.

\subsection{Auxiliary statement} \label{S3.3} 

Assume that $\fr{L}$ is a finite Lie algebra over $\F _p$. 
Let $\c A=\c A(\fr{L})$ be the enveloping algebra of $\fr{L}$. 
Then we have a canonical embedding $\fr{L}\To \c A$.  
Provide $\c A$ with a standard structure of a coalgebra 
$\Delta :\c A\To \c A\otimes\c A$ by setting $\Delta (l)=l\otimes 1+1\otimes l$ 
for all $l\in \fr{L}$. 

Let $J=J(\fr{L})$ be the augmentation ideal of $\c A$ generated by all $l\in \fr{L}$. 
Note that $\c A\otimes\c A$ can be identified 
with the enveloping algebra of $\fr{L}\oplus \fr{L}$ 
and the appropriate augmentation ideal equals 
$J(\fr{L}\oplus \fr{L})=J\otimes \c A+\c A\otimes J$. 

Suppose $\fr{L}$ has nilpotent class $<p$. Then we have the following 
interpretation of the Campbell-Hausdorff operation $\circ $ on $\fr{L}$ 
in the envelopping algebra $\c A$:
\medskip 

$\alpha $)\ $\fr{L}=\{a\in\c A\,\op{mod}\,J(\fr{L})^p\ |
\ \Delta a\equiv a\otimes 1+1\otimes a\ \op{mod}\,
J(\fr{L}\oplus \fr{L})^p\}$;
\medskip 

$\beta $)\ the truncated exponential 
$\wt{\exp}$ establishes a group isomorphism 
$\iota :G(\fr{L})\To\c D(\fr{L})$, where 
$$\c D(\fr{L})=\{a\in 1+J(\fr{L})\,\op{mod}\,J(\fr{L})^p\ |\ 
\Delta a\equiv a\otimes a\,\op{mod}\,J(\fr{L}\oplus \fr{L})^p\}$$ 
is the group of \lq\lq\ diagonal elements of $\c A$ modulo degree 
$p$\rq\rq\ with respect to the operation induced 
by the multiplication in $\c A$;
\medskip 

$\gamma $)\ $\iota ^{-1}:\c D(\fr{L})\To G(\fr{L})$ is given via the truncated logarithm 
$\wt{\log}$. 
\medskip

Let $l_1,\dots ,l_r$ 
be an $\F _p$-basis of $\fr{L}$. Then by the Poincare-Birkhoff-Witt Theorem, 
$\c B_1=\{l_{i_1}\dots l_{i_s}\ |\ s\geqslant 0,\, i_1\leqslant \dots \leqslant i_s\}$  
is an $\F _p$-basis of $\c A$ and $\c A\,\op{mod}\,J(\fr{L})^p$ can 
be identified with the submodule $\c M_1$ of $\c A$ generated by the 
elements of $\c B_1^{<p}:=\{l_{i_1}\dots l_{i_s}\in \c B_1\ |\ s<p\}$.

For similar reasons, use the basis $\{(l_i,0), (0,l_i)\ |\ 1\leqslant i\leqslant r\}$ 
of $\fr{L}\oplus\fr{L}$ to construct the $\F _p$-basis for $\c A\otimes\c A$ in the form 
$$\c B_2=\{l_{i_1}\dots l_{i_{s}}\otimes l_{j_1}\dots l_{j_{t}}\ |\ s,t\geqslant 0,\, 
i_1\leqslant \dots \leqslant i_{s},\, j_1\leqslant\ldots\leqslant j_{t}\}\, .$$  
Then $\c A\otimes\c A\,\op{mod}\,J(\fr{L}\oplus\fr{L})^p$ can be identified 
with the module $\c M_2$ generated by the subset $\c B_2^{<p}$ of $\c B_2$ consisting of elements 
with $s+t<p$. 

Let $\delta ^+=\Delta -\id _{\c A}\otimes 1-1\otimes\id _{\c A}$. Then 
$\delta ^+(\c M_1)\subset\c M_2$ and it is easy to see that:
\medskip 

$\bullet $\ $\fr{L}\subset\Ker\,\delta ^+$; 
\medskip 

$\bullet $\ if $l\in \c B_1^{<p}\setminus\fr{L}$ then $l\notin \Ker\,\delta ^+$; 
\medskip 

$\bullet $\ if $l',l''\in \c B_1^{<p}\setminus\fr{L}$ then 
$\delta ^+(l')$ and $\delta ^+(l'')$ are linear combinations of disjoint groups 
of elements of $\c B_2^{<p}$. 
\medskip 

In other words, we have a direct sum of non-zero submodules 
$$\delta ^+(\c M_1)=\underset{l\in\c B_1^{<p}\setminus\fr{L}}\oplus \F _p\delta ^+(l)\,.$$

The above facts prove $\alpha )$. The verification of $\beta $) and $\gamma $) is formal. 
\medskip 

In this paper we are dealing with more elaborate situation.

Suppose $\fr{L}$ is provided with a   
decreasing filtration of ideals $\{\fr{L}^{i}\}_{i\geqslant 0}$ such that 
$\fr{L}^0=\fr{L}$ and $\fr{L}^i=0$ if $i\geqslant p$. 
 Define the weight function on $\fr{L}$ by 
setting $\op{wt}^*(0)=\infty $ and $\op{wt}^*(l)=i$ if $l\in \fr{L}^i\setminus \fr{L}^{i+1}$. 

Assume in addition that the filtration $\{\fr{L}^i\}$ is \lq\lq central\rq\rq , i.e. for any 
$i,j\geqslant 0$, $[\fr{L}^i,\fr{L}^j]\subset \fr{L}^{i+j}$. 

Suppose the $\F _p$-basis $\{l_i\ |\ 1\leqslant i\leqslant r\}$ of $\fr{L}$ is 
compatible with the filtration $\{\fr{L}^i\}_{i\geqslant 0}$, i.e. 
there are $0=j_0\leqslant j_1\leqslant\dots \leqslant j_p=r$ such that 
for any $i\geqslant 0$, 
$\{l_j\ |\ j_i< j\leqslant r\}$ is an $\F _p$-basis of $\fr{L}^i$. 
Use again $\c B_1$ as a basis of $\c A$ over $\F _p$. Extend $\op{wt}^*$ to $\c A$ 
by setting for every non-zero $\F _p$-linear combination, 
$$\op{wt}^*\left (\sum _{i_1,\dots ,i_s}\alpha _{i_1\dots i_s}l_{i_1}\dots l_{i_s}\right )
=\min \{\op{wt}^*(l_{i_1})+\dots +\op{wt}^*(l_{i_s})\ |\ \alpha _{i_1\dots i_s}\ne 0\}\,.$$ 

Let  
$\c A^i=\{a\in\c A\ |\ \op{wt}^*(a)\geqslant i\}$.  Then 
for any $i,j\geqslant 0$, $\c A^i\c A^j\subset \c A^{i+j}$ 
(use that $\{\fr{L}^i\}$ is \lq\lq central\rq\rq ). In particular, 
$\{\c A^i\}_{i\geqslant 0}$ 
is a decreasing filtration of ideals of $\c A$. Obviously, $\c A^i\cap \fr{L}=\fr{L}^i$.

Let $B$ be a $\Z _p$-linear operator on $\fr{L}$ such that 
for any $l\in \fr{L}^i$, $B(l)\equiv l\,\op{mod}\,\fr{L}^{i+1}$. 
For $l\in \fr{L}$ and $n\in\N $, set in the appropriate $p$-group $G(\fr{L})$,  
$l[n]:=l\circ B(l)\circ \dots \circ B^{n-1}(l)$. 

\begin{Prop} \label{P3.2} Suppose $l\in\fr{L}^1$. 
For $1\leqslant i\leqslant p-1$ there are (unique) 
$l_i\in \fr{L}^i$ such that for any $n\geqslant 0$, 
$l[n]=l_1n+l_2n^2+\dots +l_{p-1}n^{p-1}$. 
\end{Prop}

\begin{proof} Prove the existence of $l_i\in \fr{L}^i$. 
 (For the uniqueness of $l_i$, proceed similarly to Proposition \ref{P3.1}a).)

Clearly, $B=\wt{\exp }(\c B)$, where 
$\c B$ is a linear operator on $\fr{L}$ such that 
for all $i$, $\c B(\fr{L}^i)\subset \fr{L}^{i+1}$. 
If for $0\leqslant i\leqslant p-1$, $l'_{i}=\c B^i(l)/i!$ then 
$l'_{i}\in \fr{L}^{i+1}$ and for any $m\geqslant 0$, 
$B^m(l)=\wt{\exp }(m\c B)(l) =\sum _{i\geqslant 0}l'_{i}m^i$. (We set $0^0=1$.) 

Let $\c E :\fr{L}\To \c A$ be the map given by the truncated exponential.  
Then for $i\geqslant 0$, there are 
$d_{i}\in \c A ^{i+1}$ such that for any $m\geqslant 0$, 
$$\c E (B^m(l))=1+\sum _{i\geqslant 0}d_{i}m^i\,.$$
Therefore, \ 
$\c E(l)\c E(B(l))\dots \c E(B^{n-1}(l))=$ 
$$1+
\sum _{\substack{1\leqslant s<n\\ i_1,\dots ,i_s\geqslant 0}}
\left (\sum _{0\leqslant m_1<\dots <m_s<n}
m_1^{i_1}\dots m_s^{i_s}\right )d_{i_1}\dots d_{i_s}\,.$$
Let $d(i_1,\dots ,i_s):=i_1+\dots +i_s+s$ and 
$$\sum _{0\leqslant m_1<\dots <m_s<n}m_1^{i_1}\dots m_s^{i_s}=f_{i_1\dots i_s}(n)\,.$$

Note that $d_{i_1}\dots d_{i_s}\in \c A^{d(i_1,\dots ,i_s)}$. 

\begin{Lem} \label{L3.3} 
If $s\geqslant 1$, $i_1,\dots ,i_s\geqslant 0$ and 
$d(i_1,\dots ,i_s)<p$ then there are polynomials $F_{i_1\dots i_s}\in\Z _p[U]$ 
such that:
\medskip 

{\rm a)}  for all $n$, $F_{i_1\dots i_s}(n)=f_{i_1\dots i_s}(n)$;
\medskip 
 
{\rm b)} $F_{i_1\dots i_s}(0)=0$;
\medskip 

{\rm c)}  $\deg F_{i_1\dots i_s}=d(i_1,\dots ,i_s)$. 
\end{Lem}

\begin{proof}[Proof of Lemma] First, consider the case $s=1$. 

Apply induction on $i_1$. 

If $i_1=0$ then $f_0(n)=n$ and we can take $F_0=U$. 

Suppose $i_1\geqslant 1$, $d(i_1)<p$ (i.e. $0\leqslant i_1\leqslant p-2$) and our 
Lemma is proved for all indices $j<i_1$. 

For any $m<n$ we have,  
$$(m+1)^{i_1+1}-m^{i_1+1}=\sum _{0\leqslant j\leqslant i_1}C_j(i_1)m^j\, ,$$
where all $C_j(i)\in\Z _p$. Therefore, for any $n\geqslant 0$, 
$$n^{i_1+1}=\sum _{0\leqslant j\leqslant i_1}C_j(i_1)f_j(n)
=\sum _{0\leqslant j<i_1}C_j(i_1)F_j(n)+(i_1+1)f_{i_1}(n)$$ 
and we can take as $F_{i_1}(U)$ the polynomial 
$$\frac{1}{i_1+1}\left (U^{i_1+1}-
\sum _{0\leqslant j<i_1}C_j(i_1)F_j(U)\right )= 
\sum _{j\leqslant i_1+1}A_j(i_1)U^j\in\Z _p[U]\, .$$ 
Clearly, 
the degree of $F_{i_1}$ equals $i_1+1=d(i_1)$ and $F_{i_1}(0)=0$. 
The case $s=1$ is considered. 
\medskip  

Suppose $s>1$ and use induction on $s$. Then for any $m<n$, 
$$f_{i_1\dots i_s}(m+1)-f_{i_1\dots i_s}(m)=
\sum _{0\leqslant m_1<\dots <m_s=m}m_1^{i_1}\dots m_s^{i_s}
=m^{i_s}F_{i_1\dots i_{s-1}}(m)\,.$$
By induction assumption we have 
$$F_{i_1\dots i_{s-1}}(U)=\sum _{j\leqslant d(i_1,\dots ,i_{s-1})}
A_j(i_1,\dots ,i_{s-1})U^j\in\Z _p[U]\, .$$
Then for any $n\geqslant 1$ 
(note that $d(i_1,\dots ,i_s)-1=d(i_1,\dots ,i_{s-1})+i_s$), 
$$f_{i_1\dots i_s}(n)=
\sum _{i_s\leqslant j\leqslant d(i_1,\dots ,i_s)-1}A_{j-i_s}(i_1,\dots ,i_{s-1})F_j(n)\,,$$
and we can take $F_{i_1\dots i_s}=\sum _{i_s\leqslant j
\leqslant d(i_1,\dots ,i_s)-1}A_{j-i_s}(i_1,\dots ,i_{s-1})F_j$. Clearly, the degree 
of $F_{i_1\dots i_s}$ equals $d(i_1,\dots ,i_s)$ and $F_{i_1\dots i_s}(0)=0$. 
\end{proof} 

The above lemma implies that for all $n\geqslant 1$, 
$$\c E(l[n])=1+\sum _{1\leqslant i\leqslant p-1}d_i'n^i+a(l,n)\, ,$$
where all $d_i'\in \c A^i$ and $a(l,n)\in\c A^p$ (recall that $\c A^p\supset J(\fr{L})^p$).

Applying to this equality the truncated logarithm we obtain that 
$l[n]=d''_1n+\dots +d''_{p-1}n^{p-1}+b(l,n)$, where all $d''_i\in \c A^i$  
and $b(l,n)\in \c A^p$. Therefore, for all $1\leqslant n\leqslant p-1$, we have 
$d''_1n+\dots +d''_{p-1}n^{p-1}\in\fr{L}+\c A^p$. 
This implies that all $d''_i\in\fr{L}+\c A^p$ (use that 
$\det (n^{i})_{1\leqslant n,i<p}\not\equiv 0\,\op{mod}\,p$), i.e. 
$d''_i\in \c A^i\cap (\fr{L}+\c A^p)=\fr{L}^i+\c A^p$ 
(use that for $0\leqslant i<p$, $\c A^i\cap\fr{L}=\fr{L}^i$). Finally, if 
$l_i\in \fr{L}$ are such that $d_i''-l_i\in\c A^p$ then 
$$l[n]-(l_1n+l_2n^2+\dots +l_{p-1}n^{p-1})\in \fr{L}\cap\c A^p=0\, .$$

The proposition is proved.
\end{proof}

 As a matter of fact, the proof of Proposition \ref{P3.2} gives the following result:
\medskip 

$\bullet $\ {\it If $i^0\geqslant 1$ and $l\in \fr{L}^{i^0}$ then 
for $1\leqslant i\leqslant p-i^0$ 
there are unique $l_i\in \fr{L}^{i+i^0-1}$ such 
that for any $n\geqslant 0$, $l[n]=l_{1}n+\dots +
l_{p-i^0}n^{p-i^0}$}. 
\medskip 

We should formally follow the above proof of Proposition \ref{P3.1}. Then $l\in \fr{L}^{i_0}$ 
implies that all $l_i'\in \fr{L}^{i+i_0}$, 
$d_i\in \c A^{i+i_0}$. 
\medskip 
Lemma \ref{L3.3} remains unchanged and, finally, all 
$d_i'\in\c A^{i+i_0-1}$ and all $l_i\in \c A^{i+i^0-1}\cap \fr{L}=\fr{L}^{i+i_0-1}$ 
if $i\leqslant p-i_0$. 

This allows us to state the following result.

\begin{Prop} \label{P3.4} 
 There are linear maps $\pi _i:\fr{L}^1\To\fr{L}^1$ such that 
for any $j\geqslant 0$, $\pi _i(\fr{L}^j)\subset\fr{L}^{i+j-1}$ 
(in particular, $\pi _i=0$ if $i\geqslant p$) and for any $l\in\fr{L}^1$ 
and $n\in\N $, $l[n]=\sum _i\pi _i(l)n^i$. 
\end{Prop}

\subsection{Lie algebra $\bar{\c M}^f$ and 
the action of $\id _{\bar{\c L}}\otimes h(p)$} \label{S3.4} 

Here we study the action of $\id _{\bar{\c L}}
\otimes h(p)$ on 
$\bar f=f\,\op{mod}\,\c M_{<p}(p-1)\in\bar{\c M}_{<p}$. 

Note that if $h_{<p}^0$ is the lift from the end of Subsection \ref{S2.3} 
then $h_{<p}^0(f)=c^0\circ (\Ad\,h_{<p}^0\otimes\id _{\c K_{<p}})f$, where 
$c^0\in\c N(1)\subset \c M(1)$, cf. the proof of Lemma \ref{L2.6} step b). 
 
Suppose $h_{<p}$ is any lift of $h$. Then we can use the existence of $l\in\c L=\c L(1)$ 
such that $h_{<p}=h_{<p}^0\eta _0^{-1}(l)$: if  
$(\id _{\c L}\otimes h_{<p})f=
c\circ (A\otimes \id _{\c K _{<p}})f$ then 
by Proposition \ref{P2.3}, 
$c=c^0\circ l\in\c L(1)_k+\c M(1)$. In other words, generally $c\notin\c N(1)$ but 
it always belongs to $\c L(1)_k+\c M(1)\subset\c M$.

Proceeding in $\bar{\c M}$ we have for 
$h(p)=h_{<p}|_{\c K(p)}$, 
$$(\id _{\bar{\c L}}\otimes h(p))\bar f=
\bar c\circ (\bar A\otimes \id _{\c K(p)})\bar f\, ,$$ 
where we set $\bar c=c\,\op{mod}\,\c M(p-1)\in\bar{\c M}$ and 
$\bar A=A\,\op{mod}\,\c L(p)=\op{Ad}\,h(p)=\wt{\exp}(\op{ad}\,h(p))$.

For $n\in\N $, let 
\begin{equation} \label{E3.2} (\id _{\c L}\otimes h_{<p}^n)f=c(n)\circ f(n)\, ,
\end{equation}
where  
$c(n)=(\id _{\c L}\otimes h^{n-1})(c\circ(A\otimes h^{-1})c\circ  
\dots \circ (A\otimes h^{-1})^{n-1}c )$  
and $f(n)=(A^n\otimes \id _{\c K_{<p}})f$. 

Proceeding similarly to Subsection \ref{S3.1} 
we obtain that 
$$\bar f(n):=f(n)\,\op{mod}\,\c M_{<p}(p-1)=
\sum _{i\geqslant 0}\bar f^{(i)}n^i\,,$$  
where $\bar f^{(0)}=\bar f$ and for all $1\leqslant i<p$, 
$\bar f^{(i)}=(\op{ad}^ih(p)\otimes\id _{\c K(p)})\bar f/i!\in $
$(\bar A\otimes\id _{\c K (p)}-\id _{\bar{\c M}_{<p}})^i
\bar{\c M}_{<p}\subset \bar{\c M}_{<p}(i)\,.$ 

Define the new filtration ${\c M}[i]$ on ${\c M}$ by setting $\c M[0]:=\c M$ and 
for $i\geqslant 1$, 
${\c M}[i]:=\c L(i)_k+\c M(i)$.  Consider 
the appropriate filtrations $\bar{\c M}[i]=\c M[i]\, \op{mod}\,\c M(p-1)$ on 
$\bar{\c M}$ and $\bar{\c M}_{<p}[i]=\bar{\c M}[i]+\bar{\c M}_{<p}(i)$ on $\bar{\c M}_{<p}$. 

\begin{Prop} \label{P3.5} There are $c_i\in\c M[i]$ such that 
for all $n\in\N $, $c(n)\equiv \sum _{i\geqslant 1}c_in^i\,\op{mod}\,\c M(p-1)$. 
\end{Prop}

\begin{proof} Consider the Lie algebra $\mathfrak{L}=\bar{\c M}$ with 
filtration $\mathfrak{L}^i:=\bar{\c M}[i]$. Clearly, $\mathfrak{L}$ and its 
filtration $\{\mathfrak{L}^i\}_{i\geqslant 0}$ satisfy the assumptions from 
Subsection \ref{S3.3} and $\bar c\in \mathfrak{L}^1$ (cf. the beginning of this Subsection).  
It remains to apply  Proposition \ref{P3.2}. 
\end{proof}

\begin{Cor} \label{C3.6} For all $n\in\N $, 
$$(\id _{\bar{\c L}}\otimes h(p)^n)\bar f=\sum _{i\geqslant 0}\bar f_in^i\,,$$
where $\bar f_0=\bar f$ and all $\bar f_i\in\bar{\c M}_{<p}[i]$. 
\end{Cor} 

\begin{definition}
$\bar{\c M}^f$ is the minimal Lie subalgebra in $\bar{\c M}_{<p}$ containing 
$\bar{\c M}$ and all the elements 
$(\op{Ad}^n\,h(p)\otimes\id _{\c K(p)})\bar f$ with $n\in\N $.
\end{definition}

Note that $\bar{\c M}^f$ does not depend on a choice of the lift $h(p)$. 
We can also define $\bar{\c M}^f$ as the minimal subalgebra in $\bar{\c M}_{<p}$ containing 
$\bar{\c M}$ and all $\bar f^{(i)}$, $1\leqslant i<p$. Clearly, 
$\id _{\bar{\c L}}\otimes h(p)$ acts on 
$\bar{\c M}^f$ (use that $A\otimes \id _{\c K(p)}$ and 
$\id _{\bar{\c L}}\otimes h(p)$ commute) and this action is completely determined by 
the knowledge of $(\id _{\bar{\c L}}\otimes h(p))\bar f$. Roughly speaking, 
$\bar{\c M}^f$ is much smaller than $\bar{\c M}_{<p}$ 
but it is still provided with a strict action of $\c G_h$. 
In addition,  the filtration $\bar{\c M}_{<p}[i]$ 
induces the $\c G_h$-equivariant filtration $\bar{\c M}^f[i]$ on $\bar{\c M}^f$, and for  
all $i$, $\bar f^{(i)}$ and $\bar f_i$ 
belong to $\bar{\c M}^f[i]$. 

Now we can apply the results of Subsection \ref{S3.2} and introduce  
the appropiate action   
$\id _{\bar{\c L}}\otimes h(p)^U:\bar{\c M}^f\To \bar{\c M}^f\otimes\F_p[[U]]$ 
of $\mathbb{G}_{a,\F _p}$ on $\bar{\c M}^f$. 
This action appears as the extension of the action 
$\id _{\bar{\c L}}\otimes h^U:\bar{\c M}\To \bar{\c M}\otimes\F _p[[U]]$ 
from Subsection \ref{S3.1} by setting 
$$(\id _{\bar{\c L}}\otimes h(p)^U)\bar f=
\sum _{i\geqslant 0}\bar f_i\otimes U^i\, .$$ 
By Proposition \ref{P3.1} the action of $h(p)$ is completely determined by 
the differential $d(\id _{\bar{\c L}}\otimes h(p)^U)$. 
\medskip 

\subsection{Differential $d(\id _{\bar{\c L}}\otimes h(p)^U)$} \label{S3.5}

Using the calculations from Subsection \ref{S3.4} we obtain 
$$\id _{\bar{\c L}}\otimes h(p)^U:\bar f\mapsto \bar c(U)\circ \bar f(U)\, ,$$
where $\bar c(U)=\sum _{i\geqslant 1}c_iU^i\,\op{mod}\,\c M(p-1)$ and 
$\bar f(U)=\bar f+\sum _{i\geqslant 1}\bar f^{(i)}U^i$. 

It makes sense to introduce the formal operator 
$$\op{Ad}^Uh(p):\bar{\c L}\To \bar{\c L}\otimes\F _p[[U]]$$ 
such that for any $l\in\bar{\c L}=\c L/\c L(p)$, 
$\op{Ad}^Uh(p)l=\sum _{i\geqslant 0}l_iU^i$, where 
$l_i=0$ if $i\geqslant p$ and for any $n\in\N $, 
$\op{Ad}^Uh(p)|_{U=n}=\op{Ad}^nh(p)$. Similarly to 
Subsection \ref{S3.2}, for all $i\geqslant 0$, 
$l_i=\op{ad}^ih(p)(l)/i!$ and 
$\Ad ^Uh(p)\equiv \id _{\bar{\c L}} +\op{ad}h(p)\,U\,
\op{mod}\,U^2$. 
This gives the following formal identity
(note $\sigma U=U$):  
\begin{equation} \label{E3.3} 
(\id _{\bar{\c L}}\otimes h^U)(e)\circ \bar c(U)= 
(\sigma \bar c)(U)\circ \sum _{a\in\Z ^0(p)}
t^{-a}(\Ad ^Uh(p)\otimes\id _{k})D_{a0}\, .
\end{equation} 
The proof formally goes along the lines of the proof 
that $(c,A)$ satisfies identity \eqref{E2.1} 
in Proposition \ref{P2.3}.

As a result,  we can specify $(\id _{\bar{\c L}}
\otimes h(p)^U)\bar f$ by the following linearization 
of \eqref{E3.3}. 
Recall, cf. Subsection \ref{S2.1}, that   
$$h(t)=tE(\omega _h^p)\equiv t\wt{\exp}(\omega _h^p)
\op{mod}\,t^{pc_0+1}\, ,$$
where $\omega _h^p=\sum _{i\geqslant 0}A_i(h)t^{c_0+pi}$, 
all $A_i(h)\in k$ and $A_0(h)\ne 0$.  Then by Proposition 
\ref{P2.1}, 
$h^U(t)\equiv t\wt{\exp}(U\omega _h^p)\,\op{mod}\,t^{pc_0+1}$  
and  
$$d(\id _{\bar{\c L}}\otimes h^U)e = 
-\sum _{a\in\Z ^0(p)}t^{-a}
\omega _h^paD_{a0}\otimes U\,\op{mod}\,\c M(p-1)\,.$$

\begin{Prop} \label{P3.7}
We have the following recurrent congruence modulo $\c M(p-1)$
for $\bar c_1=c_1\,\op{mod}\,\c M(p-1)$ and
$V_{a0}:=\op{ad}\,h(p)(D_{a0})\,\op{mod}\,\c L(p)_k$, $a\in\Z ^0(p)$,   
\begin{equation} \label{E3.4} 
\sigma \bar c_1-\bar c_1+\sum _{a\in \Z ^0(p)}t^{-a}V_{a0}\equiv 
\end{equation}

$$-\sum _{k\geqslant 1}\,\frac{1}{k!}\,t^{-(a_1+\dots +a_k)}
\omega _h ^p
[\dots [a_1D_{a_10},D_{a_20}],\dots ,D_{a_k0}]$$
$$-\sum _{k\geqslant 2}\frac{1}{k!}t^{-(a_1+\dots +a_k)}
[\dots [V_{a_10},D_{a_20}],\dots ,D_{a_k0}]$$
$$-\sum _{k\geqslant 1}\,\frac{1}{k!}\,t^{-(a_1+\dots +a_k)}
[\dots [\sigma \bar c_1,D_{a_10}],\dots ,D_{a_k0}]$$
(the indices  
$a_1,\dots ,a_k$ in all above sums run over $\Z ^0(p)$).  
 \end{Prop}

\begin{proof} The following properties are very well-known from the 
Campbell-Hausdorff theory. 
Suppose $X$ and $Y$ are generators of a free Lie $\Q [[U]]$-algebra. Then  
$$ 
 (UY)\circ X \equiv X\circ \left (U\sum _{k\geqslant 0} 
\frac{1}{k!}[\dots [Y,\underset{k\text{ times }}
{\underbrace{X],\dots ,X}}]\right )\, ,
$$ 
$$
X+UY\equiv X\circ 
\left (U\sum _{k\geqslant 1}\frac{1}{k!}
[\dots [Y,\underset{k-1\text{ times }}{\underbrace{X],
\dots ,X]}}\right )\,\op{mod}\,U^2
$$

For the first formula cf. \cite{Bou}, Ch.II, 
Section 6.5 or Exercise 1 for Ch.II, Section 6.  
The second 
congruence is much more important; it can be 
extracted from \cite{Bou}, Ch.II, Section 6.5, Prop.5 or 
Ch.II, Exercise 3 for Section 6.

Using that the coefficients in the above formulas 
are $p$-integral in degrees $<p$ we can use them 
in the context of Lie $\F_p$-algebras in the 
following form (where $E_0(x)=(\wt{\exp}(x)-1)/x$): 
 \begin{equation} \label{E3.5} 
 (UY)\circ X = X\circ \left (U\wt{\exp}(\op{ad}X)(Y)\right )\,\op{mod}\,U^2
\end{equation} 
\begin{equation} \label {E3.6}
X+UY=X\circ (U\,E_0(\op{ad}X)(Y))
\,\op{mod}\,U^2
\end{equation}
\begin{remark} a) In the above formulas and 
this paper we use the following notation: 
$(\op{ad}X)Y=[Y,X]$ and $(\op{Ad}X)Y=(-X)\circ Y\circ X$  
(this notation is opposite to the notation from  
\cite{Bou}). 

b) Note the following easy rules: 
$X\circ (Y+U^2Z)\equiv X\circ Y\,\op{mod}\,U^2$ and 
$(UX)\circ (UY)\equiv U(X+Y)\,\op{mod}\,U^2$. 
\end{remark} 
\medskip 

Then for the  left-hand-side (LHS) of \eqref{E3.3} modulo $U^2$ we have: 
$$(e+d(\id _{\bar{\c L}}\otimes h^U)e+\dots )\circ (\bar c_1U+\dots )\equiv $$
$$e\circ E_0(\op{ad}e)(d(\id _{\bar{\c L}}\otimes h^U)e)\circ 
(\bar c_1U+\dots )\equiv $$

$$e\circ (E_0(\op{ad}e)(d(\id _{\bar{\c L}}\otimes h^U)e)
+\bar c_1U)$$

Similarly, the RHS of \eqref{E3.3} modulo $U^2$ appears in the following form  
$$((\sigma \bar c_1)U+\dots )\circ 
\left (e+U\sum _{a\in\Z ^0(p)}t^{-a}V_{a0}+\dots \right )\equiv $$
$$e\circ \left (U\sum _{a\in\Z ^0(p)}E_0(\op{ad}e)(t^{-a}V_{a0})+
U\wt{\exp}(\op{ad}e)(\sigma \bar c_1)\right )$$

It remains to cancel by $e$ and equalize the coefficients for $U$. 
\end{proof}

 Any solution $\{\bar c_1, \{V_{a0}\ |\ a\in\Z ^0(p)\}\}$ 
of congruence \eqref{E3.4} modulo $\c M(p-1)$ can be uniquely lifted to a 
 solution  
$\{c_1, \{V_{a0}\ |\ a\in\Z ^0(p)\}\}$ of \eqref{E3.4} 
 modulo $\c L(p)_{\c K}\subset \c M(p-1)$. This follows 
easily from Lemma \ref{L2.2}b) because $\sigma $ 
 is nilpotent on $\c M(p-1)\,\op{mod}\,\c L(p)_{\c K}$ 
(use that $\c M(p-1)\subset\c L_{\m }+\c L(p)_{\c K}$). 
 In other words, we have a unique lift of 
 $$\bar c_1\in\c M\,\op{mod}\,\c M(p-1)\subset \c L_{\c K}\,\op{mod}\,\c M(p-1)$$ 
 to $c_1\in \c L_{\c K}\,\op{mod}\,\c L(p)_{\c K}$. 
 This allows us to prove that 
the number of different solutions $\{\bar c_1, \{V_{a0}\ |\ a\in\Z ^0(p)\}\}$  of \eqref{E3.4} 
is $|\c L/\c L(p)|$. Indeed, we can arrange 
the recurrent procedure of solving 
 congruences \eqref{E3.4} modulo $\c L(s)_{\c K}$, where $s=1,\dots ,p$. When 
 $s=1$ we have only trivial solution. 
Then each solution modulo $\c L(s)_{\c K}$ gives a 
 unique extension for all $V_a\,\op{mod}\,\c L(s+1)_k$ and $|\c L(s)/\c L(s+1)|$ 
 different extensions for $c_1\,\op{mod}\,\c L(s+1)_{\c K}$. 
(Compare with the calculations from Subsection \ref{S2.3}.) 
Finally, the number of different solutions of congruence 
\eqref{E3.4} is equal to the number of different 
 lifts of $h$ to $\Aut \,\c K(p)$ which coincides with 
the order $|\Gal (\c K_{<p}^{G(\c L(p))}/\c K)|=|\bar{\c L}|$. 
This is not very much surprising because the lift $h(p)$ is completely 
 determined by $\bar f_1U=d(\id _{\bar{\c L}}\otimes h(p)^U))\bar f$ and $\bar f_1$ 
is uniquely recovered from the knowledge of the appropriate solution 
$\{\bar c_1,\{V_{a0}\ |\ a\in\Z ^0(p)\}\}$ due to 
the following proposition \ref{P3.8} below.  
 
 Recall that for $m\geqslant 0$, 
 $$B_m=\sum _{0\leqslant v\leqslant k\leqslant m}
 (-1)^v\binom{k}v\frac{v^m}{k+1}$$ 
 are the Bernoulli numbers. One of their well-known properties is 
 that 
 $$x/(1-\exp(-x))=\sum _{m\geqslant 0}B_m(-x)^m/m!\,.$$
 
 \begin{Prop} \label{P3.8} 
 $d(\id _{\bar{\c L}}\otimes h(p)^U)\bar f=\bar f_1\otimes U$, where 
 $$\bar f_1=(\op{ad}\,h(p)\otimes \id _{\c K(p)})\bar f+
 \sum _{n\geqslant 0}(-1)^n(B_n/n!)
[\dots [\bar c_1,\underset{n\ \text{times}}{\underbrace {\bar f],\dots ,\bar f}}]\,.$$
   \end{Prop}

\begin{proof} 
In earlier notation we have modulo $U^2$ (use \eqref{E3.5} and \eqref{E3.6}):
$$\left (\id _{\bar{\c L}}\otimes h(p)^U\right )\bar f\equiv 
\bar f+\bar f_1U\equiv (\bar c_1U)\circ (\bar f+\bar f^{(1)}U) $$ 
$$\equiv 
(\bar f+\bar f^{(1)}U)\circ (U\wt{\exp}(\op{ad}\,\bar f)\bar c_1)$$
$$\equiv \bar f\circ (E_0(\op{ad}\,\bar f)\bar f^{(1)}U+\wt{\exp}(\op{ad}\,\bar f)\bar c_1U)$$
$$\equiv \bar f+(\bar f^{(1)}+E_0(\op{ad}\,\bar f)^{-1}(\wt{\exp}(\op{ad}\,\bar f))\bar c_1)U\,.$$
It remains to note that $E_0(x)^{-1}\exp (x)=x/(1-\exp (-x))$. 
\end{proof}

\begin{remark} a) As we already mentioned the above proposition 
implies that the knowledge of the differential $\bar c_1$ of $\bar c$ is sufficient to recover 
the action of $h(p)$ on $\bar f$. In other words, we recover 
the element $(\id _{\bar{\c L}}\otimes h(p)^U)\bar f=
\bar c(U)\circ \bar f(U)$ and therefore, the element $\bar c$. 
This fact can be obtained directly by  establishing a cocycle relation for 
$\bar c(U)$ and verifying that this relation is sufficient to recover $\bar c(U)$ from 
$\bar c_1$.

b)Suppose $\c L'$ is an ideal of $\c L$ such that $\c L'\supset\c L(p)$. Then 
 we can repeat the above arguments to prove that the solutions of \eqref{E3.4} 
 modulo $\c L'_{\c K}$ 
 describe uniquely the lifts of $h$ to automorphisms of $\c K_{<p}^{G(\c L')}$. 
\end{remark}

\subsection{Special cases} \label{S3.6} 

Recurrent  relation \eqref{E3.4} describes explicitly step by step  
the action of the lift $h(p)$. We can agree, for example, to 
find at each step the appropriate values of $\bar c_1$ and $V_{a0}$ 
by the use of the operators 
$\c R$ and $\c S$ from Subsection \ref{S2.2}. 
This will specify uniquely the lift $h(p)$ together with its action by conjugation on 
$\bar{\c L}=\c L/\c L(p)$ and, therefore, will determine  
the structure of $L_h$ (and of the group $\c G_{h}$). 

Let (as earlier) $\omega _h^p=\sum _{i\geqslant 0}\,A _i(h)t^{c_0+pi}$, where 
all $A _i(h)\in k$ and $A _0(h)\ne 0$. Then \eqref{E3.4} 
modulo $C_2(\bar{\c L})_{\c K}+\c M(p-1)$ gives the following congruence 
\begin{equation} \label{E3.7} \sigma c_1-c_1+\sum _{a\in\Z ^0(p)}t^{-a}V_{a0}\equiv 
-\sum _{\substack{a\in \Z ^0(p)\\ i\geqslant 0}}
A_i(h)t^{c_0+pi-a}aD_{a0}\, .
\end{equation}

Applying operator $\c R$, cf. Lemma \ref{L2.2}, we obtain:
\medskip 

$\bullet $\ $V_{00}=(\op{ad}\,h(p))D_{00}=\alpha _0
\op{ad}\,h(p)D_0\in \alpha _0C_2(\bar{\c L})$;
\medskip 

$\bullet $\ for all $b\in\Z ^+(p)$, 
$$V_{b0}=(\op{ad}\,h(p))D_{b0} 
\equiv -\sum _{i\geqslant 0}A _i(h)bD_{b+c_0+pi,0}\,\op{mod}\,C_2(\bar{\c L})_k\, .$$

The second relation means that 
all generators of $\bar{\c L}_{k}$ of the form 
$D_{an}$ with $a>c_0$ can be eliminated from 
the minimal system of  generators of $L_{h,k}$. Indeed,  
because $A_0(h)\ne 0$, all $D_{b+c_0,0}$ belong to the 
ideal of second commutators $C_2(L_h)_k=((\op{ad}\, h(p))\c L_k+C_2(\c L_k))/\c L(p)_k$, and  
for any $n\in\Z /N_0$, all $D_{b+c_0,n}=\sigma ^nD_{b+c_0,0}$ also belong to $C_2(L_h)_k$. 
The first  relation 
then means that $L_h$ has only one relation 
with respect to any minimal set of generators. 
This terminology formally makes sense because  
in the category of Lie $\F _p$-algebras of nilpotent 
class $<p$ the algebras of the form 
$\mathfrak{L}/C_p(\mathfrak{L})$, where $\mathfrak{L}$ is 
a free Lie $\F _p$-algebra, play a role of 
free objects. The same remark also can be 
used for the category of, say, 
$p$-groups of period $p$ and of nilpotent class $<p$. Therefore, 
$\c G_{h}$ can be treated as an object of this category 
with finitely many generators and 
one relation.

As an illustration of Proposition \ref{P3.7}, use the relation 
\eqref{E3.7} modulo $\c L(2)_{\c K}+\c M(p-1)$ and make  the 
next central step to obtain the following explicit formulas for 
$V_{a0}$ modulo $\c L(3)_k=C_3(L_h)_k$ (the elements $\c F_{\gamma ,-N}^0$ 
are generators of ramification ideals introduced in Subsection \ref{S1.4}).  

\begin{Prop} \label{P3.10} We have  the following congruences modulo $\c L(3)_k$:

$$V_{00}\equiv 
-\alpha _0\sum _{\substack{i\geqslant 0
\\ 0\leqslant n<N_0}}
\sigma ^n(A_{i}(h))\sigma ^n(\c F^0_{c_0+pi,0})\, ,
$$
and for all $a\in\Z ^+(p)$, 
$$V_{a0}\equiv 
-\sum _{\substack{n\geqslant 1 \\ i\geqslant 0}}
\sigma ^{n}(A_{i}(h)\c F^0_{c_0+pi+a/p^n,-n})
-\sum _{\substack{m\geqslant 0 \\ i\geqslant 0}}
\sigma ^{-m}(A_{i}(h)\c F^0_{c_0+pi+ap^m,0})
\,.$$
\end{Prop} 
\medskip 

Before sketching the proof of this proposition we 
explain why the sums in the last formula are finite. 

\begin{Prop} \label{P3.11} Suppose $a\in\Z ^0(p)$. Then: 
\medskip 

{\rm a)}\ for any $N,m\geqslant 0$, 
$\c F^0_{c_0+pi+ap^m,-N}\equiv\c F^0_{c_0+pi+ap^m,0}\,\op{mod}\,\c L(3)_k$; 
\medskip

{\rm b)}\ for any $N\geqslant n\geqslant 1$, $\c F^0_{c_0+pi+a/p^n,-N}
\equiv\c F^0_{c_0+pi+a/p^n,-n}\,\op{mod}\,\c L(3)_k$;
\medskip 

{\rm c)}\ if $m\geqslant 0$ and $c_0+pi+ap^m>2c_0-1$  
then  $\c F^0_{c_0+pi+ap^m,0}\in\c L(3)_k$;
\medskip 

{\rm d)}\ if $n\in\N $ and $(c_0-1)(1+p^{-n})<c_0$ then $\c F^0_{c_0+pi+a/p^n,-n}\in\c L(3)_k$. 
\end{Prop} 
\medskip 

\begin{proof}
 a)\ If it is false then $\c F^0_{c_0+pi+ap^m,-N}$ 
 should contain a term of the form 
$a_1[D_{a_10},D_{a_2n_2}]$, where $n_2\leqslant -1$ and $a_1+a_2p^{n_2}=c_0+pi+ap^m\in\Z $;  
this implies $a_2=0$ and $a_1=c_0+pi+ap^m\geqslant c_0$; therefore, 
$D_{a_10}\in\c L(2)_k$ and our commutator belongs to $\c L(3)_k$. 

b)\ It is obvious if $a\ne 0$ -- in this case both elements don't contain 
linear terms and for any second commutator $a_1[D_{a_10},D_{a_2n_2}]$ we should have 
$a_2\ne 0$ and $n_2=-n$. If $a=0$ then cf. a).  

c)\ $\c F^0_{c_0+pi+ap^m,0}$ can contain a linear term only if $m=0$ which then must be 
equal to $aD_{c_0+pi+a,0}$, but then $c_0+pi+a>2c_0$ and it belongs to $\c L(3)_k$;  
if we have a second commutator $a_1[D_{a_10},D_{a_2n_2}]$ then the condition 
$a_1+p^{n_2}a_2> 2c_0-1$ implies also that this commutator belongs to $\c L(3)_k$.

d)\ In this case there is no linear term, and 
any appeared second commutator $a_1[D_{a_10},D_{a_2n_2}]$ should be 
such that $n_2=-n$, $a_1, a_2\leqslant c_0-1$ but then 
$a_1+a_2p^{n_2}$ will be less than $c_0<c_0+pi+a/p^n$. 
\end{proof}

\begin{proof} [Proof of Proposition \ref{P3.10}]  

From \eqref{E3.7} we obtain (apply the operator $\c S$ from Subsection \ref{S2.2})

$$c_1\equiv  
\sum _{\substack{0<a<c_0+pi\\ i,n\geqslant 0}}
\sigma ^nA_i(h)t^{p^n(c_0+pi-a)}aD_{an}\,\op{mod}\,\c L(2)_{\c K}+\c M(p-1) .$$
(Modulo $\c L(2)_{\c K}$ we can ignore all terms with $a>c_0$.) 
Then the right-hand side of \eqref{E3.4} modulo $\c L(3)_{\c K}+\c M(p-1)$ appears as 

$$-\sum _{a,i}A_i(h)t^{c_0+pi-a}aD_{a0}-\frac{1}{2}\sum _{a_1,a_2,i}
A_i(h)t^{c_0+pi-a_1-a_2}a_1[D_{a_10},D_{a_20}]$$

$$+\frac{1}{2}\sum _{a_1,a_2,i}A_i(h)t^{-(a_1+a_2)}a_1[D_{a_1+c_0+pi,0}, D_{a_20}]$$
$$-\sum _{\substack{a_1,a_2,n,i\\0<a_1<c_0+pi}}\sigma ^{n}(A_i(h))
t^{p^{n}(c_0+pi-a_1)-a_2}a_1[D_{a_1,n},D_{a_20}]$$

In the above sums the indices $a,a_1,a_2$ run over $\Z ^0(p)$, $i\geqslant 0$ and  
$n\geqslant 1$.  
The third sum can be ignored because all $D_{a_1+c_0+pi,0}\in C_2(L_h)_k$ 
and for the similar reason we can ignore the restriction $0<a_1<c_0+pi$ in the last sum. 

Now note that the terms from the first line can be grouped as follows:
\medskip 

--- the constant terms (i.e. the coefficients for $t^0=1$) appear as 
$$-\frac{1}{2}\sum _{i}A_i(h)\sum _{a_1+a_2=c_0+pi}
a_1[D_{a_10},D_{a_20}]=-\sum _{i}A_i(h)\c F^0_{c_0+pi,0}\,;$$

--- the remaining terms are grouped with respect to the 
condition $a=c_0+pi+b$ or $a_1+a_2=c_0+pi+bp^m$, where 
$b\in\Z ^+(p)$ and $m\geqslant 0$, and appear as 
$$-\sum _{i}A_i(h)\sum _{b,m} t^{-bp^m}\c F^0_{c_0+pi+bp^m,0}\,;$$

The terms from the last line are grouped (modulo $\c L(3)_{\c K}$) with respect to the 
condition $a_1+a_2/p^n=c_0+pi+b/p^n $, where $b\in\Z ^+(p)$ and $n\geqslant 1$, and appear as 
$$-\sum _i\sigma ^n(A_i(h))\sum _{a}t^{-a}\sigma ^n\c F^0_{c_0+pi+a/p^n,-n}\,.$$

It remains to recover the values of $V_b$ by applying  
the operator $\c R$ from Subsection \ref{S2.2}. 
\end{proof}

\section{Arithmetical lifts} \label{S4} 

Recall that the lifts $h_{<p}\in\Aut\,\c K_{<p}$ of 
$h\in\op{Aut}\c K$ generate  
the group $\wt{\c G}_h\subset\Aut (\c K_{<p})$. 
The images $h(p)$ 
of all $h_{<p}$ generate the group  $\wt{\c G}_h/G(\c L(p))=
\wt{\c G}_h/C_p(\wt{\c G}_h)\subset \Aut\,\c K(p)$ 
and by results of Section \ref{S3}, 
can be described quite efficiently via the 
differentials $d(\id _{\bar{\c L}}\otimes h(p)^U)$. 
In this Section we introduce the concept of arithmetical lift $h_{<p}$ of $h$ and  
prove that this property depends only on the image $h(p)$ of $h_{<p}$. 
We also obtain a characterization 
of this property in terms related to 
the differentials $d(\id _{\bar{\c L}}\otimes h(p)^U)$.

\subsection{Review of ramification theory} \label{S4.1} 
The following brief sketch of the ramification theory of 
continuous automorphisms of complete 
discrete valuation fields with finite residue field of characteristic $p$ 
(we need only this case) is based 
on the papers \cite{De, Wtb1, Wtb2}. 

Let $\c E$ be a basic complete discrete valuation field with finite  
residue field $k_{\c E}$. Let   
$R_0(\c E)$ be the completion of a  separable closure $\c E_{sep}$ of $\c E$. 
Note that in the characteristic 0 case, we have  $R_0(\c E)=\mathbb{C} _p$,  
and in the characteristic $p$ case, we have  $R_0(\c E)=\op{Frac}R:=R_0$ is 
the field of fractions of Fontaine's ring $R=\varprojlim O_{\mathbb{C}_p}/p$ 
(the projective limit is taken with respect to the transition maps induced by 
taking $p$-th powers). 

 Denote by  
$v_{\c E}$ the unique extension of the normalized valuation on $\c E$ to $R_0$.  
Let $\c I$ be the group of all continuous automorphisms of $R_0$ 
which are compatible with $v_{\c E}$ and induce the identity map on 
the residue field of $R_0$. 

Agree that all fields below $E,F,L$ etc, are finite extensions of $\c E$ in $\c E_{sep}$ 
and use the appropriate notation $v_E$, $k_E$, etc. 
Let $\m _E$ be the maximal ideal of the valuation ring of $E$.  
Note that the inertia subgroup $\Gamma ^0_E$ of $\Gamma _E=\Gal (\c E_{sep}/E)$ 
is a subgroup in $\c I$. 

Let $\c I_E=\{\iota |_E\ |\ \iota\in\c I\}$. 

For $g\in\c I_E$, let $v(g)=\op{min}\,\left\{ v_E(g(a)-a)\ |\ a\in\m _E\right\}-1$. 

For $x\geqslant 0$, set 
$\c I_{E,x}=\{g\in\c I_E\ |\ v(g)\geqslant x\}\,.$

For a field extension $F/E$, let  
$\c I_{F/E}=\{\iota\in\c I_F\ |\ \iota |_E=\id _E\}$. 
For $x\geqslant 0$, let 
$$\c I_{F/E,x}=\c I_{F,x}\bigcap\c I_{F/E}\,.$$ 

If $\iota _1,\iota _2\in \c I_{F/E}$ and $x\geqslant 0$ 
then $\iota _1$ and $\iota _2$ are {\it $x$-equivalent} iff 
for any $a\in \m_F$,  $v_F(\iota _1(a)-\iota _2(a))\geqslant 1+x$. 
Denote by $(\c I_{F/E}:\c I_{F/E,x})$ 
the number of 
$x$-equivalent classes in $\c I_{F/E}$.  
Then the Herbrand function for $F/E$ can be defined for all $x\geqslant 0$, 
as 
$\varphi _{F/E}(x)=\int _0^x(\c I_{F/E}:\c I_{F/E,x})^{-1}dx$. 
This function has the following properties:
\medskip 

$\bullet $\ $\varphi _{F/E}$ is a piece-wise linear function   
with finitely many edges;
\medskip 

$\bullet $\ if $L\supset F\supset E$ is a tower of finite field extensions then 
for any $x\geqslant 0$, $\varphi _{L/E}(x)=\varphi _{F/E}(\varphi _{L/F}(x))$; 
\medskip 

$\bullet $\ the last edge point of the graph of $\varphi _{F/E}$ is 
$(x(F/E), v(F/E))$, where 
$$x(F/E)=\op{inf}\,\left\{x\geqslant 0\ |\ (\c I_{F/E}:\c I_{F/E,x})=|\c I_{F/E}|\right\}$$
is the largest lower and $v(F/E)=\varphi _{F/E}(x(F/E))$ 
is the largest upper ramification numbers for the extension $F/E$.
\medskip  

The following proposition is just a direct adjustment 
of the appropriate fact from the classical 
ramification theory for finite Galois extensions. 

\begin{Prop}\label{P4.1}  Suppose $g\in\c I_{E}$ and 
$v(g)=y$. Then 
$$\max\{v(f)\ |\ f\in\c I_F,\ f|_E=g\}=\varphi ^{-1}_{F/E}(y)\, .$$ 
\end{Prop}

\begin{proof} We can assume that $F/E$ is totally ramified of degree $d$. 

Suppose $\theta $ is a uniformizing element in $F$ and $P(T)\in E[T]$ 
is its minimal monic polynomial over $E$. Then $P(T)=T^d+a_1T^{d-1}+\dots +a_d$ is an 
Eisenstein polynomial and $v(g)=v_E(g(a_d)-a_d)-1=y$. 

Note that for all $1\leqslant i< d$, $v_E(g(a_i)\theta ^{d-i}-a_i\theta ^{d-i})>
v_E(g(a_d)-a_d)$, Therefore, 
$v_E(g_*P(\theta ))=v_E(g_*(P)(\theta )-P(\theta ))=1+y$.

Let $\theta _1, \dots ,\theta _d$ be 
all roots of $g_*P(T)$ in $\hat E_{sep}$. 
Then all $d$ different lifts $f_i$ of $g$ to $F$ are uniquely determined by 
the condition $f_i(\theta )=\theta _i$, $i=1,\dots ,d$. 
Clearly, $v(f_i)=v_F(\theta -\theta _i)-1$. 

Assume that $x=v(f_1)$ is maximal, i.e. 
$1+x\geqslant v_F(\theta -\theta _i)$ for 
all $i$. It remains to prove that $y=\varphi _{F/E}(x)$. 

Let $A_i:=v_F(\theta _i-\theta _1)-1\geqslant 0$. Note $A_1=+\infty $.  
Then 
$$v_F(g_*P(\theta ))=\sum _{1\leqslant i\leqslant d}v_F(\theta -\theta _i)=
\sum _{1\leqslant i\leqslant d}\min\{1+x, 1+A_i\}=d+\varphi (x)$$

The function 
$\varphi (x)=\sum _{1\leqslant i\leqslant d}\min\{x, A_i\}$ is peace-wise linear, 
$\varphi (0)=0$ and if $x$ is different from all $A_i$ then 
$$\varphi '(x)=|\{A_i\ |\ A_i>x\}|=|\c I_{F/E,x}|=(\c I_{F/E}:
\c I_{F/E, x})^{-1}d=d\varphi '_{F/E}(x)\, .$$ 

Therefore, $\varphi (x)=d\varphi _{F/E}(x)$ and, finally, 
$1+y=v_E(g_*P(\theta ))=$ 
$d^{-1}v_F(g_*P(\theta ))=d^{-1}(d+d\varphi _{F/E}(x))=1+\varphi _{F/E}(x).$
\end{proof}

\begin{Cor} \label{C4.2} 
The restriction $\c I_F\To\c I_E$ given by the correspondence 
$f\mapsto g:=f|_E$ defines for any $x_0\geqslant 0$, the surjection 
$\c I_{F,x_0}\To\c I_{E,y_0}$, where $y_0=\varphi _{F/E}(x_0)$. 
\end{Cor}

\begin{proof} Let $f\in\c I_{F,x_0}$ and $v(g)=y$. 
By Proposition \ref{P4.1},  
$x_0\leqslant v(f)\leqslant\varphi ^{-1}_{F/E}(y)$. This implies that 
$y_0\leqslant y$, i.e. $g\in\c I_{E,y_0}$. 

On the other hand, if $g\in\c I_{E,y_0}$ then $v(g)=y\geqslant y_0$ and by 
Proposition \ref{P4.1} there is $f\in\c I_{F,\varphi ^{-1}_{F/E}(y)}\subset\c I_{F,x_0}$ 
such that $g=f|_E$. 
\end{proof}

\begin{definition}
 The ramification filtration $\{\c I_{/E}^{(y)}\}_{y\geqslant 0}$ on $\c I$ with 
 the upper numbering over $E$ is a decreasing sequence of the subsets 
 $\c I_{/E}^{(y)}\subset \c I$ for all $y\geqslant 0$, such that 
 $$\c I_{/E}^{(y)}=\{\iota\in\c I\ |\ \forall\, F/E,\ \iota |_F
\in\c I_{F,\varphi ^{-1}_{F/E}(y)}\}\,.$$
\end{definition}

Note that for any $y\geqslant 0$, $\c I_{/E}^{(y)}=\c I_{/F}^{(y_F)}$, 
where $\varphi _{F/E}(y_F)=y$. Also, 
$\Gamma _E^{(y)}:=\Gamma _E\cap \c I_{/E}^{(y)}$ is the usual 
higher ramification subgroup $\Gamma ^{(y)}_E$ of $\Gamma _E$ with 
the upper number $y$ from \cite{Se1}.  
The largest ramification number $v(F/E)$ is characterized by the following property: 
\medskip 

$\bullet $\  {\it the ramification subgroup 
$\Gamma _E^{(y)}$ acts trivially on $F$ iff $y>v(F/E)$.} 
\medskip 

\subsection{Arithmetical lifts} \label{S4.2} 
Use the notation from Subsection \ref{S4.1}.

\begin{definition} For a field extension $F/E$ we say that $f\in\c I_F$ is 
arithmetical over $E$ (or $f$ is an arithmetical lift of $g=f|_E$) if 
$v(g)=\varphi _{F/E}(v(f))$. Equivalently, 
$f$ is arithmetical over $E$ if there is $\iota\in\c I^{(v(g))}_{/E}$ 
such that $\iota |_F=f$. 
\end{definition}

Note that Corollary \ref{C4.2} implies that $f$ is arithmetical 
over $E$ iff $v(f)=\max\,\{v(f')\ |\ 
f'\in\c I_F\,, f'|_E=g\}$. In particular, arithmetical lifts always exist.

Proposition \ref{P4.1} and Corollary \ref{C4.2} imply the following property.

\begin{Prop} \label{P4.3} Suppose $E\subset L\subset F$ are 
finite field extensions and $f\in\c I_F$. Then:

{\rm a)} $f$ is arithmetical over $E$ iff $f$ is arithmetical 
over $L$ and $f|_L$ is arithmetical over $E$;
\medskip 

{\rm b)} suppose $F/E$ is Galois, $f, f'\in\c I_F$ are such that 
$f|_E=f'|_E=g$ and $f$ is arithmetical over $E$; then $f'$ is arithmetical over $E$ iff 
there is $\tau\in\Gamma _E^{(v(g))}$ such that $f'=f\,(\tau |_F)$. 
\end{Prop}

\begin{proof} The part a) follows from the composition property of the 
 Herbrand function. 
As for the part b), note that $f=\iota |_F$, where 
$\iota \in\c I_{/E}^{(v(g))}$ and 
 there is $\tau\in\Gamma _E$ such that for $\iota ':=\iota\tau $, 
we have $f'=\iota '|_F$. 
We must verify that 

$\bullet $\ \  $\iota '\in\c I_{/E}^{(v(g))}$ iff $\tau\in\c I_{/E}^{(v(g))}\cap\Gamma _E=
 \Gamma _E^{(v(g))}$. 
\medskip 

Suppose $\iota '\in\c I^{(v(g))}_{/E}$. Then for any finite field extension $E'/E$, 
and  any $a\in \m _{E'}$, we have that 
$$\varepsilon ':=\varphi ^{-1}_{E'/E}(v(g))+1\leqslant v_{E'}(\iota '(a)-a)= 
v_{E'}(\iota (\tau a-a)+(\iota (a)-a))\, .$$
But $v_{E'}(\iota (a)-a)\geqslant \varepsilon '$ 
(use that $\iota\in\c I^{(v(g))}_{/E}$) implies  
$v_{E'}(\tau a-a)\geqslant \varepsilon '$ and, therefore, 
$\tau\in\Gamma _E^{(v(g))}$.

Inversely, if $\tau\in\Gamma _E^{(v(g))}$ and $a\in\m _{E'}$ then 
$v_{E'}(\tau a-a)\geqslant \varepsilon '$ and  
$v_{E'}(\iota '(a)-a)$ 
$=v_{E'}(\iota (\tau a-a)+\iota (a)-a)\geqslant \varepsilon '$,
 i.e. $\iota '\in\c I_{/E}^{(v(g))}.$
\end{proof}

As a direct  application of the above proposition note the following. 

Suppose $g\in\c I_E$, $v_g=v(g)$ and $\c E^{(v_g)}\subset\c E_{sep}$ 
is the subfield fixed by $\Gamma _E^{(v_g)}$. 
We shall call $f\in\c I$ arithmetical over $E$ if for any 
finite extension $F/E$ the restriction $f|_F$ 
is arithmetical over $E$. 

\begin{Cor}\ \label{C4.4} {\rm a)}\ $\iota \in\c I$ is 
arithmetical lift of $g=\iota |_E$ if and only if  
$\iota ^{(v_g)}:=\iota |_{\c E^{(v_g)}}$ is arithmetical over $E$;
\medskip 

{\rm b)}\ $\iota ^{(v_g)}$ is a unique arithmetical lift of $g$ to $\c E^{(v_g)}$. 
\end{Cor}

\begin{proof} 
 Suppose $F/E$ is Galois, $\Gal (F/E)=\Gamma $, $F^{(v_g)}=F^{\Gamma ^{(v_g)}}$,  
$f\in\c I_F$, $f|_E=g$ and $f|_{F^{(v_g)}}=f^{(v_g)}$.  

If $f$ is arithmetical over $E$ then by Proposition \ref{P4.3}a) 
$f^{(v_g)}$ is also arithmetical over $E$. 

Inversely, suppose 
$f^{(v_g)}$ is arithmetical over $E$ and $f'\in\c I_F$ is arithmetical lift of 
$f^{(v_g)}$ to $F$. Then there is 
$\tau\in\Gal (F/F^{(v_g)})=\Gamma ^{(v_g)}$ such that $f=f'\tau $ and by Proposition 
\ref{P4.3}b) $f$ is arithmetical over $E$. This proves a) of our proposition. 

Suppose $h,h'\in \c I_{F^{(v_g)}}$ are lifts of $g$.  Then there is 
$\tau\in \Gamma _{F^{(v_g)}}:=\Gal (F^{(v_g)}/E)$ such that 
$h'=h\tau $. If $h,h'$ are arithmetical over $E$ then by Proposition \ref{P4.3}b) 
$\tau\in\Gamma _{F^{(v_g)}}^{(v_g)}=\{e\}$ and $h=h'$.  
 \end{proof}

\subsection{Characterization of arithmetical lifts} \label{S4.3} 
Consider, as earlier, the field extension $\c K_{<p}/\c K$ 
and a lift $h_{<p}\in\Aut\,\c K_{<p}$ of $h$. 

Suppose $h_{<p}$ is arithmetical over $\c K$.

By Corollary \ref{C4.4}b) such  lift $h_{<p}$ is unique 
modulo the ramification subgroup $\c G_{<p}^{(c_0)}=G(\c L^{(c_0)})$ (note that $v(h)=c_0$). 
Therefore, we can characterize arithmetical lifts $h_{<p}$ by studying the action of 
$h_{<p}$ on 
$$f\,\op{mod}\,\c L^{(c_0)}_{\c K_{<p}}\in (\c L/\c L^{(c_0)})_{\c K^{(c_0)}}\, ,$$
where $\c K^{(c_0)}:=\c K_{<p}^{G(\c L^{(c_0)})}$, cf. Subsection \ref{S1.3}. 

The following proposition provides us with the opportunity 
to characterize arithmetical lifts $h_{<p}$ 
by working with $\bar f=f\,\op{mod}\, \c M_{<p}(p-1)$. 
(Use that  $\bar f$ allows us to control efficiently 
the lifts $h(p)=h_{<p}|_{\c K(p)}$ and Corollary \ref{C4.4}. )

\begin{Prop}\label{P4.5}
$\c L(p)\subset\c L^{(c_0)}$. 
\end{Prop} 

\begin{proof} Proposition follows easily from Lemma \ref{L4.7} below.
\end{proof} 

Note the following corollary. 

\begin{Cor} \label{C4.6}  
 $h_{<p}$ is arithmetical iff $h(p)$ 
is arithmetical (over $\c K$). 
\end{Cor} 

Indeed, use that both automorphisms are arithmetical over $\c K$ iff   
$h_{<p}|_{\c K^{(c_0)}}=h(p)|_{\c K^{(c_0)}}:=h^{(c_0)}$ 
is arithmetical over $\c K$.

\begin{Lem} \label{L4.7} 
 If $\op{wt}\, (D_{an})\geqslant s$, cf. Subsection \ref{S2.3}, then 
 $$D_{an}\in \c L_k^{(c_0)}+C_s(\c L_k)\, .$$ 
\end{Lem}

\begin{proof} [Proof of lemma] This lemma was proved in \cite{Ab1} but the proof is 
very short and we shall reproduce it. Recall that $\op{wt} (D_{an})\geqslant s$ means that 
$(s-1)c_0\leqslant a$. Use induction on $s$. 

If $s=1$ there is nothing to prove. 

Assume $s\geqslant 2$ and the lemma is proved for all $s'<s$. Consider  
$$\c F^0_{a,-N}=aD_{a0}+(\text{ commutators of order}\geqslant 2)\in\c L^{(c_0)}_k$$ 
from 
Subsection \ref{S1.3}. This element is a linear combination of 
the commutators of the form \ \ 
$a_1[\dots [D_{a_1n_1},D_{a_2n_2}],\dots ,D_{a_tn_t}]$, where 
\medskip 

--- $0=n_1\geqslant\dots \geqslant n_t\geqslant -N$;
\medskip 

--- $a=a_1p^{n_1}+\dots +a_tp^{n_t}$.
\medskip 

If for $1\leqslant i\leqslant t$, $\op{wt} (D_{a_in_i})=s_i$ then 
$a\leqslant a_1+\dots +a_t<(s_1+\dots +s_t)c_0$ and this implies that 
$s\leqslant s_1+\dots +s_t$. 

Suppose $t\geqslant 2$. Then $\op{wt}(D_{a_in_i})\geqslant\min\{s_i,s-1\}$ and 
by the inductive assumption our commutator belongs to 
$\c L_k^{(c_0)}+C_{s'}(\c L_k)$, where 
$$s'=\sum _{1\leqslant i\leqslant t}\min\{s_i,s-1\}
\geqslant\min\{s_1+\dots +s_t,s\}=s\,.$$ 
\end{proof}

As a result, the property for $h_{<p}$ to be 
arithmetical over $\c K$ can be stated in terms of 
the differential $(\id _{\bar{\c L}}\otimes h(p)^U)\bar f=
\bar f_1\otimes U$ or, equivalently 
in terms of $(\op{ad}\,h(p)\otimes \id _{\c K(p)})\bar f$ and the linear part 
$\bar c_1\in \bar{\c M}[1]$ of $\bar c(U)$, 
cf. Proposition \ref{P3.8}. 

Note that if $h_{<p}$ is arithmetical then for any $g\in \c G_{<p}$, $h_{<p}^{-1}\,g\,h_{<p}
\equiv g\,\op{mod}\,\c G^{(c_0)}$. (Indeed, $g^{-1}h_{<p}g$ is another lift of $h$ 
which is also arithmetical and, therefore, it coincides with $h_{<p}$ 
modulo $\c G_{<p}^{(c_0)}$.) Therefore, 
$\op{Ad}h_{<p}\equiv \id _{\c L}\,\op{mod}\c L^{(c_0)}$.  
 In particular, 
$$(\Ad h_{<p}\otimes \id _{\c K_{<p}})f\equiv 
f\,\op{mod}\,\c L_{\c K_{<p}}^{(c_0)}$$ 
is a necessary condition for $h_{<p}$ to be arithmetical. 
It is natural to expect that a sufficient condition 
for $h_{<p}$ to be arithmetical over $\c K$ requires additional condition which 
can be stated in terms of 
$\bar c_1\,\op{mod}\,\c L^{(c_0)}_{\c K}$, cf. 
Subsection \ref{S3.5}. Even more, we are going to establish this condition 
in terms related only to 
$c_1(0)\in\c L_k\,\op{mod}\,\c L^{(c_0)}_k$, where we set 
$\bar c_1=\sum _{m\in\Z } c_1(m)t^m\,\op{mod}\,\c M(p-1)$ with  
all $c_1(m)\in\c L_k$.

\begin{Thm} \label{T4.8}  The following properties are equivalent: 
\medskip 

{\rm a)}\ $h_{<p}$ is arithmetical over $\c K$; 
\medskip 
  
{\rm b)}\ $(\Ad h_{<p}-\id _{\c L})\c L\subset\c L^{(c_0)}$ and for a sufficiently large $N$, 
$$\bar c_1\equiv  \sum _{\gamma ,j}
\sum _{0\leqslant i<N}\sigma ^i(A_j(h)\c F^0_{\gamma ,-i}
t^{-\gamma +c_0+pj})                          
\,\op{mod}\,\c L_{\c K}^{(c_0)}+\c M(p-1)\, ;$$

{\rm c)} for a sufficiently large $N$, 
$$c_1(0)\equiv  \sum _{j\geqslant 0}
\sum _{0\leqslant i<N}\sigma ^i(A_j(h)\c F^0_{c_0+pj,-i})                          
\,\op{mod}\,\c L_k^{(c_0)}\, .$$
\end{Thm}

\begin{remark}
 Note that if $\gamma\geqslant c_0$ and $i\geqslant\wt{N}(c_0)$, cf. Theorem \ref{T1.2}, 
 then $\c F^0_{\gamma ,-i}\in\c L_k^{(c_0)}$. There is also $\delta >0$, cf. 
Subsection \ref{S4.4}, such that 
 if $\c F^0_{\gamma ,-i}\ne 0$ and $\gamma <c_0$ then $\gamma <c_0-\delta $. 
 (In other words, 
 any $\gamma\in [c_0-\delta ,c_0)$ can't be presented 
 in the form $a_1+a_2p^{n_2}+\dots +a_sp^{n_s}$, where $1\leqslant s<p$, all $n_j\leqslant 0$ 
 and all $a_j\in\Z ^0(p)$.)
 Therefore, in b) 
 we can take 
 $N\geqslant \max\{\wt{N}(c_0), \log_p((p-1)c_0/\delta )\}$ and in c) $N\geqslant\wt{N}(c_0)$ 
(use that under these conditions the appropriate RHS's do not depend on $N$).  
\end{remark}

\subsection{Auxiliary result} \label{S4.4} 

We review here a technical result from \cite{Ab3}, Section 3. (Note that 
all results in \cite{Ab3} were obtained in the contravariant setting.) 
This paper deals with explicit calculations with 
ramification ideals in Lie algebras over $\Z /p^{M+1}$. It is much easier to follow 
these calculations when assuming that $M=0$ (we need only this case). 
First, introduce the relevant 
objects and assumptions. 
\medskip 

{\it Introduction of objects}. 
\medskip  

Set $M=0$ (we need the period $p$ case but all 
constructions in Section 3 of \cite{Ab3}  
were done modulo $p^{M+1}$). 
Let $A=[0,(p-1)v_0)\cap\Z ^0(p)$, where $v_0\geqslant 0$ 
(later we shall specify $v_0=c_0$). (In \cite{Ab3} we used $pv_0$ 
in the definition of $A$ instead of $(p-1)v_0$ but everything works with $(p-1)v_0$.) 
Let $\c L(A)$ be a free Lie algebra over 
$k\simeq\F _{p^{N_0}}$ with the set 
of generators 
$$\{\c D _{an}\ |\ a\in A^+=A\cap\Z ^+(p), n\in\Z /N_0\}\cup\{\D _0\}\, .$$
As a matter of fact, we agreed in \cite{Ab3} that 
$n\in\Z $ and $\c D_{an_1}=\c D_{an_2}$ iff $n_1\equiv n_2\,\op{mod}\,N_0$.
For $n\in\Z $, set $\D _{0n}=(\sigma ^n\alpha _0)\D _0$ and note that 
again $\c D_{0n}$ depends only on $n\,\op{mod}N_0$.  
Consider the $\sigma $-linear morphism $\c L(A)\To\c L (A)$ such that 
for all $a,n$, $\D _{an}\mapsto \D _{a,n+1}$ and denote 
this morphism  also by $\sigma $. 
Then $\c L^0:=\c L(A)|_{\sigma =\id }$ is a 
free Lie algebra over $\F _p$ and 
$\c L ^0_k=\c L(A)$. 

Consider the contravariant analogue of the 
elements $\c F^0_{\gamma ,-N}$ from 
Subsection \ref{S1.4} (use the same conditions for all involved indices) 
$$\c F_{\gamma ,-N}=\hskip -5pt \sum _{1\leqslant s<p}(-1)^{s-1}\hskip -5pt 
\sum _{\substack {a_1,\dots ,a_s\\n_1,
\dots ,n_s}}a_1\eta (n_1,\dots ,n_s)[\dots [\D _{a_1n_1},
\D _{a_2n_2}],\dots ,\D _{a_sn_s}]\, .$$
Recall that $a_1,\dots ,a_s$ run over $A$ and $n_1,\dots ,n_s$ run over $\Z $ such that 
$\gamma (\bar a,\bar n)=a_1p^{n_1}+\dots +a_sp^{n_s}=\gamma $. 

Denote by $\c L^0_N(v_0)$ the minimal ideal in $\c L^0$ such that its extension of scalars 
$\c L^0_N(v_0)_k$ contains all $\c F_{\gamma ,-N}$ with $\gamma\geqslant v_0$. 
Let $\wt{N}(v_0,A)$ be such that 
the ideals $\c L^0_N(v_0)$ coincide for all 
$N\geqslant \wt{N}(v_0,A)$ and denote this ideal 
by $\c L^0(v_0)$. 

Let $\Gamma =\Gamma (A,v_0)$ be the set of all $\gamma =a_1p^{n_1}+\dots +a_sp^{n_s}$, 
where all $a_i\in A$, 
$0=n_1\geqslant n_2\geqslant \dots \geqslant n_s$, $1\leqslant s<p$. 
\medskip 

{\it Choise of parameters $\delta , r^*, N^*$}:
\medskip 

a) let  $\delta =\delta (A,v_0)>0$ be sufficiently small such that 
$v_0-\delta>\max\{\gamma\ |\ \gamma\in\Gamma , \gamma <v_0\}$, $p\delta <2v_0$ and    
$v_0-\delta\in\Z [1/p]$; 
\medskip  

b) let $r^*$ be such that $v_p(r^*)=0$ and $v_0-\delta <r^*<v_0$; 
\medskip 

c) let $N^*\in\N $ be such that $N^*\geqslant \wt{N}(v_0,A)+1$ and for $q=p^{N^*}$,  
we have $r^*(q-1)=b^*\in\N $ (note $v_p(b^*)=0$), 
$a^*=q(v_0-\delta )\in p\N$; 
\medskip 

d) note that if $q$ satisfies the conditions from c) 
then any its power $q^A$ with $A\in\N $ 
also satisfies these conditions; therefore, we can enlarge (if necessary) $q$ 
to obtain the following inequalities: 
$$r^*-(v_0-\delta )>\frac{r^*+p(v_0-\delta )}{q}\,,\ \ 
v_0-r^*>\frac{-r^*+\varphi _{(p)}(e_{(p)}v_0(p-1))}{q}$$

 All above constructions and choices were made in Subsection 3.1 of \cite{Ab3}, 
except the additional conditions $p\delta <2v_0$ and 
the second inequality in d).  In this 
inequality $\varphi _{(p)}$ and $e_{(p)}$ are the Herbrand function and, resp., 
the ramification index of the extension $\c K{(p)}/\c K$. Recall that 
$\c K(p)$ is a subfield of $\c K_{<p}$, fixed by $G(\c L(p))$ and $[\c K(p):\c K]<\infty $.
\medskip 

We need the auxiliary field extension 
$\c K'=\c K(r^*,N^*)$ of $\c K$ such that:

--- $[\c K':\c K]=q$;
\medskip 

--- the Herbrand function $\varphi _{\c K'/\c K}$ 
has only one edge point $(r^*,r^*)$;
\medskip 

--- $\c K'=k((t'))$, where $t=t^{\prime\,q}E(t^{\prime\,b^*})^{-1}$ 
with the Artin-Hasse exponential 
$E(X)=\exp (X+X^p/p+\dots +X^{p^n}/p^n+\dots )$.
\medskip 

The field $\c K'$ played very important role in our approach to 
the ramification filtration in \cite{Ab1, Ab2, Ab3, Ab8, Ab9, Ab10}. 
(Note that $\c K'/\c K$ is not a $p$-extension if $N^*>1$.) 

Adjust the notation from \cite{Ab3} to our situation 
by setting $\hat N=\wt{N}=N^*-1$ (in particular, $\wt{N}$ could be different 
from $\wt{N}(v_0,A)$ introduced earlier). 

Let $\hat e_{\c L}^{(0)}=\sum _{a\in A}t^{-a}\D _{a0}$ and 
$e_{\c L}^{\prime (q)}=
\sum _{a\in A}t^{\prime\,-aq}\D _{a0}$. 
(We follow maximally close the notation 
from \cite{Ab3}.) Clearly, the elements $\hat e_{\c L}^{(0)}$ and 
$e_{\c L}':=\sum _{a\in A}t^{\prime\,-a}\D _{a,-N^*}$ 
are analogs of our element $e$ introduced in Subsection \ref{S1.3} and 
$\sigma ^{N^*}e_{\c L}'=e_{\c L}^{\prime (q)}$. Note that both these elements belong to 
$\c L^0_{\c K'}=\c L(A)\otimes _k\c K'$ (for $\hat e_{\c L}^{(0)}$ 
use that $t=t^{\prime q}E(t^{\prime b^*})^{-1}$). 

The technical result from \cite{Ab3} we are going to apply below  
deals with estimates in the 
envelopping algebra $\c A$ of $\c L^0$. We can describe this result as follows. 

Let $J$ be the augmentation ideal in $\c A$. Adjusting the notation from \cite{Ab3} 
note that (since we work with the case $M=0$)   
$O_1=\c K'$, $t_1=t'$, $O_0=k[[t']]$, $J_1=J_{\c K'}$  and $J_O=J\otimes O_0$. 

Use the map $\wt{\exp}$ from $\c L^0_{\c K'}$ to $J_{\c K'}\,\op{mod}\,J_{\c K'}^p$ 
from Subsection \ref{S3.3}. We obtain 
the elements 
$E_0=\wt{\exp}(\hat e^{(0)}_{\c L})$, $E'_0=\sigma ^{N^*}\wt{\exp}(e'_{\c L})$  
and (where we specified $m=1$) 
the element 
$\Phi _0^{(\wt{N})}=\Phi _{01}^{(\wt{N})}=\Phi _{11}\Phi _{21}$, 
cf. the first paragraph on p.890 in the proof of Lemma 2 in Subsection 3.10 of \cite{Ab3}. 
Explicit expressions for $\Phi _{11}$ and $\Phi _{21}$ from the second paragraph on p.890 
must be written in the following way  
$$\Phi _{11}=\wt{\exp}(e'^{(q)}_{\c L})\,\wt{\exp}(\sigma e'^{(q)}_{\c L})\,\dots\, 
\wt{\exp}(\sigma ^{\wt{N}}e'^{(q)}_{\c L})$$
$$\Phi _{21}=\wt{\exp}(-\sigma ^{\wt{N}}\hat e^{(0)}_{\c L})\,
\dots \,\wt{\exp}(-\sigma \hat e^{(0)}_{\c L})
\,\wt{\exp}(-\hat e^{(0)}_{\c L})\,.$$
(By misprint they appeared in \cite{Ab3} as the products of the same factors 
but taken in the opposite order.) Note that when adjusting the notation from \cite{Ab3} 
to our situation we have that $\c E_{0-\hat N}(a,n)=\sigma ^nE(a, t^{\prime b^*})$ and, therefore, 
$\c E_{0-\hat N}(a,n)\sigma ^n(t_1^{-qa}\c D_{a0})$ coincides with $\sigma ^n(t^{-qa}\c D_{a0})$. 

Using the properties $\alpha )-\gamma )$ from 
Subsection \ref{S3.3}  we obtain that  
$\Phi _0^{(\wt{N})}=\wt{\exp}(\phi _0^{(\wt{N})})$, where 
$\phi _0^{(\wt{N})}\in G(\c L^0_{\c K'})=G(\c L(A)\otimes _k\c K')$ is equal to 
$$\phi _0^{(\wt{N})}=e^{\prime (q)}_{\c L}
\circ (\sigma e^{\prime (q)}_{\c L})\circ 
\dots \circ (\sigma ^{\wt{N}}e^{\prime (q)}_{\c L})\circ 
(-\sigma ^{\wt N}\hat e_{\c L}^{(0)})
\circ \dots \circ (-\sigma\hat e_{\c L}^{(0)})
\circ (-\hat e_{\c L}^{(0)})\, .$$

Then the properties (a) and (b) of $\Phi _0^{(\wt{N})}$ from 
Proposition 9 of Subsection 3.9 in \cite{Ab3} imply the 
following properties of the element $\phi _0^{(\wt{N})}$, 
cf. the proposition from Subsection 3.10 of \cite{Ab3} 
(where $\c L_O:=\c L^0\otimes O_0$)

\begin{Prop} \label{P4.9} 
\medskip 
 a) $\phi _0^{(\wt N)}, \sigma\phi _0^{(\wt N)}\in\c L^0(v_0)_{\c K'}+\sum _{1\leqslant j<p}
t'^{-ja^*}C_j(\c L _{O})$;
\medskip 

b) $\phi _0^{(\wt N)}\circ \hat e^0_{\c L}\equiv 
e^{\prime (q)}_{\c L}\circ \sigma\phi _0^{(\wt N)}
\,\op{mod}\,\c L\c H_1^0$, where 
\medskip 

$\c L\c H_1^0=\c L^0(v_0)_{\c K'}+
t'^{q(b^*-a^*)}\sum _{1\leqslant j<p}t'^{-(j-1)a^*}C_j(\c L_{O})$. 
\end{Prop}

This technical result from \cite{Ab3} can be translated  into the covariant 
setting and the notation from this paper as follows. 
\medskip 

Let $v_0=c_0$. 

Consider the map $\Pi $ from $\c L^0$ to 
$\c L$ such that $\Pi _k(\c D_{an})=D_{an}$ for all 
$a\in A$ and $n\in\Z /N_0$ and for any 
$l_1,l_2\in\c L^0$, $\Pi ([l_1,l_2])=
[\Pi (l_2),\Pi (l_1)]$. 

Then the (ramification) ideal $\c L^0(v_0)$ is mapped to 
$\c L^{(c_0)}$. Essentially, $\Pi $ is a morphism of Lie algebras 
(where $\c L^0$ is taken with the opposite Lie structure) and it 
induces isomorphism of the appropriate quotients by $\c L^0(c_0)$ and 
$\c L^{(c_0)}$, respectively (use that by Proposition \ref{P4.5} all $D_{an}\in\c L^{(c_0)}_k$ if 
$a>(p-1)c_0$). 

Clearly, $\Pi _{\c K'}(\hat e_{\c L}^{(0)})\equiv e\,\op{mod}\,\c L^{(c_0)}_{\c K'}$ and 
$$\Pi _{\c K'}(e'_{\c L})\equiv e':=
\sum _{a\in \Z ^0(p)}t^{\prime -a}D_{a,-N^*}\,\op{mod}\,\c L_{\c K'}^{(c_0)}\, .$$ 
If $\phi _0:=\Pi _{\c K'}(\phi _0^{(\wt N)})$ then 
$\phi _0\equiv (-\phi )\circ 
(\sigma ^{N^*}\phi ')\,\op{mod}\,\c L^{(c_0)}_{\c K'}$,  
where we set 
$\phi =(\sigma ^{\wt N}e)\circ \dots \circ (\sigma e)\circ e$ and 
$\phi '=(\sigma ^{\wt N}e')\circ \dots \circ (\sigma e')\circ e'$. 

Let 
$$\c M_{\c K'}:=\sum _{1\leqslant j<p}t^{-c_0j}
 \c L(j)_{\m '}+\c L(p)_{\c K'}\, ,$$
where $\m '$ is the maximal ideal of the valuation ring $O_0$ 
of $\c K'$. Similarly, set 
$$\c M _{\c K'_{<p}}=\sum _{1\leqslant j<p}t^{-c_0j}
 \c L(j)_{\m '_{<p}}+\c L(p)_{\c K'_{<p}}$$
where $\c K'_{<p}$ and $\m '_{<p}$ are the analogs of 
$\c K_{<p}$ and $\m _{<p}$ for $\c K'$.  

Note that the above introduced modules 
$\c M_{\c K'}$ and $\c M_{\c K'_{<p}}$ are not obtained from  
$\c M$ and, resp.,  $\c M_{<p}$ when we replace $\c K$ by $\c K'$.  
Under such replacement we shall obtain  
from $\c M$ and $\c M_{<p}$ the following modules 
$$\c M':=\sum _{1\leqslant j<p}t^{\prime\,-c_0j}
 \c L(j)_{\m '}+\c L(p)_{\c K'}\, ,$$
$$\c M'_{<p}:=\sum _{1\leqslant j<p}t^{\prime\,-c_0j}
 \c L(j)_{\m '_{<p}}+\c L(p)_{\c K'_{<p}}\,.$$
However, $\sigma ^{N^*}\c M'\subset \c M_{\c K'}$ and 
$\sigma ^{N^*}\c M'_{<p}\subset \c M_{\c K'_{<p}}$. 

Now we use the special choice of involved parameters 
to deduce from above Proposition \ref{P4.9} the following proposition.

\begin{Prop} \label{P4.10} 
{\rm a)} $\phi _0, \sigma (\phi _0)\in \c M _{\c K'}+\c L_{\c K'}^{(c_0)}$;
\medskip 

{\rm b)} $e\circ \phi _0\equiv (\sigma \phi _0)\circ (\sigma ^{N^*}e')\,
\op{mod}\,\left (t^{c_0(p-1)}\c M_{\c K'}+\c L^{(c_0)}_{\c K'}\right )$
 \end{Prop}

\begin{proof} a) From the definition of $a^*$ it follows that 
$a^*=(c_0-\delta )q<c_0q$. Therefore, for $1\leqslant j<p$, 
$$t^{\prime\,-ja^*}\Pi (C_j(\c L_O))\subset t^{\prime\,-ja^*}O_{0}C_j(\c L)\subset t^{-jc_0}\m 'C_j(\c L)
\subset t^{-jc_0}\c L(j)_{\m '}\,.$$
 
For part b), we need 
for $1\leqslant j<p$,  
$$q(b^*-a^*)-(j-1)a^*>(p-j-1)qc_0\,.$$
This can be rewritten as 
$q(r^*-(c_0-\delta ))>r^*+(p-2)c_0-(j-1)\delta $.  
This follows from the inequality $p\delta<2v_0$ in a) and the first inequality in d) 
from the beginning of this 
subsection. 
\end{proof}

\subsection{Implication a)\ $\Leftrightarrow $ b), I}\label{S4.5} 

Suppose $h_{<p}$ is arithmetical. This means that $h^{(c_0)}=h_{<p}|_{\c K^{(c_0)}}
=h(p)|_{\c K^{(c_0)}}$ is (a unique) arithmetical lift of $h$. Then the appropriate 
$\bar c_1=c_1\,\op{mod}(\c M(p-1)+\c L^{(c_0)}_{\c K_{<p}})$ appears as the ``linear part of $c$'' 
if and only if  
$$(\id _{\bar{\c L}}\otimes h(p)^U)\bar f=c_1U\circ f\,
\op{mod}\,(\c M_{<p}U^2+t^{c_0(p-1)}\c M_{<p}U
+\c L_{\c K_{<p}}^{(c_0)}U)\, .$$

Consider the field $\c K'$ from Subsection \ref{S4.4}. This field is 
isomorphic to $\c K$ and this isomorphism can be extended to an isomorphism 
of $\c K_{<p}$ and its analog $\c K'_{<p}$. Let $f'\in\c M'_{<p}$ be such that 
$\sigma f'=e'\circ f'$. Then Proposition \ref{P4.10}\,b) implies the followimng lemma. 

\begin{Lem} \label{L4.11} 
$f'$ can be chosen in such a way that 
$$f\equiv \phi _0\circ \sigma ^{N^*}f'\,\op{mod}\,\left (t^{c_0(p-1)}\c M_{\c K'_{<p}}
+\c L_{\c K'_{<p}}^{(c_0)}\right )\,.$$
\end{Lem} 

\begin{proof} Let $g=(-f)\circ \phi _0\circ \sigma ^{N^*}f'\in\c M'_{\c K'_{<p}}$. 
Then by Proposition \ref{P4.10}b) 
$$\sigma g\equiv g\,\op{mod}\,(t^{c_0(p-1)}\c M_{\c K'_{<p}}+\c L^{(c_0)}_{\c K'_{<p}})\, .$$ 

This congruence implies that 
$$g\in\c L+t^{c_0(p-1)}\c M_{\c K'_{<p}}+\c L^{(c_0)}_{\c K'_{<p}}$$ 
(use that 
$\sigma $ is topologically nilpotent on $t^{c_0(p-1)}\c M_{\c K'_{<p}}\,\op{mod}\,\c L(p)_{\c K'_{<p}}$). 
Therefore, there is $l\in\c L$ such that 
$g\equiv l\,\op{mod}\,(t^{c_0(p-1)}\c M_{\c K'_{<p}}+\c L^{(c_0)}_{\c K'_{<p}})$ and we obtain 
our lemma with $f'$ replaced by $f'\circ (-l)$.  
\end{proof}

\subsection{Implication a)\ $\Leftrightarrow $ b), II}\label{S4.5a} 

Now note that $\c K\subset\c K'$ induces the embeddings 
$\c K_{<p}\subset \c K'\c K_{<p}\subset \c K'_{<p}$. 

Suppose $g\in\c I_{\c K}$ and $\hat g\in\c I$ is its arithmetical 
lift (i.e. for any finite field extension $\c E/\c K$, 
$v(\hat g|_{\c E})=\varphi ^{-1}_{\c E/\c K}(v(g))$). 
Introduce (similarly to 
$\c M_{\c K'_{<p}}$) 
$$\c M_{R_0}=\sum _{1\leqslant j<p}t^{-c_0j}\c L(j)_{\m _R}+\c L(p)_{R_0}\, .$$
Then Lemma \ref{L4.11} implies that modulo $t^{c_0(p-1)}\c M_{R_0}+\c L^{(c_0)}_{R_0}$ we have 
$$(\id _{\c L}\otimes g_{<p})f\equiv (-\id _{\c L}\otimes g)\phi 
\circ (\id _{\c L}\otimes g')\sigma ^{N^*}\phi '
\circ (\id _{\c L}\otimes g'_{<p})\sigma ^{N^*}f'
\,.$$ 
Here 
$g _{<p}:=\hat g|_{\c K _{<p}}$, $g'_{<p}:=\hat g|_{\c K '_{<p}}$ and $g':=\hat g |_{\c K'}$ are all  
arithmetical over $\c K$. 
(Recall, $\phi _0\equiv (-\phi )\circ (\sigma ^{N^*}\phi ')$, cf. Subsection \ref{S4.3}.)

\begin{Prop} \label{P4.12}
Suppose $v(g)=c_0$. Then 
\medskip 

{\rm a)}\ $(\id _{\c L}\otimes g'_{<p}-\id _{\c K'_{<p}})
\sigma ^{N^*}f'\in t^{c_0(p-1)}\c M_{R_0}$;
\medskip 

{\rm b)}\ $(\id _{\c L}\otimes g'-\id _{\c K'})
\sigma ^{N^*}\phi '\in t^{c_0(p-1)}\c M_{R_0}$. 
\end{Prop}

\begin{proof}  Let $\c K'(p)$ be an analogue of $\c K(p)$ for $\c K'$.  \newline 
If we set  $g'_{(p)}=\hat g|_{\c K'(p)}$ then it is arithmetical over $\c K$ and 
$$v(g'_{(p)})=\varphi ^{-1}_{(p)}(\varphi ^{-1}_{\c K'/\c K}(c_0))=
\varphi ^{-1}_{(p)}(r^*+q(c_0-r^*))>e_{(p)}c_0(p-1)\, ,$$
cf. item d) in Subsection \ref{S4.3}. This means that 
for any $a\in \c K'(p)$, 
\begin{equation} \label{E4.1}
g'_{(p)}(a)-a\in at^{\prime c_0(p-1)}R\,.
\end{equation} 

Now notice that $f'\,\op{mod}\,\c L(p)_{\c K'_{<p}}\in\bar{\c L}_{\c K'(p)}$, 
cf. Subsection \ref{S1.3}. This implies that $f'\in\c M_{\c K'{(p)}}+\c L(p)_{\c K'_{<p}}$, 
where $\c M_{\c K'(p)}$ is an analogue of $\c M_{\c K'_{<p}}$ for $\c K'(p)$. 
Now the property \eqref{E4.1} implies that 
$$(\id _{\c L}\otimes g'_{<p})f'-f'\in t^{\prime c_0(p-1)}\c M'_{R_0}+\c L(p)_{R_0}=
t^{\prime c_0(p-1)}\c M'_{R_0}\, ,$$
where $\c M'_{R_0}:=
\sum _{1\leqslant j<p}t^{\prime -c_0j}\c L(j)_{\m _R}+\c L(p)_{R_0}$, 
and we obtain a) by applying $\sigma ^{N^*}$. 

For similar reasons, 
$$v(g')=r^*+q(c_0-r^*)>\varphi _{(p)}(e_{(p)}c_0(p-1))
\geqslant c_0(p-1)$$
(we use that $\varphi _{(p)}(e_{(p)}x)\geqslant x$ for any $x\geqslant 0$), 
and then for any $a\in\c K'$, 
$$g'(a)-a\in at^{\prime c_0(p-1)}R\, .$$
This implies  
$$(\id _{{\c L}}\otimes g')e'-e'\in t^{\prime c_0(p-1)}\c M'_{R_0}\, , \ \   
(\id _{\c L}\otimes g')\phi '-\phi '\in t^{\prime c_0(p-1)}\c M'_{R_0}\, ,$$  
and we obtain b) by applying  $\sigma ^{N^*}$. 
\end{proof}

\begin{Cor} \label{C4.13} 
Suppose $g\in\c I_{\c K}$, $v(g)=c_0$ and $g_{<p}$ is a lift of $g$ to $\c K_{<p}$.  Then the 
following conditions are equivalent:
\medskip 

{\rm a)}\ $g_{<p}$ is arithmetical lift of $g$;
\medskip 

{\rm b)}\ $(\id _{\c L}\otimes g_{<p})f\equiv (-\id _{\c L}\otimes g)\phi\circ \phi\circ f\,\op{mod}\,
(t^{c_0(p-1)}\c M_{R_0}+\c L^{(c_0)}_{R_0})\,.$ 
\end{Cor}

\begin{proof}
 Assume that $g_{<p}$ is arithmetical. We can assume that $g_{<p}=g'_{<p}|_{\c K_{<p}}$ 
 where $g'_{<p}\in\c I_{\c K'_{<p}}$ is arithmetical lift of $g$. 
 Then Lemma \ref{L4.11} and Proposition \ref{P4.12} imply that modulo 
 $t^{c_0(p-1)}\c M_{R_0}+\c L^{(c_0)}_{R_0}$
 $$(\id _{\c L}\otimes g_{<p})f\equiv (-\id _{\c L}\otimes g)\phi 
 \circ (\id _{\c L}\otimes g')\sigma ^{N^*}\phi '\circ (\id _{\c L}\otimes g'_{<p})\sigma ^{N^*}f'$$
 $$\equiv (-\id _{\c L}\otimes g)\phi\circ \phi \circ \phi _0\circ \sigma ^{N^*}f'
 \equiv (-\id _{\c L}\otimes g)\phi\circ \phi\circ f\, ,$$
 and we obtained b). 
 
 Assume that b) holds. If $g^o_{<p}\in\c I_{\c K_{<p}}$ is an arithmetical lift of $g$ 
 then we can apply b) and obtain 
 $$(\id _{\c L}\otimes g_{<p})f\equiv (\id _{\c L}\otimes g^o_{<p})f\,
 \op{mod}\,(t^{c_0(p-1)}\c M_{R_0}+\c L^{(c_0)}_{R_0})\,.$$
 On the other hand, there is $l\in G(\c L)$ such that $g_{<p}=g^o_{<p}\eta _0^{-1}(l)$. 
 Then the above congruence implies that 
 $$l\in t^{c_0(p-1)}\c M_{R_0}+\c L^{(c_0)}_{R_0}
 \subset \m _R\c L_R+\c L^{(c_0)}_{R_0}\,.$$
 But then $l\in \left (\m _R\c L_R+\c L_{R_0}^{(c_0)}\right )|_{\sigma =\id }=\c L^{(c_0)}$. 
 Therefore, $g_{<p}$ is also arithmetical. 
\end{proof}

\subsection{Implication a)\ $\Leftrightarrow $ b), III}\label{S4.6}  

Let $1\leqslant n<p$. 
Applying Corollary \ref{C4.13} to $g=h^n$ and its lift $h^n_{<p}$ we obtain 
that the following two properties are equivalent:
\medskip 

$\bullet $\ $h^n_{<p}$ is arithmetical;
\medskip 

$\bullet $\ 
$(\id _{\c L}\otimes h^n_{<p})f=c(n)\circ (A^n\otimes\id _{\c K_{<p}})f$, where 
$(A^n-\id _{\c L})\c L\subset\c L^{(c_0)}$ and 
$c(n)\equiv (-\id _{\c L}\otimes h^n)\phi \circ\phi \,\op{mod}\,\c M(p-1)+\c L^{(c_0)}_{\c K}$. 

Clearly, the first condition holds if and only if $h_{<p}$ is arithmetical. 

The second condition means that $(A-\id _{\c L})\c L\subset\c L^{(c_0)}$ 
and 
$$c(U)\equiv (-\id _{\c L}\otimes h^U)\phi \circ \phi \, \op{mod}\,\c M(p-1)+\c L_{\c K}^{(c_0)}\, .$$ 
The both parts of the last  congruence can be recovered 
uniquely by their linear terms: this is obvious for $(-\id _{\c L}\otimes h^U)\phi \circ \phi $ 
and was explained in Subsection \ref{S3.5} for $c(U)$. Therefore, the equivalence of 
a) and b) will be proved if we show that the linear part of $(-\id _{\c L}\otimes h^U)\phi \circ \phi $ 
takes prescribed value from part b) of our theorem. 

Recall that 
$\phi =(\sigma ^{\wt N}e)\circ \dots \circ (\sigma e)\circ e$.

Apply identites \eqref{E3.5} and \eqref{E3.6} from Subsection \ref{S3.2}, 
use the definition of the elements $\c F^0_{\gamma ,-N}\in\c L_k$ 
from Subsection \ref{S1.4} and the abbreviation $d_h:=d(\id _{\c L}\otimes h^U)$ 
to obtain the following congruences modulo $U^2$:  
$$e+d_he\equiv e\circ 
\left (\sum _{k\geqslant 1}(1/k!)[\dots [d_he,
\underset{k-1\text{ times }}{\underbrace{e],
\dots ,e}}]\right )\qquad\qquad $$
$$\qquad\qquad\qquad\qquad \equiv 
e\circ \left (-U\sum _{\gamma>0\, ,j\geqslant 0} A_j(h)
\c F^0_{\gamma ,0}\,t^{-\gamma +c_0+pj}\right )$$

Similarly, 
$$\sigma e+\sigma d_he\equiv 
\sigma e\circ \left (\sum _{k\geqslant 1}(1/k!)[\dots [\sigma d_he,
\underset{k-1\text{ times }}{\underbrace{\sigma e],
\dots ,\sigma e}}]\right )$$
then 
$$(\sigma e+\sigma d_he)\circ e\equiv $$
$$(\sigma e)\circ e\circ \left (\sum_{\substack{k_0\geqslant 1\\
k_1\geqslant 0}}\frac{1}{k_0!k_1!}
[\dots [\sigma d_he,\underset{k_0-1\text{ times }}
{\underbrace{\sigma e],\dots ,
\sigma e}}],\underset{k_1\text{ times }}
{\underbrace{e],\dots ,e}}]\right )$$

$$=(\sigma e)\circ e\circ \left (-U
\sum _{\substack{\gamma >0 \\ j\geqslant 0}}
\sigma (A_j(h)\c F^0_{\gamma ,-1}t^{-\gamma +c_0+pj})\right ) $$
and taking above formulas together we obtain 

$$(\sigma e+\sigma d_he)\circ (e+d_he)\equiv (\sigma e)\circ e\circ \left (-U
\sum _{\substack{\gamma >0 \\ j\geqslant 0}}\sum _{0\leqslant i\leqslant 1}
\sigma ^i(A_j(h)\c F^0_{\gamma ,-i}t^{-\gamma +c_0+pj})\right ) $$

We can continue similarly to obtain that 

$$(\id\otimes h^U)\phi \equiv \phi \circ 
\left (-U\sum _{\substack{\gamma >0\\ j\geqslant 0}}
\sum _{0\leqslant i\leqslant \wt{N}}
\sigma ^i(A_j(h)\c F^0_{\gamma ,-i}t^{-\gamma +c_0+pj})\right )
\,\op{mod}\,U^2$$

So, the linear term takes the prescribed value and the statements a) and b) of theorem are equivalent.

\subsection{The end of proof of Theorem \ref{T4.8}} \label{S4.7} 

Obviously, b) implies c). 

Suppose a lift $h_{<p}$ has ingredients $c_1$ and 
$\{V_{a0}\ |\ a\in\Z ^0(p)\}$ and $c_1(0)$ satisfies the condition c) of our theorem. 
Take the maximal $1\leqslant s_0\leqslant p$ such that $h_{<p}|_{\c K_{<p}^{G(\c L(s_0))}}$ 
is arithmetical. If $s_0=p$ then $h(p)$ is arithmetical and this implies that $h_{<p}$ is arithmetical. 

Suppose $s_0<p$. 

Let $h^o_{<p}$ be 
some arithmetical lift of $h$ with the appropriate ingredients $c_1^o$ and $\{V_a^o\ |\ a\in\Z ^0(p)\}$. 
Therefore, 
$$c_1\equiv c_1^o\,\op{mod}\,\c L^{(c_0)}_{\c K}+\c L(s_0)_{\c K}\, .$$
Note that for 
all $a\in\Z ^0(p)$, $V_{a0}\in \c L^{(c_0)}_k+\c L(s_0)_k$ and 
$V_a^o\in\c L_k^{(c_0)}$. Then recurrent relation 
\eqref{E3.4} (considered at the $s_0$-th step) implies that 
$$\sigma c_1-c_1+\sum _{a\in\Z ^0(p)}t^{-a}V_{a0}\equiv 
\sigma c^o_1-c^o_1\,\op{mod}\,\c L^{(c_0)}_{\c K}+\c L(s_0+1)_{\c K}\, .$$
Therefore, by Lemma \ref{L2.2}b), all $V_{a0}\in \c L^{(c_0)}_{k}+\c L(s_0+1)_{k}$ and  
$$c_1-c_1^o\equiv c_1(0)-c_1^o(0)
\,\op{mod}\,\c L_{\c K}^{(c_0)}+\c L(s_0+1)_{\c K}\,.$$ 
So, if $c_1(0)$ satisfies c) then $c_1\equiv c_1^o\,\op{mod}\,\c L_{\c K}^{(c_0)}+
\c L(s_0+1)_{\c K}$ and the restriction $h_{<p}|_{\c K_{<p}^{G(\c L(s_0+1))}}$
is arithmetical. The contradiction. 
Theorem \ref{T4.8} is completely proved.

\section{Explicit calculations in $L_h$} \label{S5} 

In this Section we apply the above techniques to 
study the lifts $h(p)=h_{<p}|_{\c K(p)}$. In Subsection \ref{S4} we studied 
the properties of $h_{<p}|_{\c K^{(c_0)}}$ and that was sufficient to 
characterize the property of $h_{<p}$ to be arithmetical over $\c K$. 
If we want to describe completely the structure of the Lie algebra $L_h$ we need 
to study the invariants $\ad\,h(p)$ and $c_1$ of $h(p)$.    

Suppose $h(p)$ is given, as earlier, via 
$$(\id _{\bar{\c L}}\otimes h(p))\bar f=
\bar c\circ (\op{Ad}\,h(p)\otimes \id _{\c K(p)})\bar f$$ 
with the appropriate $\bar c\in\c M\,\op{mod}\,\c M(p-1)$. Then the relevant elements 
$c_1\in\c L_{\c K}\,\op{mod}\,\c M(p-1)$ and 
$V_{a0}=\ad h(p)(D_{a0})\in\bar{\c L}_k=\c L_k/\c L(p)_k$, $a\in\Z ^0(p)$, 
satisfy recurrent relation \eqref{E3.4}. 
This allows us to proceed  
from solutions $(c_1, \sum _at^{-a}V_{a0})$ obtained modulo $\c M(p-1)+\c L(s)_{\c K}$ 
to the appropriate ``more precise'' solutions modulo $\c M (p-1) +\c L(s+1)_{\c K}$, 
for all $1\leqslant s<p$.  

As earlier, let  $c_1=\sum _{m\in\Z }c_1(m)t^m$, where 
all $c_1(m)\in\bar{\c L}_k$. Introduce 
$c_1^+=\sum _{m>0}c_1(m)t^{m}$ and 
$c_1^-=\sum _{m<0}c_1(m)t^m$. Then 
$$c_1=c_1^-+c_1(0)+c_1^+\, .$$

In this Section we find ``precise'' formulas for $c^+$, $c(0)$ 
and $V_0=\alpha _0^{-1}V_{00}=\op{ad}h(p)(D_0)$. 
When choosing $c_1^+$ we use the operator $\c S$ from 
Subsection \ref{S2.2}. 
When choosing $c_1(0)$ we must act more carefully. 
The expression for $\op{ad}h(p)(D_0)$  is given in Proposition \ref{P5.4} 
below.

It would be very interesting to resolve completely recurrent relation \eqref{E3.4} 
and to find reasonably compact formulas for $c_1^-$ and all the elements 
$V_{a0}=\op{ad}\,h(p)(D_{a0})$, $a\in\Z ^+(p)$. This would generalize explicit calculations from 
Subsection \ref{S3.6}. Some steps in this direction were made recently 
by K.\,McCabe (PhD Thesis, Durham University).

\subsection{Explicit formula for  $c_1^+$} \label{S5.1} 

Consider all $(\bar a,\bar n)=(a_1,n_1,\dots ,a_s,n_s)$ such that $1\leqslant s<p$, 
all $a_i\in\Z ^0(p)$ and $n_1\geqslant n_2\geqslant\dots \geqslant n_s=0$. 

Set $\gamma (\bar a,\bar n)=a_1p^{n_1}+a_2p^{n_2}+\dots +a_sp^{n_s}$. 

Set $D_{(\bar a,\bar n)}=[\dots [D_{a_1n_1},D_{a_2n_2}],\dots ,D_{a_sn_s}]$ 
and use the weight function 
$\op{wt}(D_{(\bar a,\bar n)})=\op{wt}(D_{a_1n_1})+\dots +\op{wt}(D_{a_sn_s})$ 
from Subsection \ref{S2.4}. 

Denote by $\delta ^+(c_0)$ the minimum of all positive values of 
$$(c_0+pj)-p^{-n_1}\gamma (\bar a,\bar n)\, ,$$ 
where $j\geqslant 0$ and $(\bar a,\bar n)$ runs over the set of 
all above vectors with additional condition 
$\op{wt}(D_{(\bar a,\bar n)})<p$. 

Finally, let $N^+(c_0)=\min \{n\geqslant 0\ |\ p^n\delta ^+(c_0)\geqslant c_0(p-1)\}$.

Relation \eqref{E3.4} implies that modulo $\c M (p-1)$ 
\begin{equation} \label{E5.1}
 \sigma c_1^+-c_1^+\equiv 
\end{equation} 
$$-\sum_{\substack{k\geqslant 1\\ j\geqslant 0}}\frac{1}{k!}A_j(h)
\sum _{a_1,\dots ,a_k}t^{c_0+pj-(a_1+\dots +a_k)}
[\dots [a_1D_{a_10},D_{a_20}],\dots ,D_{a_k0}]
$$
$$-\sum _{m,k\geqslant 1}\frac{1}{k!}\sum _{a_1,\dots ,a_k}t^{pm-(a_1+\dots +a_k)}
[\dots [\sigma c_1(m),D_{a_10}],\dots ,D_{a_k0}]\,.$$
In both above sums the indices $a_1,\dots ,a_k$ run over 
$\Z ^0(p)$ with the restrictions $a_1+\dots +a_k<c_0+pj$ for the first sum 
and $a_1+\dots +a_k<pm$ for the second sum. 

Note that $c_1^+\,\op{mod}\,\c M(p-1)$ is defined uniquely by \eqref{E5.1}.  
Of course, it is obtained by applying the operator $\c S$ 
from Subsection \ref{S2.2} to the RHS of 
the above congruence. 

\begin{definition}
For $n^*\geqslant n_*$, let $\c F_{\gamma ,[n^*,n_*]}^0$ be the partial sum 
of $\sigma ^{n^*}\c F^0_{\gamma ,n_*-n^*}$ containing  
only the terms   
$[\dots [D_{a_1n_1},D_{a_2n_2}],\dots ,D_{a_sn_s}]$, such that  
$n_1=n^*$ and $n_s=n_*$. In other words, we keep only the terms such that 
$n^*=\op{max}\{n_i\ |\ 1\leqslant i\leqslant s\}$ and 
$n_*=\op{min}\{n_i\ |\ 1\leqslant i\leqslant s\}$. 
 \end{definition}

\begin{Prop} \label{P5.1} Let $N^0\in \N$ be such that $N^0\geqslant N^+(c_0)-1$. Then 
$$c_1^+\equiv  
\sum _{\substack{j\geqslant 0\\ 0\leqslant n\leqslant N^0}}
\sum _{\gamma <c_0+pj}\sigma ^n(A_j(h)
\c F^0_{\gamma ,-n})t^{p^n(c_0+pj-\gamma )}
\op{mod}\,\c M (p-1)\,.$$
\end{Prop}

\begin{remark} 
 The RHS of the above congruence does not depend 
on a choice of $N^0\geqslant N^+(c_0)-1$. 
\end{remark}

\begin{proof}[Proof of Proposition] 
Prove proposition by establishing the formula for $c_1^+$ modulo 
$\c M (p-1)+C_s(\c L_{\c K})$ by induction on $1\leqslant s\leqslant p$. 

If $s=1$ there is nothing to prove. 

Suppose $s<p$ and proposition is proved 
modulo $\c M (p-1)+C_s(\c L_{\c K})$. Prove that modulo 
$\c M (p-1)+C_{s+1}(\c L_{\c K})$ 
\begin{equation} \label{E5.2} \sigma c_1^+-c_1^+
\equiv -\sum _{\substack{j\geqslant 0\\ 0\leqslant n\leqslant N^0}}
\sigma ^n(A_j(h))
\sum _{\gamma <c_0+pj}\c F^0_{\gamma ,[n,0]}t^{p^n(c_0+pj-\gamma )}\, .
\end{equation}

Note that for $n=0$, 
$$\c F^0_{\gamma ,[0,0]}=\sum _{a_1,\dots ,a_k}
\frac{1}{k!}[\dots [a_1D_{a_10},D_{a_20}],
\dots ,D_{a_k0}]$$ 
and for $n>0$, 
$$\c F^0_{\gamma ,[n,0]}=
\sum _{\substack{k\geqslant 1, \gamma '>0\\ a_1,\dots ,a_k}}
\frac{1}{k!}[\dots [\sigma ^n\c F^0_{\gamma ',-(n-1)},D_{a_10}],\dots ,D_{a_k0}]\,.$$

In both sums the indices $a_1,\dots ,a_k$ run over 
$\Z ^0(p)$ with the restrictions 
$a_1+\dots +a_k=\gamma $ in the first case and 
$p^n\gamma '+a_1+\dots +a_k=p^n\gamma $ in the second case. 

The first formula allows us to identify the first line of the RHS in \eqref{E5.1} 
with the part of \eqref{E5.2} which corresponds to $n=0$. 
The second formula  
allows us to rewrite modulo $C_{s+1}(\c L_{\c K})$ 
the second line of the RHS in \eqref{E5.1} (under inductive assumption)  
as the part of \eqref{E5.2} which corresponds to $n>0$.

Denote by $-\Omega $ the right-hand side of \eqref{E5.2}. 
Applying $\c S$ we obtain that  
modulo $\c M (p-1)+C_{s+1}(\c L_{\c K})$ it holds  
$c_1^+\equiv \sum _{m\geqslant 0}\sigma ^m\Omega $ 
and  

$$c_1^+\equiv \sum _{n,m,j}\sum _{\gamma <c_0+pj}
\sigma ^{n+m}\left (A_j(h)\c F^0_{\gamma ,[0,-n]}\right )t^{p^{n+m}(c_0+pj-\gamma )}
\,.$$
Modulo $\c M (p-1)$ we can assume that 
$n_1=n+m\leqslant N^0$ 
and rewrite the above RHS as 
$$\sum _{\gamma ,j,n_1}\sigma ^{n_1}\left (A_j(h)\sum _{0\leqslant m\leqslant n_1}
\c F^0_{\gamma ,[0,-m]}\right )t^{p^{n_1}(c_0+pj-\gamma )}\,.$$
It remains to note that $\sum _{0\leqslant m\leqslant n_1}
\c F^0_{\gamma ,[0,-m]}=\c F^0_{\gamma ,-n_1}$. 

The proposition is proved. 
\end{proof} 

\subsection{Explicit calculations with $c_1(0)$} \label{S5.2} 

By \eqref{E3.4} we have modulo $\c L(p)_{k}$ that (here $V_0=\alpha _0^{-1}V_{00}=\ad h(p)(D_0)$)
\begin{equation} \label{E5.3} 
 \sigma c_1(0)-c_1(0)+\alpha _0V_0\equiv \qquad\qquad\qquad\qquad\qquad\qquad\qquad  
\end{equation}
$$-\sum _{\substack{k\geqslant 1\\ j\geqslant 0}}
\sum _{a_1,\dots ,a_k}\frac{1}{k!}A_j(h)[\dots [a_1D_{a_10},D_{a_20}],\dots ,D_{a_k0}]$$
$$-\sum _{\substack{k,m\geqslant 1\\ a_1,\dots ,a_k}}\frac{1}{k!} 
[\dots [\sigma c^+_1(m),D_{a_10}],\dots ,D_{a_k0}]$$
$$-\sum _{k\geqslant 2}\frac{1}{k!}[\dots [V_0,\underset{k-1\ \text{times}} 
{\underbrace{D_{00}],\dots ,D_{00}}}]$$
$$\qquad\qquad -\sum _{k\geqslant 1}\frac{1}{k!}[\dots [\sigma c_1(0),
\underset{k \ \text{times}}{\underbrace{D_{00}],\dots ,D_{00}}}]$$

In the first and second sums the indices $a_i$ run over $\Z ^0(p)$ 
with the restrictions $a_1+\dots +a_k=c_0+pj$ in the first case and 
$a_1+\dots +a_k=pm$ in the second case. 

\begin{definition}
 For $n\geqslant 0$, denote by $\c F^+_{\gamma ,[n,0]}$ the partial sum 
of $\c F^0_{\gamma ,[n,0]}$ which contains only the terms with  
$[\dots [D_{a_1n_1},D_{a_2n_2}],\dots ,D_{a_sn_s}]$ such that 
if for some $i_1\geqslant 0$, $0=n_s=\dots =n_{s-i_1}<n_{s-i_1-1}$ 
then at least 
one of $a_s,\dots ,a_{s-i_1}$ is not zero. 
\end{definition}

Fix $N^0\geqslant N^+(c_0)-1$. 

\begin{Lem}\label{L5.2} The sum of the first two lines in  
the RHS of \eqref{E5.3} equals 
$$-\sum _{\substack {0\leqslant n\leqslant N^0\\ j\geqslant 0}}\sigma ^n(A_j(h))
\c F^+_{c_0+pj,[n,0]}$$ 
\end{Lem}

\begin{proof} For the first line use the above definition with $n=0$. 

For the second line use the following 
identity 
$$\sum _{\substack{k\geqslant 1\\a_1,\dots ,a_k}}(1/k!)
[\dots [\sigma ^{n}\c F^0_{\gamma ,-n+1}, D_{a_10}],\dots ,D_{a_k0}]=
\c F^+_{c_0+pj,[n,0]}$$
where $n\in\N $, $\gamma <c_0+pj$ and $a_1,\dots ,a_k$ run over 
$\Z ^0(p)$ such that $a_1+\dots +a_k=p^{n}(c_0+pj-\gamma )$.  
 \end{proof}
  
Introduce the operators 
$$G_0=\wt{\exp}\,(\alpha _0\,\ad D_{0}),\ \ \  
F_0=E_0(\alpha _0\,\op{ad}D_{0})$$ 
on $\c L_k$ (recall that $E_0(x)=(\wt{\exp}x-1)/x$). Note that for $l\in\c L_k$, 
$$F_0(l)=\sum _{k\geqslant 1}\frac{\alpha _0^{k-1}}{k!}
[\dots [l,\underset{k-1\ \text{times}}
{\underbrace{D_{0}],\dots ,D_{0}}}],\ \ 
G_0(l)=\sum _{k\geqslant 0}\frac{\alpha _0^k}{k!}[\dots [l,
\underset{k\ \text{times}}{\underbrace{D_{0}],\dots ,D_{0}}}]\, .$$

With this notation we can rewrite \eqref{E5.3} in the following form 
$$(G_0\sigma -\id )c_1(0)+F_0(\alpha _0V_0)=
-\sum _{j\geqslant 0}
\sum _{0\leqslant i\leqslant N^0}\sigma ^i(A_j(h))\c F^+_{c_0+pj,[i,0]}$$

\begin{Lem} \label{L5.3}
 Suppose $l(\alpha ,\gamma )=
\sum _{0\leqslant i\leqslant N^0}\sigma ^i(\alpha\c F^0_{\gamma ,-i})$, 
where $\alpha\in k$. Then 
$$(G_0\sigma -\id )l(\alpha ,\gamma )=
-\sum _{0\leqslant i\leqslant  N^0}\sigma ^i(\alpha )\c F^+_{\gamma ,[i,0]}
+G_0\sigma ^{N^0+1}(\alpha\c F^0_{\gamma ,-N^0})$$
\end{Lem}

\begin{proof}[Proof of lemma] 
Directly from definitions it follows for $i\geqslant 0$, that 
$(G_0\sigma )(\sigma ^i\c F^0_{\gamma ,-i})=
\sigma ^{i+1}\c F^0_{\gamma ,-(i+1)}-\c F^+_{\gamma ,[i+1,0]}$. 
Therefore, 
$$(G_0\sigma )l(\alpha ,\gamma )=
\sum _{1\leqslant i\leqslant N^0+1}\sigma ^{i}(\alpha\c F^0_{\gamma ,-i})
-\sum _{1\leqslant i\leqslant N^0+1}(\sigma ^i\alpha )\c F^+_{\gamma ,[i,0]}$$
$$=l(\alpha ,\gamma )-\sum _{0\leqslant i\leqslant N^0}(\sigma ^i\alpha )\c F^+_{\gamma ,[i,0]}
+\sigma ^{N^0+1}(\alpha )\left (-\c F^+_{\gamma ,[N^0+1,0]}+
\sigma ^{N^0+1}\c F^0_{\gamma ,-(N^0+1)}\right )\, .$$
It remains to note that $-\c F^+_{\gamma ,[N^0+1,0]}+\sigma ^{N^0+1}\c F_{\gamma ,-(N^0+1)}=
G_0\sigma ^{N^0+1}\c F^0_{\gamma ,-N^0}$. 
\end{proof} 

Summarize the above calculations.

\begin{Prop} \label{P5.4} Suppose $h(p)$ is a lift of $h$ to $\c K(p)$ 
with the ``linear ingredient''   
$c_1=c_1^-+c(0)+c_1^+$, $V_0=(\op{ad}\,h(p))D_0$ and $N^0\geqslant N^+(c_0)-1$. Then    
$$c_1(0)=c^0+
\sum _{\substack{0\leqslant i\leqslant N^0 \\ j\geqslant 0}}
\sigma ^i(A_j(h)\c F^0_{c_0+pj,-i})\in\bar{\c L}_k\, ,$$  
where  $c^0\in\bar{\c L}_k$ and $V_0\in\bar{\c L}$ are arbitrary solutions of the equation    
\begin{equation} \label{E5.4} 
(G_0\sigma -\id )c^0+F_0(\alpha _0V_0)=-G_0\sigma ^{N^0+1}\Omega ^0\,,
\end{equation} 
with $\Omega ^0=\sum _{j\geqslant 0}
A_j(h)\c F^0_{c_0+pj,-N^0}.$ 
\end{Prop}

\begin{remark}  a) Modulo $[\bar{\c L}_k,D_0]$  
equation \eqref{E5.4} looks like 
$$(\sigma -\id )c^0+\alpha _0V_0\equiv -\sigma ^{N^0+1}\Omega ^0\, ,$$ 
and, therefore, admits explicit solutions 
(use the operators $\c R$ and $\c S$ from Subsection \ref{S2.2} and Lemma \ref{L2.2}b). 
This implies  $V_0=\op{ad}\,h(p)(D_0)$ is congruent modulo $[\c L_k,D_0]$ to 
(recall that $|k|=p^{N_0}$)
$$-(\id _{\c L}\otimes \op{Tr}_{k/\F _p})(\sigma ^{N^0+1}\Omega ^0)
\equiv -\sum _{0\leqslant n<N_0}\sigma ^n(\Omega ^0);$$ 
\medskip 

 b)  if $k=\F _p$ then \eqref{E5.4} can be solved: here 
 $\sigma =\id $ and we can set $c^0=-\Omega ^0(=-\sigma ^{N^0}\Omega ^0)$; 
 this implies the existence of a lift $h(p)$ such that the Demushkin relation 
 appears in the form 
$$\op{ad}h(p)(D_0)+F_0^{-1}(\Omega ^0)=0\, ;$$  
\medskip 

c) the appearance of operators $F_0$ and $G_0$ in the LHS of \eqref{E5.4} is related to 
a ``bad influence'' of the generators $D_{0n}$; 
this influence can be seen already at the explicit expressions of the elements 
$\c F^0_{\gamma ,-N}$ from Subsection \ref{S1.4}: the elements of the form $D_{0n}$ 
don't contribute to $\gamma $ and therefore can appear with almost no restrictions 
in all terms of $\c F^0_{\gamma ,-N}$; e.g. if $a\in\Z ^0(p)$ then 
$\c F^0_{a,-N}$ contains together with the linear term $aD_{a0}$ 
all terms from    
$(\sigma ^{-N}G_0)(\sigma ^{-N+1}G_0)\dots (\sigma ^{-1}G_0)F_0(aD_{a0})$.
\end{remark} 

Finally note that Proposition \ref{P5.4} allows us to control 
arithmetic lifts of $h$ if we require also that  
$N^0\geqslant\wt{N}(c_0)$, cf. Subsection \ref{S1.4} for the definition of $\wt{N}(c_0)$. 

\begin{Prop} \label{P5.5} 
Suppose $N^0\geqslant\max\{N^+(c_0)-1, \wt{N}(c_0)\}$. Then 
\eqref{E5.4} admits a solution $c^0\in\bar{\c L}^{(c_0)}_k$ and $V_0\in\bar{\c L}^{(c_0)}$  
and the corresponding lift $h(p)$ is arithmetical. 
\end{Prop}

\begin{proof}
For $n\geqslant 1$, define the triples $(X_n,Y_n,Z_n)$ by the following recurrent relations:
\medskip 

$Z_1=-G_0\sigma ^{N^0+1}\Omega ^0$, \ \ $X_n=\c S(Z_n)$, \ \ $Y_n=\alpha _0^{-1}\c R(Z_n)$ 
\medskip 

$Z_{n+1}=
-(G_0-\id )\sigma X_n-(F_0-\id )(\alpha _0Y_n)\, .$ 
\medskip 

Then is it easy to see that:
\medskip 

1)\ for all $n$, $Z_n, X_n\in (\ad ^{n-1} D_0)\bar{\c L}_k^{(c_0)}$ and 
$Y_n\in (\ad ^{n-1} D_0)\bar{\c L}^{(c_0)}$;
\medskip 

2)\ $c^0=X_1+\dots +X_{p-1}$ and $V_0=Y_1+\dots +Y_{p-1}$ satisfy \eqref{E5.4}. 
\medskip 

Indeed, for any ideal $\c L'$ in $\bar{\c L}$ and $n\geqslant 1$, 
the operators $\c R$ and $\c S$ map $(\ad ^{n-1}D_0)\c L'_k$ to itself and 
the operators $G_0-\id $ and $F_0-\id $ map $(\ad ^{n-1}D_0)\c L'_k$ 
to $(\ad ^{n}D_0)\c L'_k$. This proves the first property. 

As for the second property, proceed as follows: 

$$\sum _{1\leqslant i<p}(G_0\sigma -\id )X_i+\sum _{1\leqslant i<p}F_0(\alpha _0Y_i)$$
$$=\sum _{1\leqslant i<p}(G_0 -\id )\sigma X_i+\sum _{1\leqslant i<p}(F_0-\id )(\alpha _0Y_i)
+\sum _{1\leqslant i<p}\left ((\sigma -\id )X_i+\alpha _0Y_i\right )$$
$$=-(Z_2+\dots +Z_{p-1}+Z_p)+(Z_1+Z_2+\dots +Z_{p-1})=Z_1\,.$$

Finally Theorem \ref{T4.8}c) implies that the appropriate lift $h(p)$ is arithmetical.
\end{proof}

\medskip

\section{Applications to the mixed characteristic case} \label{S6} 

Let $K$ be a finite field extension of $\Q _p$ with the residue field 
$k\simeq \F _{p^{N_0}}$ and the ramification index $e_K$. Let $\pi _0$ be 
a uniformising element in $K$. Denote by $\bar K$ an algebraic closure of $K$, set 
$\Gamma _K=\Gal (\bar K/K)$ and denote by $I_K$ the inertia 
subgroup of $\Gamma _K$. We assume that $K$ contains a primitive 
$p$-th root of unity $\zeta _1$.

\subsection{An exact sequence for $\Gamma _{<p}$} \label{S6.1} 
For $n\in\N $, choose $\pi _n\in\bar K$ such that $\pi _n^p=\pi _{n-1}$. 
Let $\wt{K}=\bigcup _{n\in\N }K(\pi _n)$,  
$\Gamma _{<p}:=\Gamma _K/\Gamma _K^pC_p(\Gamma _K)$ and $\Gamma _{\wt{K}}=
\Gal (\bar K/\wt{K})$. Then a natural embedding $\Gamma _{\wt{K}}\subset\Gamma _K$ 
induces a continuous group homomorphism 
$\iota :\Gamma _{\wt{K}}\To \Gamma _{<p}$.  

We have $\Gal (K(\pi _1)/K)=\langle \tau _0\rangle ^{\Z /p}$, where $\tau _0(\pi _1)
=\pi _1\zeta _1$. Let $j:\Gamma _{<p}\To \Gal (K(\pi _1)/K)$ be a natural epimorphism. 

\begin{Prop} \label{P6.1}
 The following sequence 
$$\Gamma _{\wt{K}}\overset{\iota }\To \Gamma _{<p}\overset{j}
\To \langle \tau _0\rangle ^{\Z /p}\To 1$$ 
is exact. 
\end{Prop}

\begin{proof}
For $n\geqslant 2$, let $\zeta _n\in\bar K$ be such that $\zeta _n^{p}=\zeta _{n-1}$.  

Consider $\wt{K}'=\bigcup _{n}K(\pi _n,\zeta _n)$. Then 
$\wt {K}'/K$ is Galois with the Galois group 
$\Gamma _{\wt{K}'/K}=\langle \tilde \sigma, \tilde\tau _0\rangle $. 
Here for any $n\in\N $ and some $s_0\in \Z$, 
$\tilde\sigma \zeta _n=\zeta _n^{1+ps_0}$, 
$\tilde\sigma\pi _n=\pi _n$, $\tilde\tau _0\zeta _n=\zeta _n$, 
$\tilde\tau_0\pi _n=\pi _n\zeta _n$ and 
$\tilde\sigma ^{-1}\tilde\tau _0\tilde\sigma =\tilde\tau _0^{(1+ps_0)^{-1}}$. 

Therefore, $C_2(\Gamma _{\wt{K}'/K})\subset 
\langle\tilde\tau _0^p\rangle\subset\Gamma _{\wt{K}'/K}^p
=\langle \tilde\sigma ^p, \tilde\tau _0^p\rangle $, 
$\Gamma ^p_{\wt{K}'/K}C_p(\Gamma _{\wt{K}'/K})=
\langle \tilde\sigma ^p, \tilde\tau _0^p\rangle $ and we have a natural exact sequence 
$$\langle\tilde\sigma\rangle \To \Gamma _{\wt{K}'/K}/
\Gamma ^p_{\wt{K}'/K}C_p(\Gamma _{\wt{K}'/K})\To 
\langle\tilde\tau _0\rangle 
\,\op{mod}\,\langle\tilde\tau _0^p\rangle =\langle\tau _0\rangle ^{\Z /p}\To 1\,.$$
Note that  $\Gamma _{\wt{K}'}$ together with a lift 
$\hat\sigma \in\Gamma _{\wt{K}}$ of $\tilde\sigma $  
generate $\Gamma _{\wt{K}}$. 

The above short exact sequence implies that 
$\op{Ker}
\left (\Gamma _{<p}\To \langle\tau _0\rangle ^{\Z /p} \right )$
is generated by $\hat\sigma $ and the image of $\Gamma _{\wt{K}'}$. 
So, the kernel coincides with the image of $\Gamma _{\wt{K}}$ in $\Gamma _{<p}$.  
\end{proof}

\subsection{The field-of-norms functor} \label{S6.2} 
Let $R$ be Fontaine's ring. We have a natural embedding 
$k\subset R$ and an element $t=(\pi _n\,\op{mod}\,p)_{n\geqslant 0}\in R$. 
If $\c K=k((t))$ and $R_0=\op{Frac}\,R$ then $\c K$ is a closed subfield of $R_0$ and 
the theory of the field-of-norms functor $X$ \cite{Wtb1}, Subsection 4.3, 
identifies $X(\wt{K})$ with $\c K$ 
and $R_0$ with 
the completion of the separable closure $\c K_{sep}$. 
In particular, we have a natural inclusion $\iota _K:\Gamma _K\To\Aut\, R_0$ 
which induces the identification of $\c G=\op{Gal} (\c K_{sep}/\c K)$ and 
$\Gamma _{\wt{K}}\subset \Gamma _K$. This identification is compatible with 
the ramification filtrations on $\Gamma _K$ and $\c G$. The simplest version of this 
compatibility states that 
if $v\geqslant 0$ and $v'=\varphi _{\wt{K}/K}(v)$, where $\varphi _{\wt{K}/K}$ is 
the Herbrand function for our infinite APF extension $\wt{K}/K$, 
then 
\begin{equation}\label{E6.1} 
\iota _K(\Gamma _{\wt{K}}\cap\Gamma _K^{(v')})=\c G^{(v)}\, .
\end{equation}

As a matter of fact, there is a more general property
\begin{equation} \label{E6.2} 
\iota  _K(\Gamma _K)\cap\c I_{/\c K}^{(v)}=\iota  _K
\left (\Gamma _K^{(v')}\right )\, .
\end{equation} 
This result is formulated in \cite{Wtb1}, Subsection 3.3, 
in the case when our infinite APF extension 
is Galois but the proof works word-by-word without this assumption.

We use the results of the above sections and use the appropriate 
notation related to our field $\c K$, e.g. $\c G_{<p}=\op{Gal}(\c K_{<p}/\c K)$, where 
$\c K_{<p}$ is the subfield of $\c K_{sep}$ fixed by $\c G^pC_p(\c G)$.  
The identification $\iota _K|_{\Gamma _{\wt{K}}}$ composed with the morphism $\iota $ from 
Proposition \ref{P6.1} induces a natural continuous morphism of groups 
$\iota _{<p}:\c G_{<p}\To\Gamma _{<p}$. Now Proposition \ref{P6.1} implies 
the following property. 

\begin{Prop} \label{P6.2} 
The sequence 
$$\c G_{<p}\overset{\iota _{<p}}\To \Gamma _{<p}
\overset{j}\To \langle\tau _0\rangle ^{\Z /p}\To 1\,$$ 
is exact. 
\end{Prop}

 Note that $\c G_{<p}$ is infinite but $\Gamma _{<p}$ is finite. 
 The finiteness of $\Gamma _{<p}$ follows easily from local class field theory. Indeed, 
 for $1\leqslant s<p$, let $K[s]$ be the subfield of $\bar K$ fixed by the group 
 $\Gamma _K^pC_{s+1}(\Gamma _K)$. Then all $K[s+1]/K[s]$ are abelian Galois extensions 
 with Galois groups of period $p$. By induction on $s$ and 
local class field theory these groups are quotients of 
 the finite groups $K[s]^*/K[s]^{*p}$ (use that $[K[s]:\Q _p]<\infty $) and,therefore, 
 for $K[p-1]=K_{<p}$, $[K_{<p}:K]<\infty $.

\subsection{Auxiliary statements} \label{S6.3} 
Suppose $v_{\c K}$ is the unique extension of a normalized valuation of $\c K$ to $R_0$.     
Let $\eta $ be a closed embedding of $\c K$ into $R_0$ which is compatible 
with $v_{\c K}$, i.e. for any $a\in \c K$, $v_{\c K}(a)=v_{\c K}(\eta (a))$. 

Let $c_0:=e^*(=e_Kp/(p-1))$. 
As earlier, consider $\c M\subset\c L_{\c K}$, 
$\c M_{<p}\subset\c L_{\c K_{<p}}$ and 
 $\c M_{R_0}\subset\c L_{R_0}$. 
We know that $e\in \c M$, $f\in \c M_{<p}$ (these elements  
were chosen in Subsection \ref{S1.3}) and for 
similar reasons, if $\hat\eta \in\Aut R_0$ is a lift of $\eta $ then 
$(\id _{\c L}\otimes\hat\eta )f\in\c M_{R_0}$. 

Below we consider the condition 
$(\id _{\c L}\otimes \eta )e\equiv e\,\op{mod}\,t^{(p-1)c_0}\c M_{R_0}$. 
In particular, this congruence holds modulo $\c L_{\m _R}+\c L(p)_{R_0}$ and 
following the coefficient for $D_{10}$  we obtain that $\eta\in\c I_{\c K,v}$, 
where $v=v(\eta )>0$. 

\begin{Prop} \label{P6.3} 
Suppose $(\id _{\c L}\otimes \eta )e\equiv e\,\op{mod}\,t^{(p-1)c_0}\c M_{R_0}$. 
Then 

{\rm a)}\ there is $m\in t^{(p-1)c_0}\c M_{R_0}$ such that 
$$(\id _{\c L}\otimes\eta )e\equiv (-\sigma m)\circ e\circ m\,\op{mod}\,\c L(p)_{R_0}\,;$$

 {\rm b)}\ if $\hat\eta$ is a lift of $\eta $ to $R_0$ then there is 
a unique $l\in G(\c L)\,\op{mod}\,G(\c L(p))$ such that 
$$(\id _{\c L}\otimes\hat\eta )f\equiv f\circ l\,\op{mod}\,t^{(p-1)c_0}
\c M_{R_0}\, .$$

{\rm c)}\ there is a unique lift $\eta (p)$ of $\eta $ to $\c K(p)$ 
such that 
 $(\id _{\bar{\c L}}\otimes\eta (p))\bar f=\bar f$, 
 where $\bar f=f\,\op{mod}\, t^{(p-1)c_0}\c M_{R_0}$. 
\end{Prop}

\begin{proof} a) 
 Note that $t^{(p-1)c_0}\c M_{R_0}$ is an ideal in $\c M_{R_0}$ 
and for any $i\in\N$ and $m^0\in t^{(p-1)c_0}C_i(\c M_{R_0})$, 
there is  $m_i\in t^{(p-1)c_0}C_{i}(\c M_{R_0})$ 
such that $\sigma m_i-m_i\equiv m^0\,\op{mod}\,\c L(p)_{R_0}$. 
(Use that $\sigma $ is topologically nilpotent on $t^{(p-1)c_0}\c M_{R_0}/\c L(p)_{R_0}$.) 

Therefore,  there is $m_1\in t^{(p-1)c_0}\c M_{R_0}$ such that 
$\eta (e)\equiv e-\sigma m_1+m_1\,\op{mod}\,\c L(p)_{R_0}$. This implies that 
$$\sigma (m_1)\circ \eta (e)\equiv e\circ m_1\,
\op{mod}\,t^{(p-1)c_0}C_2(\c M_{R_0})+\c L(p)_{R_0}\, .$$ 
Similarly, there is $m_2\in 
t^{(p-1)c_0}C_2(\c M_{R_0})$ such that 
$$\sigma (m_1)\circ \eta (e)\equiv -\sigma m_2+m_2 +e\circ m_1\, \op{mod}\c L(p)_{R_0},$$  
$$\sigma (m_2\circ m_1)\circ \eta (e)\equiv e\circ (m_2\circ m_1) 
\op{mod}\,t^{(p-1)c_0}C_3(\c M_{R_0})+\c L(p)_{R_0}\, ,$$ 
and so on.  
This gives 
$m_i\in t^{(p-1)c_0}C_i(\c M_{R_0})$, $1\leqslant i<p$, such that 
$$\sigma (m_{p-1}\circ \dots \circ m_1)\circ\eta (e)\equiv 
e\circ (m_{p-1}\circ \dots \circ m_1)\, \op{mod}\,\c L(p)_{R_0}\, .$$  
This proves a) with $m=m_{p-1}\circ\dots\circ m_1$. 

b) Let $(\id _{\c L}\otimes \hat\eta )f=f'$. Then for the above 
element $m$, we have 
$\sigma (m\circ f')\equiv e\circ (m\circ f')\,\op{mod}\,\c L(p)_{R_0}$ and, therefore,  
$$\sigma ((-f)\circ m\circ f')\equiv (-f)\circ m\circ f'\,\op{mod}\c L(p)_{R_0}\, .$$ 
This implies the existence of 
$l\in\c L$ such that $m\circ f'\equiv f\circ l\,\op{mod}\,\c L(p)_{R_0}$ 
(use that $\bar{\c L}_{R_0}|_{\sigma =\id }=\bar{\c L}$).

Suppose $l'\in\c L$ also satisfies statement b) of our lemma. Then we have  
$f\circ l\equiv f\circ l'\,\op{mod}\,
t^{(p-1)c_0}\c M_{R_0}$, $l\equiv l'\,\op{mod}\,t^{(p-1)c_0}\c M_{R_0}$ and 
$$l\circ (-l')\in \left (t^{(p-1)c_0}\c M_{R_0}\right )|_{\sigma =\id }
\subset \left (\c L_{\m _R}+\c L(p)_{R_0}\right )|_{\sigma =\id }=\c L(p)\,.$$

c) This follows from part b) because $\Gal (\c K_{<p}/\c K(p)=\c L(p)$. 

Proposition is proved.  
\end{proof} 
\medskip

\subsection{Isomorphism $\kappa _{<p}$} \label{S6.4}

Let $\varepsilon =(\zeta _n\,\op{mod}\,p)_{n\geqslant 0}\in R$ be Fontaine's element 
(here $\zeta _0=1$ and for $n\geqslant 1$, $\zeta _n$ 
were defined in Subsection \ref{S6.1}). 

Let $\zeta _1=1+\sum _{i\geqslant 1}[\beta _i]\pi _0^i$ 
where $[\beta _i]$ are Teichmuller representatives of $\beta _i\in k$. 
Use the identification of rings  
$R/t^{pe_K}\simeq O_{\bar K}/p$, coming from the natural projection $R\To (O_{\bar K}/p)_1$. 
This implies (note $pe_K=(p-1)c_0$, where $c_0=e^*$, cf. Subsection \ref{S6.3}) 
$$\varepsilon \equiv 1+\sum _{i\geqslant 0}\alpha _it^{c_0+pi}\,\op{mod}\,t^{(p-1)c_0}R$$
where all $\alpha _i=\beta _i^p\in k$, $\alpha _0\ne 0$ (note $\varepsilon\notin\c K$).

Assume that $h\in\op{Aut}\c K$ from Subsection \ref{S2.1} is such that 
for all $i$, $\alpha _i(h)=\alpha _i$ (and $h|_k=\id _k$). Then $v(h)=c_0$, 
cf. Subsection \ref{S4.1}, and   
$$h(t)\equiv t\varepsilon \,\op{mod}\,t^{(p-1)c_0+1}R\, .$$

This implies that for any $\tau\in\Gamma _K$, there is 
$\tilde h\in\langle h\rangle \subset\Aut\,\c K$ 
 such that 
$\iota _K(\tau )|_{\c K}\,(t)\equiv \tilde h(t)\,\op{mod}\, t^{(p-1)c_0+1}R$, 
where $\iota _K:\Gamma _K\To \Aut R_0$ was defined in Subsection \ref{S6.2}. 
Indeed, 
there is $m\in\Z _p$ such that 
$$\iota _K(\tau )(t)=t\varepsilon ^m\equiv 
t\left (1+\sum _{i\geqslant 0}\alpha _it^{c_0+pi}\right )^m
\equiv h^m(t)\,\op{mod}\,t^{(p-1)c_0+1}R$$
(use that $h(t^p)\equiv t^p\,\op{mod}\,t^{pc_0}R$), and we can take $\tilde h=h^m$. 
Clearly, such $\tilde h$ is unique modulo the subgroup $\langle h^p\rangle $.

 This means that 
$\eta :=\iota _K(\tau )|_{\c K}\tilde h^{-1}:\c K\To R_0$ satisfies the assumption 
from Proposition \ref{P6.3}. Let $\eta (p)$ be the lift from the 
part c) of that proposition, $\hat\eta\in\Aut R_0$ be such that $\hat\eta |_{\c K(p)}=\eta (p)$  
and 
$\tilde h(p):=(\hat\eta ^{-1}\iota _K(\tau ))|_{\c K(p)}$. Then 
$\tilde h(p)|_{\c K}=\tilde h$ and by Galois theory $\tilde h (p)\in\Aut \c K(p)$. 
As a result, $\tilde h(p)\in\wt{\c G}_h/C_p(\wt{\c G}_h)$ is a unique 
lift of $\tilde h$ such that   
$$(\id _{\bar{\c L}}\otimes \iota _K(\tau ))\bar f = 
(\id _{\bar{\c L}}\otimes \tilde h(p))\bar f\, .$$

If $\tilde h$ is replaced by an element of $\langle h^p\rangle $ then 
$\tilde h(p)$ is replaced by an element from $(\wt{\c G}_h/C_p(\wt{\c G}_h)^p$ but 
this will not affect $(\id _{\bar{\c L}}\otimes \tilde h(p))\bar f$. 
Therefore, the image of $\tilde h(p)$ in $\c G_h$ is well-defined.

As a result,  we obtained the map 
of sets $\kappa :\Gamma _K\To \c G_h$ uniquely characterized by 
the following equality  in 
$\bar {\c M}_{R_0}=\c M_{R_0}\,\op{mod}\,t^{c_0(p-1)}\c M_{R_0}$ 
$$(\id _{\bar{\c L}}\otimes \iota _K(\tau ))\bar f = 
(\id _{\bar{\c L}}\otimes \hat\kappa (\tau ))\bar f\, ,$$
where $\hat\kappa (\tau )\in\wt{\c G}_h/C_p(\wt{\c G}_h)\subset \Aut\c K(p)$ is any lift 
of $\kappa (\tau )\in\c G_h$ with respect to the natural projection 
$\wt{\c G}_h/C_p(\wt{\c G}_h)\To \c G_h$.

\begin{Prop} \label{P6.4} 
$\kappa $ induces a group isomorphism $\kappa _{<p}:\Gamma _{<p}\To\c G_h$. 
 \end{Prop}

\begin{proof} Suppose $\tau_1,\tau \in\Gamma _K$. 
Let $\bar c\in\bar{\c L}_{\c K}$ and $\bar A\in\Aut _{\bar{\c L}}$ be such that 
 $(\id _{\bar{\c L}}\otimes \hat\kappa (\tau ))\bar f=
\bar c\circ (\bar A\otimes \id _{\c K(p)})\bar f$. Then 
$$(\id _{\bar{\c L}}\otimes\hat\kappa (\tau_1\tau ))\bar f= 
(\id _{\bar{\c L}}\otimes \tau _1\tau )\bar f
= (\id _{\bar{\c L}}\otimes \tau _1)(\id _{\bar{\c L}}\otimes \tau )\bar f= $$
$$(\id _{\bar{\c L}}\otimes \tau _1)(\id _{\bar{\c L}}\otimes\hat\kappa (\tau ))\bar f= 
(\id _{\bar{\c L}}\otimes \tau _1)(\bar c\circ (\bar A\otimes \id _{\c K(p)})\bar f)= $$
$$(\id _{\bar{\c L}}\otimes \tau _1)\bar c\circ (\bar A\otimes \tau _1)\bar f
= (\id _{\bar{\c L}}\otimes \hat\kappa (\tau _1))
\bar c\circ (\bar A\otimes\id _{\c K(p)})(\id _{\bar{\c L}}\otimes \tau _1)\bar f= $$
$$(\id _{\bar{\c L}}\otimes \hat\kappa (\tau _1))
\bar c\circ (\bar A\otimes\id _{\c K(p)})(\id _{\bar{\c L}}
\otimes \hat\kappa (\tau _1))\bar f= (\id _{\bar{\c L}}\otimes \hat\kappa (\tau _1))
(\bar c\circ (\bar A\otimes\id _{\c K(p)})\bar f=$$
$$(\id _{\bar{\c L}}\otimes \hat\kappa (\tau _1))(\id _{\bar{\c L}}\otimes\hat\kappa (\tau ))\bar f
= (\id _{\bar{\c L}}\otimes \hat\kappa (\tau _1)\hat\kappa (\tau ))\bar f$$
and, therefore, $\kappa (\tau _1\tau )=\kappa (\tau _1)\kappa (\tau )$ 
(use that $\c G_h$ acts strictly on the orbit of $\bar f$).  

In particular, 
$\kappa $ factors through the natural projection $\Gamma _K\to\Gamma _{<p}$  
and defines the group homomorphism $\kappa _{<p}:\Gamma _{<p}\to\c G_h$.

Recall that we have the field-of-norms identification of 
$\Gamma _{\wt{K}}$ with $\c G$ and, therefore, $\kappa _{<p}$ identifies   
the groups $\kappa (\Gamma _{\wt{K}})$ and $G(\bar{\c L})\subset\c G_h$. 
Besides, $\kappa _{<p}$ induces a group isomorphism of 
$\langle \tau _0\rangle ^{\Z /p}$ and $\langle h\rangle ^{\Z /p}$. Now Proposition \ref{P6.2} 
implies that $\kappa _{<p}$ is a group isomorphism. 
\end{proof}

\subsection{Ramification filtrations} \label{S6.5}
Recall that $\Gamma _{<p}=G(L)$ has the induced fitration by the images  
$\Gamma ^{(v)}_{<p}$, $v\geqslant 0$, of the ramification subgroups 
$\Gamma _K^{(v)}$ with respect to the projection $\op{pr}_{<p}:\Gamma _K\To\Gamma _{<p}$. 
This gives the appropriate filtration by the 
ideals $L^{(v)}$ of the Lie algebra $L$.

As earlier in Subsection \ref{S6.2}, the elements of 
$i_K(\Gamma _K)\subset\Aut R_0$ 
can be considered as the elements of the ramification subsets 
$\c I_{/\c K}^{(v)}$, $v\geqslant 0$. This gives the induced filtration 
$L_{/\c K}^{(v)}$ on $L$ (the notation indicates 
to the ``upper numbering with respect to $\c K$'') such that $G(L^{(v)}_{/\c K})$ is the image of 
$\iota _K^{-1}(\iota _K(\Gamma _K)\cap\c I^{(v)}_{/\c K})$ 
under the projection $\op{pr}_{<p}$.   
By property \eqref{E6.2} we have 
$ L_{/\c K}^{(v)}=L^{(\varphi _{\wt{K}/K}(v))}$.

The elements of $\c G_h=G(L_h)$ are related to the field automorphisms $\Aut \c K(p)$, i.e.  
we have a natural embedding 
$\wt{\c G}_h/C_p(\wt{\c G}_h)\subset\Aut \c K(p)$ and 
then use the  
projection $\wt{\c G}_h/C_p(\wt{\c G}_h)\To\c G _h$, cf. Section \ref{S3}. 

Therefore, we can define for any $v\geqslant 0$, 
the ideal  $L^{(v)}_h$ in $L_h$ 
as the image of  
$\wt{\c G}_h/C_p(\wt{\c G}_h)\cap (\op{res}_{\c K(p)}\c I^{(v)}_{/\c K})$ in 
$\c G_h$. Here 
for any $\iota\in\c I$, $\op{res}_{\c K(p)}\iota =\iota |_{\c K(p)}$, i.e. 
$\op{res}_{\c K(p)}\c I^{(v)}_{/\c K}=\c I_{\c K(p), v'}$, 
where  $\varphi _{\c K(p)/\c K}(v')=v$. 

\begin{Prop} \label{P6.5} 
 For any $v\geqslant 0$, $\kappa _{<p}(L_{/\c K}^{(v)})=L_h^{(v)}$. 
\end{Prop}

\begin{proof} We need the following lemma. 

\begin{Lem}\label{L6.6}
 Let $\eta (p)\in\c I_{\c K(p)}$ 
be the morphism from 
Proposition \ref{P6.3}c). Then 
 $\eta (p)$ is a unique arithmetical
 lift of $\eta =\eta (p)|_{\c K}$.  
\end{Lem}

This lemma will be proved in Subsection \ref{S6.6} below.

Continue with the proof of our proposition.  

Suppose $\tau\in\Gamma _K$ and for some $v\geqslant 0$, 
$\iota _K(\tau)\in\c I_{/\c K}^{(v)}$ (in particular, $\tau\in I_K\subset\Gamma _K$), i.e. 
$\op{pr}_{<p}(\tau)\in L^{(v)}_{/\c K}$. 
Consider $g=\kappa (\tau )=\kappa _{<p}(\op{pr}_{<p}(\tau ))\in\c G_h$. 
If $v'\geqslant 0$ and 
$g\in L_h^{(v')}$ then there is a lift 
$g(p)\in\wt{\c G}_h/C_p(\wt{\c G}_h)\subset \Aut\c K(p)$ of $g$ 
such that $g(p)\in \op{res}_{\c K(p)}\c I^{(v')}_{/\c K}$. 

Let $\eta (p):=\iota _K(\tau )|_{\c K(p)}g(p)^{-1}\in\c I_{\c K(p)}$ 
and $\eta :=\eta (p)|_{\c K}\in\c I_{\c K}$. Using the formulas 
from the beginning of Subsection \ref{S6.4} we obtain that 
\begin{equation} \label{E6.3} (\id _{\c L}\otimes \eta )e\equiv e\,
\op{mod}\,t^{(p-1)c_0}\c M_{R_0}\, .
\end{equation} 
Then the definition of $\kappa $ implies  that 
$(\id _{\bar{\c L}}\otimes \eta (p))\bar f=\bar f$,  and 
by Lemma \ref{L6.6}, $\eta (p)$ is arithmetical lift of $\eta $.

We can easily see that \eqref{E6.3} implies that 
$$\eta (t^{-(p-1)c_0+1})\equiv t^{-(p-1)c_0+1}\,\op{mod}\,\m _R\, $$ 
and, therefore, there is $v^o>(p-1)c_0-1$ such that 
$\eta\in\c I_{\c K, v^o}$. Therefore, 
$\eta (p)\in \op{res}_{\c K(p)}\c I^{(v^o)}_{/\c K}$, 
or equivalently, 
$$\iota _K(\tau )|_{\c K(p)}
\equiv g(p)\,\op{mod}\,\op{res}_{\c K(p)}\c I^{(v^o)}_{/\c K}\, .$$

So, for all  
$0\leqslant v\leqslant (p-1)c_0-1$ and $\tau\in\Gamma _K$,  
$$\op{pr}_{<p}(\tau )\in L_{/\c K}^{(v)}\ \ \Leftrightarrow \ \ 
\kappa _{<p}(\op{pr}_{<p}\tau )
\in L_h^{(v)}\, .$$ 

It remains to prove that if $v^o>(p-1)c_0-1$ then   
$L_{/\c K}^{(v^o)}=L_h^{(v^o)}=0$. 

Suppose $\tau\in\Gamma _K$ is such that 
$\op{pr}_{<p}(\tau )\in L_{/\c K}^{(v^o)}$. 
We can assume that $\iota _K(\tau )\in\c I_{/\c K}^{(v^o)}$. 
Clearly, there is $m\in\Z _p$ such that $\iota _K(\tau )(t)=t\varepsilon ^m$. 
Then $m\equiv 0\,\op{mod}\,p$ because 
$\iota _K(\tau )|_{\c K}\in\c I_{\c K,v^o}$ and $v^o>c_0$. 

Let $\hat\tau _0\in\Gamma _K$ be a lift of $\tilde\tau _0$ from  
the proof of Proposition \ref{P6.1}. Note that  
$\iota _K(\hat\tau _0)|_{\c K}\in\c I_{\c K,c_0}$ and 
$\iota _K(\hat\tau _0)(t)=t\varepsilon $.  
This implies that 
$\hat\tau _0 ^{-m}\tau \in\Gamma _{\wt{K}}$ and 
$\iota _K(\hat\tau _0^{-m}\tau )\in\c G=\Gal (\c K_{sep}/\c K)$.  

Note that $\iota _K(\hat\tau _0 )(t)\equiv h(t)\,\op{mod}\, t^{(p-1)c_0}\m _R$. 
Therefore, we can apply the arguments from Subsection \ref{S2.5}  
(cf. application of Lemma \ref{L2.9} in the proof of 
Proposition \ref{P2.7}) to prove 
that $(\id _{\bar{\c L}}\otimes \iota _K(\hat\tau _0 ^p))\bar f=\bar f$. 
By Lemma \ref{L6.6}, 
$\iota _K(\hat\tau _0 ^p)|_{\c K(p)}$ is arithmetical 
over $\c K$. 
Hence  $\iota _K(\hat{\tau _0 }^p)|_{\c K}\in\c I_{\c K,(p-1)c_0}$ implies that 
$\iota _K(\hat\tau _0 ^p)|_{\c K(p)}\in\op{res}_{\c K(p)}
\c I_{/\c K}^{((p-1)c_0)}$ 
and 
$$\iota _K(\hat\tau _0 ^{-m}\tau )|_{\c K(p)}\in 
\op{res}_{\c K(p)}\c I_{/\c K}^{(v')}\cap\Gal (\c K(p)/\c K)=
\Gal (\c K(p)/\c K)^{(v')}\, ,$$  
where $v'=\min\{(p-1)c_0,v^o\}>(p-1)c_0-1$. By 
Proposition \ref{P2.11} this ramification subgroup is trivial 
and $\iota _K(\hat\tau _0 ^{-m}\tau )|_{\c K(p)}=e$. 

It remains to note that 
$\kappa _{<p}(\op{pr}_{<p}\tau )=\kappa (\tau )=\kappa (\hat\tau _0 ^{-m}\tau )$ 
appears as the image of $\iota _K(\hat\tau _0 ^{-m}\tau )|_{\c K(p)}$ 
under the natural projection of $\wt{\c G}_h/C_p(\wt{\c G}_h)$ to $\c G_h$. 
Therefore, $\kappa _{<p}(\op{pr}_{<p}\tau )=0$ and $\op{pr}_{<p}\tau =0$. 

For similar reasons,  $L_h^{(v^o)}=0$ if $v^o>(p-1)c_0-1$. 
\end{proof} 

\subsection {Proof of Lemma \ref{L6.6}}\label{S6.6} 
The proof is based on the same idea as the proof of 
Theorem \ref{T4.8} but is considerably easier: we do not need the difficult 
technical result from \cite{Ab3}. This happens because 
we are still studying the lifts from $\c K$ to $\c K(p)$ but these lifts 
come from $\c I_{\c K}^{(v^o)}$, where $v^o>(p-1)c_0-1$, cf. below. 
(In Theorem \ref{T4.8} 
we worked with the case $v^o=c_0$.) 

First of all, the condition 
\begin{equation}\ \label{E6.4} 
(\id _{\c L}\otimes\eta )e\equiv e\,\op{mod}\,t^{(p-1)c_0}\c M_{R_0}
\end{equation}
implies $\eta |_k=\id _k$ and 
$\eta (t^{-(p-1)c_0+1})\equiv t^{-(p-1)c_0+1}\,\op{mod}\,\m _R$ 
(just follow the coefficient for $D_{(p-1)c_0-1,0}$). As a result, we obtain 
$\eta (t)\equiv t\,\op{mod}\,t^{(p-1)c_0}\m _R$, i.e. there is $v^o>(p-1)c_0-1$ such that 
$\eta\in\c I_{\c K,v^0}$. 
Going in the opposite direction we 
can easily see that this condition is also sufficient for \eqref{E6.4}. 

Prove that $\c L^{(v^o)}\subset\c L(p)$. 

It will be sufficient to verify that all generators $\c F^0_{\gamma ,-N}$ 
of $\c L^{(v^o)}_k$ (where $\gamma\geqslant v^o$), cf. Subsection \ref{S1.4}, 
belong to $\c L(p)_k$. All such $\c F^0_{\gamma ,-N}$ are linear combinations 
of commutators of the form $[\dots [D_{a_1n_1},],\dots ,D_{a_mn_m}]$, where $m<p$, 
all $a_i\in\Z ^0(p)$, all $n_i\leqslant 0$ and 
$a_1p^{n_1}+\dots +a_mp^{n_m}\geqslant v^o$. If $\op{wt}(D_{a_in_i})=s_i$, 
cf. Subsection \ref{S2.4}, then  
$(s_i-1)c_0\leqslant a_i<s_ic_0$ and 
$$(p-1)c_0-1<v^o\leqslant a_1+\dots +a_m<(s_1+\dots +s_m)c_0\, .$$
This implies that $s_1+\dots +s_m\geqslant p$ (use that 
$a_1+\dots +a_m\in\Z $). So, all 
our commutators have weight $\geqslant p$ and, therefore, belong to $\c L(p)_k$. 

Now Corollary \ref{C4.4} implies that there is only one arithmetical lift 
of $\eta $ to $\c K(p)$. Therefore, it will be sufficient to prove that 
\medskip 

$\bullet $\ {\it if 
$\eta (p)$ is arithmetical lift of $\eta $ 
then $(\id _{\bar{\c L}}\otimes \eta (p))\bar f=\bar f$.}
\medskip

As earlier in Subsection \ref{S4.4}, let $e_{(p)}$ and 
$\varphi _{(p)}$ be the ramification index and, resp., the Herbrand function for 
$\c K(p)/\c K$. 

Suppose 
\begin{equation} \label{E6.5}
v^o\geqslant\varphi _{(p)}(e_{(p)}(p-1)c_0)\, . 
\end{equation}

Then $\eta (p)\in \c I_{\c K(p),v^o_{(p)}}$, where $v^o_{(p)}\geqslant 
e_{(p)}(p-1)c_0$ and, therefore, 
$(\id _{\bar{\c L}}\otimes\eta (p))\bar f=\bar f$ (use that for any 
$a\in\c K(p)$, $\eta (p)a-a\in at^{(p-1)c_0}R)$). 
This proves our lemma under assumption \eqref{E6.5}. 

Otherwise, we can apply the trick from Subsection \ref{S4} as follows. 

We use the notation from the beginning of Subsection \ref{S4.4}. 

Take $\c K'=\c K(r^o,N^o)$, where the parameters 
$r^o\in\Q $ and $N^o\equiv 0\,\op{mod}\, N_0$ 
satisfy the following requirements 
(this can be done by enlarging (if necessary) $N^o$ with fixed $r^o$, 
cf. Subsection \ref{S4.4}): 
\medskip 

$\bullet _1)$\ $r^o(q^o-1)\in\Z ^+(p)$ where $q^o=p^{N^o}$ and $(p-1)c_0-1<r^o<v^o$;
\medskip 

$\bullet _2)$\ $r^o(1-1/q)>(p-1)c_0-1$;
\medskip 

$\bullet _3)$\ $r^o+q^o(v^o-r^o)\geqslant \varphi _{(p)}(e_{(p)}(p-1)c_0)$. 
\medskip 

Use the uniformiser $t'$ to define an analog  
$e^{\prime }=\sum _{a\in\Z ^0(p)}t^{\prime -a}D_{a0}\in\c L_{\c K'}$ 
of $e$ for $\c K'$ and set $e^{\prime (q^o)}=\sigma ^{N^o}e'=
\sum _{a\in\Z ^0(p)}t^{\prime -aq^o}D_{a0}\in\c L_{\c K'}$. 

Verify that $\bullet _2)$ implies that 
$e\equiv e^{\prime (q^o)}\,\op{mod}\,t^{(p-1)c_0}\c M_{R_0}$. 
Indeed: 
\medskip 

1) Suppose $a\geqslant (p-1)c_0$. Then  
$t^{-a}D_{a0}, t^{\prime -aq^o}D_{a0}\in\c L(p)_{R_0}$. 
\medskip 

2) Suppose $1\leqslant s<p-1$ and  
$(s-1)c_0\leqslant a\leqslant sc_0-1$, i.e.  
$D_{a0}\in\c L(s)_k$. From the definition of $\c K'$ we have  
$t-t^{\prime q^o}\in t^{\prime q^o+r^o(q^o-1)}R$. This implies  
(use $\bullet _2$) that 
$t\equiv t^{\prime q^o}\,\op{mod}\,t^{(p-1)c_0}\m _R$ 
and, therefore, 
$$(t^{-a}-t^{\prime -aq^o})D_{a0}\in t^{-a+(p-1)c_0-1}\m _RD_{a0}
\subset t^{(p-1-s)c_0}\c L(s)_{\m _R}\subset t^{(p-1)c_0}\c M_{R_0}$$
\medskip 

Now we can proceed as in the proof of Proposition \ref{P6.3}a) 
to obtain the existence of $m\in t^{(p-1)c_0}\c M_{R_0}$ such that 
$$e\equiv (\sigma m)\circ e^{\prime (q)}
\circ (-m)\,\op{mod}\,\c L(p)_{R_0}\, ,$$
and the existence of 
$f'\in\c L_{sep}$ such that $\sigma f'=e'\circ f'$ and  
\begin{equation} \label{E6.6} f\equiv m\circ \sigma ^{N^o}(f')\,\op{mod}\,\c L(p)_{R_0}\, .
\end{equation}

Consider the fields tower 
$\c K\subset\c K'\subset \c K'\c K(p)\subset \c K'(p)\subset \c K'_{<p}$, 
where $\c K'(p)$ and $\c K'_{<p}$ are analogs of $\c K(p)$ and, resp, 
$\c K_{<p}$ for $\c K'$. 
Let $\hat\eta '$ be an arithmetical lift of $\eta $ to $\c K'_{<p}$. 
Then $\eta (p):=\hat\eta '|_{\c K(p)}$, $\eta '(p):=\hat\eta '|_{\c K'(p)}$ 
 and $\eta ':=\hat\eta '|_{\c K'}$ 
are arithmetical over $\c K$. 

So,  $\eta '\in\c I_{\c K',v'^o}$, where 
$v'^o=r^o+q^o(v^o-r^o)\geqslant \varphi _{(p)}(e_{(p)}(p-1)c_0)$. Therefore, 
we can apply assumption 
\eqref{E6.5} and   
(use that $\eta '(p)$ is arithmetical over $\eta '$) 
deduce  the following congruence  
$$(\id _{\bar{\c L}}\otimes \eta '(p))f'\equiv f'\,\op{mod}t^{\prime (p-1)c_0}\c M'_{R_0}$$ 
(here $\c M'_{R_0}$ is an analogue of $\c M_{R_0}$ for $\c K'$). This 
implies that 
$$(\id _{\bar{\c L}}\otimes \eta '(p))\sigma ^{N^o}(f^{\prime })
\equiv \sigma ^{N^o}(f^{\prime })\,\op{mod}t^{(p-1)c_0}\c M_{R_0}$$
(use that $\sigma ^{N^o}\c M'_{R_0}\subset \c M_{R_0}$). 
It remains to note that  \eqref{E6.6} implies now that 
$(\id _{\bar{\c L}}\otimes \eta (p))\bar f=\bar f$. 
The lemma is proved.

\subsection{Properties of $\Gamma _{<p}=G(L)$} \label{S6.7} 

Propositions \ref{P6.4} and \ref{P6.5} allow us to extend all results  
obtained for the group $\c G_h=G(L_h)$ in the characteristic $p$ case 
to the Galois group 
$\Gamma _{<p}$ together with its ramification filtration 
$\{\Gamma _{<p}^{(v)}\}_{v\geqslant 0}$ in the mixed characteristic case. 

We stated these results independently in the Introduction,  
cf. Theorems \ref{T0.1}-\ref{T0.6},  
and summarize them here briefly as follows.

$\bullet $\ {\it Group structure:}
\medskip 

---  $\Gamma _{<p}=G(L)$, where $L$ is the Lie $\F _p$-algebra such that 
$$0\To\c L/\c L(p)\To L\To\F _p\tau _0\To 0\, .$$

--- the Lie algebra $\c L$ was defined in Subsection \ref{S1.3}; 
\medskip \

--- $\c L_k$ has standard system of generators 
$$\{D_{an}\ |\ 
a\in\Z ^+(p), n\in\Z /N_0\}\cup\{D_0\}$$

--- the ideals $\c L(s)$, $2\leqslant s\leqslant p$, are 
given by Theorem \ref{T2.5} and the ideal $C_s(L)$ 
of commutators of order $\geqslant s$ in $L$ equals $\c L(s)/\c L(p)$;
\medskip 

--- the structure of $L$ is determined by a lift 
$\tau _{<p}$ of $\tau _0$ and the appropriate differentiation 
$\op{ad}\tau _{<p}$ is described via recurrent relation \eqref{E3.4}, 
cf. also more explicit information from Section \ref{S5}. 
\medskip 

$\bullet $\ {\it The ramification filtration:}
\medskip 

--- if $K[s]:=K_{<p}^{C_{s+1}(L)}$ 
then the maximal upper ramification number for 
$K[s]/K$ is $e^*=e_Kp/(p-1)$ if $s=1$ and 
$$\varphi _{\wt{K}/K}(e^*s-1)=
e^*+(e^*s-1-e^*)/p=e_K(1+s/(p-1))-1/p$$ 
if $2\leqslant s<p$ 
(use the estimate from Subsection \ref{S2.5} 
and the Herbrand function $\varphi _{K(\pi _1)/K}$); 
\medskip 

--- $\tau _{<p}$ is arithmetical, i.e. $\tau _{<p}\in L^{(e^*)}$,  
iff the appropriate solutions $c_1$ and 
$\{V_a\ |\ a\in\Z ^0(p)\}$ of \eqref{E3.4} satisfy the criterion from Theorem \ref{T4.8};
\medskip 

--- if $v\leqslant e^*$ and $\tau_{<p}$ is arithmetical 
then $\Gamma _{<p}^{(v)}$ is the subgroup of $\Gamma _{<p}$ 
generated by the image of $G(\c L^{(v)})$ and $\tau _{<p}$ (the ideals $\c L^{(v)}$ are described in  
Subsection \ref{S1.4});
\medskip 

--- if $v>e^*$ then $\Gamma _{<p}^{(v)}$ is the image of 
$G(\c L^{(v^*)})$, where $v^*=e^*+p(v-e^*)$ (use the Herbrand function 
for $K(\pi _1)/K$);
\medskip 

--- for explicit information about Demushkin relation for $L$, i.e. 
about the element $\op{ad}\,\tau _{<p}(D_0)$, 
cf. the end of Subsection  \ref{S5.2}.

\end{document}